\documentclass[twoside,centertags                 
]{amsart}

\usepackage[cmtip,color,all,curve,
arc,poly
]{xy}

\usepackage{amssymb}
\usepackage{graphicx}
\usepackage{mathrsfs}

\newtheorem{thm}[equation]{Theorem}
\newtheorem{cor}[equation]{Corollary}
\newtheorem{lem}[equation]{Lemma}
\newtheorem{prop}[equation]{Proposition}
\theoremstyle{definition}
\newtheorem{defn}[equation]{Definition}
\theoremstyle{remark}
\newtheorem{rem}[equation]{Remark}

\newcommand{\C}[1]{\mathscr{#1}}

\newcommand{\operad}[1]{\mathrm{Op}(#1)}
\newcommand{\operadf}[1]{\mathrm{OpFunc}(#1)}
\newcommand{\algebra}[2]{\mathrm{Alg}_{#2}(#1)}
\newcommand{\graphs}[2]{\mathrm{Graph}_{#2}(#1)}
\newcommand{\cats}[2]{\mathrm{Cat}_{#2}(#1)}
\newcommand{\ainfcats}[2]{A_\infty\text{-}\mathrm{Cat}_{#2}(#1)}

\def\r{\rightarrow} 

\def\into{\rightarrowtail}
\def\onto{\twoheadrightarrow}

\newcommand{\id}[1]{\mathrm{id}_{#1}}

\newcommand{\mor}[1]{\operatorname{Mor}(#1)}

\def\hom{\operatorname{Hom}}

\def\dos{\mathbf{2}}

\def\ho{\operatorname{Ho}}

\def\star{{\operatorname{St}}}
\def\link{{\operatorname{Lk}}}

\def\st{\stackrel} 

\def\unit{\mathbb{I}} 

\def\To{\longrightarrow}

\def\colim{\mathop{\operatorname{colim}}}

\newcommand{\val}[2]{\widetilde{#2}}

\newcommand{\card}{\operatorname{card}}

\newcommand{\level}[1]{\operatorname{level}(#1)}

\newcommand{\norm}[1]{\parallel\!#1\!\parallel}

\numberwithin{equation}{section}

\newdir{ >}{{}*!/-7pt/@{>}}
\newdir{> }{{}*!/10pt/@{>}}
\newdir{>> }{{}*!/10pt/@{>>}}

\begin{document}

\title
{Homotopy theory of non-symmetric operads 
}%
\author{Fernando Muro}%
\address{Universidad de Sevilla,
Facultad de Matem\'aticas,
Departamento de \'Algebra,
Avda. Reina Mercedes s/n,
41012 Sevilla, Spain\newline
\indent \textit{Home page}: \textnormal{\texttt{http://personal.us.es/fmuro}}}
\email{fmuro@us.es}

\thanks{The   author was partially supported
by the Spanish Ministry of Education and
Science under the MEC-FEDER grants  MTM2007-63277 and MTM2010-15831, and by the Government of Catalonia under the grant SGR-119-2009. }

\begin{abstract}
We endow categories of non-symmetric operads with natural model structures. We work with no restriction on our operads and only assume the usual hypotheses for model categories with a symmetric monoidal structure. We also study categories of algebras over these operads in enriched non-symmetric monoidal model categories.
\end{abstract}

\maketitle
\tableofcontents


\section{Introduction}

Operads are well-known devices encoding the laws of algebras defined by multilinear operations and relations, e.g. there are operads $\texttt{Ass}$, $\texttt{Com}$ and $\texttt{Lie}$ whose algebras are associative, commutative and Lie algebras, respectively. Morphisms of operads codify relations between different kinds of algebras, e.g. there are morphisms $\texttt{Lie}\r \texttt{Ass}\r \texttt{Com}$ telling us that any commutative algebra is an associative alegbra, and that commutators in an associative algebra yield a  Lie algebra.

There are two kinds of operads: symmetric and non-symmetric operads. Symmetric operads are needed whenever it is necessary to permute variables in order to describe the laws of the corresponding algebras, e.g. 
$\texttt{Com}$ and $\texttt{Lie}$. Non-symmetric operads are specially useful to deal with algebras in non-symmetric monoidal categories, e.g. given a commutative ring $k$ and a set $S$ which is not a singleton, the category of $k$-modules with object set~$S$, which are collections of $k$-modules indexed by $S\times S$, $M=\{M(x,y)\}_{x,y \in S}$, has a non-symmetric tensor product,
$$(M\otimes_{S} N)(x,y)= \bigoplus_{z\in S}M(z,y)\otimes_k N(x,z),$$
whose associative algebras, i.e. algebras over the operad $\texttt{Ass}$, are $k$-linear categories with object set $S$.

Any object $M$ in a symmetric monoidal category $\C{V}$, such as the category of $k$-modules, has an endomorphism symmetric operad $\texttt{End}_{\C V}(M)$  in $\C{V}$ such that, if $\mathcal{O}$ is another symmetric operad in $\C{V}$, the set of $\mathcal{O}$-algebra structures on $M$ is the set of symmetric operad morphisms $\mathcal{O}\r\texttt{End}_{\C V}(M)$. If $M$ belongs to a non-symmetric monoidal category $\C{C}$ enriched over $\C{V}$, such as the category of $k$-modules with object set $S$, then there is a non-symmetric operad $\texttt{End}_{\C C}(M)$  in $\C{V}$ such that the set of algebra structures on $M$ over a non-symmetric operad $\mathcal{O}$ in $\C{V}$ is the set of non-symmetric operad morphisms $\mathcal{O}\r\texttt{End}_{\C C}(M)$ (Definition \ref{eop}).

When the underlying symmetric monoidal category $\C V$ carries homotopical information, e.g. if we replace $k$-modules with differential graded $k$-modules, one is often more interested in a space of $\mathcal{O}$-algebra structures on $M$ rather than a plain set. Such a space can be constructed by using the powerful machinery developed by Dwyer and Kan \cite{slc, csl, fcha} provided we can place the operads $\mathcal{O}$ and $\texttt{End}_{\C C}(M)$ in an appropriate model category of operads.

Model categories of operads were first cosidered by Hinich in the differential graded context \cite{haha,ehaha}, and by Berger and Moerdijk in a more general setting \cite{ahto}. They dealt with symmetric operads and showed that restrictive hypotheses are necessary to endow the category of all operads with an appropriate model category structure, e.g. when $k$ is a $\mathbb{Q}$-algebra or when the symmetric monoidal structure in $\C{V}$ is cartesian closed and there is a symmetric monoidal fibrant replacement functor.

Motivated by our interest in spaces of differential graded category structures, we consider the non-symmetric case, which surprisingly enough does not need any restrictive hypothesys, just usual hypotheses for model categories with a monoidal structure~\cite{ammmc}.

\begin{thm}\label{elt}
Let $\C{V}$ be a cofibrantly generated closed symmetric monoidal model category. Assume that $\C V$ satisfies the monoid axiom. Moreover, suppose that there are  sets of generating cofibrations and  generating trivial cofibrations in $\C V$ with presentable sources. Then the category $\operad{\C{V}}$ of non-symmetric operads in $\C{V}$ is a cofibrantly generated model category such that a morphism  $f\colon \mathcal{O}\r \mathcal{P}$  in $\operad{\C{V}}$ is a weak equivalence (resp. fibration)  if and only if $f(n)\colon \mathcal{O}(n)\r \mathcal{P}(n)$ is a weak equivalence (resp. fibration) in $\C{V}$ for all~$n\geq 0$. Moreover, if $\C V$ is right proper then so is $\operad{\C{V}}$. Furthermore, if $\C V$ is combinatorial then $\operad{\C V}$ is also combinatorial.
\end{thm}

This theorem can be applied to all examples in \cite{ammmc}, see also the references therein: 
\begin{enumerate}
\item Complexes of modules over a commutative ring $k$ with the usual tensor product of complexes.
\item Simplicial $k$-modules with the levelwise tensor product $\otimes_k$.
\item Modules over a finite-dimensional Hopf algebra $R$ over a field $k$ with the tensor product over $k$, e.g. $R=kG$ the group-ring of a finite group $G$.
\item Symmetric spectra with their smash product, and more generally modules over a commutative ring spectrum.
\item $\Gamma$-spaces with Lydakis' smash product.
\item Simplicial functors with their smash product.
\item $S$-modules with their smash product. 
\end{enumerate}
In particular, Theorem \ref{elt} will also be useful to study spaces of spectral category structures. 

\begin{rem}
Recall from \cite[Definition 1.13 (2)]{adamekrosicky} that an object $X$ of $\C V$ is \emph{presentable} if there exists a cardinal $\lambda$ such that the representable functor $\C{V}(X,-)$ commutes with $\lambda$-filtered colimits in $\C V$. Presentable objects are also called \emph{small} or \emph{compact} in some references.  All objects are presentable in many categories of interest, e.g. in all combinatorial model categories. Actually, up to set theoretical principles 
any cofibrantly generated model category is Quillen equivalent to a combinatorial model category \cite{cgmc}. 
\end{rem}

Categories of algebras over symmetric operads do not always have a model structure with fibrations and weak equivalences defined in the underlying category. Sufficient conditions can be found in \cite{ahto}. In the framework of non-symmetric operads they do.  When both algebras and operads live in the same ambient symmetric monoidal model category $\C V$, satisfying the monoid axiom, this has been recently proved by J. E. Harper \cite[Theorem 1.2]{htmommc}. We here extend this result to algebras in a monoidal model category $\C C$ satisfying the monoid axiom  and appropriately enriched over $\C V$. This is necessary, for instance, to construct model categories of enrieched categories, of enriched $A_\infty$-categories, or of any other categorified algebraic structure, see Section \ref{appl}.

\begin{thm}\label{elt2}
Let $\C V$ and $\C C$ be cofibrantly generated biclosed monoidal model categories.  Suppose $\C V$ is symmetric and $\C C$ has a $\C V$-algebra structure given by a strong braided monoidal functor to the center of $\C C$, $z\colon \C{V}\r Z(\C{C})$, such that the composite functor $$\C{V}\st{z}\To Z(\C{C})\To\C{C}$$ is a left Quillen functor. Moreover, assume that $\C V$ and $\C C$ satisfy the monoid axiom (see Definitions \ref{max} and \ref{nsmax}). Furthermore, suppose that $\C C$ has sets of generating cofibrations and generating trivial cofibrations with presentable source. Let $\mathcal{O}$ be a non-symmetric operad in $\C V$. The category $\algebra{\mathcal{O}}{\C C}$ of $\mathcal{O}$-algebras in $\C{C}$ is a cofibrantly generated model category such that an $\mathcal{O}$-algebra morphism $g\colon A\r B$ is a weak equivalence (resp. fibration) if and only if~$g$ is a weak equivalence (resp. fibration) in $\C C$. Moreover, if $\C C$ is right proper then so is $\algebra{\mathcal{O}}{\C C}$. Furthermore, if $\C C$ is combinatorial then $\algebra{\mathcal{O}}{\C C}$ is also combinatorial.
\end{thm}

The notion of monoidal model category in \cite[Definition 3.1]{ammmc} makes sense with no modification in the non-symmetric context, see Definition \ref{defm}.

Any operad morphism $\phi\colon\mathcal{O}\r\mathcal{P}$ induces a change of operad functor
$$\phi^*\colon\algebra{\mathcal{P}}{\C C}\To \algebra{\mathcal{O}}{\C C}$$ by restricting the action of $\mathcal{P}$ to $\mathcal{O}$ along $\phi$. This functor is the identity on underlying objects in $\C C$, hence it preserves 
fibrations and weak equivalences. 
Moreover, the functor $\phi^*$ has a left adjoint $\phi_{*}$, therefore we have a Quillen adjunction,
\begin{equation}\label{qp}
\xymatrix{\algebra{\mathcal{O}}{\C C}\ar@<.5ex>[r]^-{\phi_{*}}& \algebra{\mathcal{P}}{\C C}.\ar@<.5ex>[l]^-{\phi^*}}
\end{equation}
The following result establishes conditions so that this is a Quillen equivalence if $\phi$ is a weak equivalence of operads. These conditions are the non-symmetric analogues of those considered in \cite{ahto} for symmetric operads.

\begin{thm}\label{elt3}
In the conditions of the previous theorem, assume further that $\C C$ is left proper. 
Let $\phi\colon\mathcal{O}\r\mathcal{P}$ be a weak equivalence between operads in $\C{V}$ such that for all $n\geq 0$ the objects $\mathcal{O}(n)$ and $\mathcal{P}(n)$ are cofibrant in $\C V$. Then \eqref{qp} is a Quillen equivalence, in particular the derived adjoint pair is an equivalence between the homotopy categories of algebras, 
$$\xymatrix{\ho\algebra{\mathcal{O}}{\C C}\ar@<.8ex>[r]^-{\mathbb{L}\phi_{*}}& \ho\algebra{\mathcal{P}}{\C C}.\ar@<.2ex>[l]^-{\phi^*}}$$
\end{thm}

This result will be useful to show that in many examples the homotopy theory of enriched categories coincides with the homotopy theory of $A_{\infty}$-categories, e.g. when the underlying symmetric monoidal category $\C V$ is any of the categories in the examples (1)--(6) listed above, see Section \ref{appl}.

\bigskip

\noindent\textbf{Acknowledgements.} The author wishes to thank Michael Batanin, Clemens Ber\-ger, Benoit Fresse, Javier J. Guti\'errez, Ieke Moerdijk, Andy Tonks and Bruno Vallette for conversations related to the contents of this paper, in particular for providing very interesting references.

\bigskip

\noindent\textbf{Notation.} Throughout this paper $\C{V}$ and $\C{C}$ will denote  complete and cocomplete biclosed monoidal categories \cite[1.5]{bcect} with tensor product  $X\otimes Y$ and unit object $\unit_{\C V}$ and $\unit_{\C C}$, respectively. We drop the subscript when it is clear from the context. The category $\C{V}$ will be symmetric and internal morphism objects in $\C{V}$ will be denoted by $\hom(X,Y)$. We will add homotopical hypotheses when needed.

\section{Operads}

In this section we recall the well-known notion of non-symmetric operad.

\begin{defn}
 The category~$\C{V}^{\mathbb{N}}$ of \emph{sequences} of objects $V=\{V(n)\}_{n\geq0}$ in~$\C{V}$ 
is the product of countably many copies of $\C{V}$. It
 has a right-closed non-symmetric monoidal structure given by the \emph{composition product}~$U\circ V$,
\begin{align*}
(U\circ V)(m)&=\coprod_{n\geq 0}\coprod_{\sum\limits_{i=1}^{n}p_{i}=m}U(n)\otimes V(p_{1})\otimes\cdots\otimes V(p_{n}).
\end{align*}
The unit object is $\unit_{\circ}$,
$$\unit_{\circ}(n)=\left\{
\begin{array}{ll}
\unit,&\text{the unit of }\otimes\text{ in }\C{V},\text{ if }n=1;\\
0,&\text{the initial object of }\C{V},\text{ if }n\neq 1.
\end{array}
\right.$$
\end{defn}

\begin{rem}
The fact that $\circ$ is non-symmetric is obvious from the very definition. One can easily check by writing down explicitly the  formulas of $(U\circ V)\circ W$ and $U\circ (V\circ W)$ how the symmetry constraint of $\otimes$ is used to define the associativity constraint of $\circ$. The right adjoint of $-\circ V$ is the functor $\hom_{\circ}(V,-)$ defined by
\begin{align*}
\hom_{\circ}(V, W)(n)&=\prod_{p_1,\dots,p_n\geq 0}\hom(V(p_{1})\otimes\cdots\otimes V(p_{n}), W(p_1+\cdots+p_n)),
\end{align*}
in particular $-\circ V$ preserves all colimits. On the contrary, the functor $U\circ-$ does not preserve all colimits, but it does preserve filtered colimits.
\end{rem}

\begin{rem}\label{jn}
If $\C{V}$ is a model category 
then the product category $\C{V}^{\mathbb{N}}$  is also a model category with fibrations, cofibrations and weak equivalences defined  coordinatewise \cite[Example 1.1.6]{hmc}. Moreover, if~$\C{V}$ is cofibrantly generated (resp. combinatorial) then $\C{V}^{\mathbb{N}}$  is also cofibrantly generated (resp. combinatorial). 

Indeed, let $I$ be a set of generating cofibrations  and $J$ a set of generating trivial cofibrations in $\C V$. For any $n\geq 0$, let $s_{n}\colon \C{V}\r \C{V}^{\mathbb{N}}$ be  the left adjoint of the projection onto the $n^{\text{th}}$ factor, which is defined by 
$$(s_{n}(V))(m)=\left\{
\begin{array}{ll}
V,&\text{if }m=n;\\
0,&\text{the initial object, if }m\neq n.
\end{array}
\right.$$
Given a set $S$ of morphisms in $\C{V}$ we consider the following set of morphisms 
in~$\C{V}^{\mathbb{N}}$,
$$S_{\mathbb{N}}=\bigcup_{n\geq 0}s_{n}(S).$$
The sets $I_{\mathbb{N}}$ and $J_{\mathbb{N}}$ are sets of   generating cofibrations and   generating trivial cofibrations in~$\C{V}^{\mathbb{N}}$, respectively. 

\end{rem}

\begin{defn}
 A \emph{non-symmetric operad} $\mathcal O$ in  $\C{V}$ is a monoid in the monoidal category of sequences $\C{V}^{\mathbb{N}}$ with the composition product $\circ$.
\end{defn}

\begin{rem}
The previous condensed definition of an operad $\mathcal{O}$ can be unraveled by noticing that the multiplication $\mu\colon \mathcal{O}\circ \mathcal{O}\r \mathcal{O}$ consists of a series of multiplication morphisms, $1\leq i\leq n$, $p_{i}\geq 0$, 
$$\mu_{n;p_1,\dots,p_n}\colon \mathcal{O}(n)\otimes \mathcal{O}(p_{1})\otimes\cdots\otimes \mathcal{O}(p_{n})\To \mathcal{O}(p_1+\cdots +p_n).$$
The associativity condition amounts to say that the following diagram is always commutative,
$$\xy
(0,0)*+{\mathcal{O}(n)\otimes\bigotimes\limits_{i=1}^{n}\left(\mathcal{O}(p_{i})\otimes
\bigotimes\limits_{j=1}^{p_{i}}\mathcal{O}(q_{ij})\right)},
(0,-20)*+{\left(\mathcal{O}(n)\otimes\bigotimes\limits_{i=1}^{n}\mathcal{O}(p_{i})\right)\otimes
\bigotimes\limits_{i=1}^{n}\bigotimes\limits_{j=1}^{p_{i}}\mathcal{O}(q_{ij})},
(50,10)*+{\mathcal{O}(n)\otimes\bigotimes\limits_{i=1}^{n}\mathcal{O}\left(\sum\limits_{j=1}^{p_{i}}q_{ij}\right)},
(50,-30)*+{\mathcal{O}\left(\sum\limits_{i=1}^{n}p_{i}\right)\otimes
\bigotimes\limits_{i=1}^{n}\bigotimes\limits_{j=1}^{p_{i}}\mathcal{O}(q_{ij})},
(80,-10)*+{\mathcal{O}\left(\sum\limits_{i=1}^{n}\sum\limits_{j=1}^{p_{i}}q_{ij}\right)},
\ar(0,-5);(0,-15)_{\cong}^{\text{assoc. and sym.}}
\ar(26,3);(37,7)_{\id{}\otimes\bigotimes\limits_{i=1}^{n}\mu}
\ar(29,-20);(39,-24)^{\quad\mu\otimes\id{}}
\ar(69,7);(80,-4)^{\mu}
\ar(67,-27);(80,-16)_{\mu}
\endxy$$
Here the order of tensor factors in $
\bigotimes_{i=1}^{n}\bigotimes_{j=1}^{p_{i}}\mathcal{O}(q_{ij})$ is determined by the lexicographic order of the pair $(i,j)$. Moreover,  the unit is just a morphism $u\colon\unit\r \mathcal{O}(1)$ such that the following morphisms are (compositions of) unit constraints  in $\C V$,
$$\xymatrix@C=30pt@R=10pt{
\unit\otimes \mathcal{O}(n)\ar[r]^-{u\otimes\id{}}& \mathcal{O}(1)\otimes \mathcal{O}(n)\ar[r]^-{\mu_{1;n}}& \mathcal{O}(n),\\
\mathcal{O}(n)\otimes \unit^{\otimes n}\ar[r]^-{\id{}\otimes u^{\otimes n}}& \mathcal{O}(n)\otimes \mathcal{O}(1)^{\otimes n}
\ar[r]^-{\mu_{n;1,\dots,1}}& \mathcal{O}(n).
}$$

\end{rem}

\begin{rem}\label{circi}
The multiplication morphisms in the previous remark are determined by the  following morphisms, $1\leq i\leq m$, $n\geq 0$,
$$\circ_i\colon \mathcal{O}(m)\otimes \mathcal{O}(n)\To \mathcal{O}(m+n-1),$$
defined as
$$
\xymatrix{\mathcal{O}(m)\otimes \mathcal{O}(n)\ar[d]^{\cong}_{(\text{left and right unit})^{-1}}\\ 
\mathcal{O}(m)\otimes \unit^{\otimes (i-1)}\otimes \mathcal{O}(n)\otimes \unit^{\otimes (m-i)}
\ar[d]^{\id{}\otimes u^{\otimes (i-1)}\otimes\id{}\otimes u^{\otimes (m-i)}}\\
\mathcal{O}(m)\otimes \mathcal{O}(1)^{\otimes (i-1)}\otimes \mathcal{O}(n)\otimes \mathcal{O}(1)^{\otimes (m-i)}
\ar[d]^{\mu_{m;1,\st{i-1}{\dots\dots},1,n,1,\st{m-i}{\dots\dots},1}}\\ \mathcal{O}(m+n-1)}
$$

An operad can actually be defined as a collection of morphisms $\circ_{i}$ as above together with a unit morphism $u\colon\unit\r \mathcal{O}(1)$ such that, for $1\leq i\leq m$, the following diagrams commute:
\begin{enumerate}
\item If  $1\leq j<i$,
$$\xy
(0,0)*{(\mathcal{O}(l)\otimes \mathcal{O}(m))\otimes \mathcal{O}(n)},
(0,-20)*{(\mathcal{O}(l)\otimes \mathcal{O}(n))\otimes \mathcal{O}(m)},
(40,10)*{\mathcal{O}(l+m-1)\otimes \mathcal{O}(n)},
(40,-30)*{\mathcal{O}(l+n-1)\otimes \mathcal{O}(m)},
(80,-10)*{\mathcal{O}(l+m+n-2)},
(12,8)*{\scriptstyle \circ_{i}\otimes\id{}},
(13,-28)*{\scriptstyle \circ_{j}\otimes\id{}},
(75,0)*{\scriptstyle \circ_{j}},
(78,-20)*{\scriptstyle \circ_{
i+n-1
}},
\ar(0,-3);(0,-17)_{\cong}^{\text{assoc. and sym.}}
\ar(9,2);(23,9)
\ar(9,-22);(23,-29)
\ar(57,9);(80,-7)
\ar(57,-29);(80,-13)
\endxy$$

\item If  $i\leq j<m+i$,
$$\xy
(0,0)*{(\mathcal{O}(l)\otimes \mathcal{O}(m))\otimes \mathcal{O}(n)},
(0,-20)*{\mathcal{O}(l)\otimes (\mathcal{O}(m)\otimes \mathcal{O}(n))},
(40,10)*{\mathcal{O}(l+m-1)\otimes \mathcal{O}(n)},
(40,-30)*{\mathcal{O}(l)\otimes \mathcal{O}(m+n-1)},
(80,-10)*{\mathcal{O}(l+m+n-2)},
(12,8)*{\scriptstyle \circ_{i}\otimes\id{}},
(12,-28)*{\scriptstyle \id{}\otimes\circ_{j-i+1}},
(75,0)*{\scriptstyle \circ_{j}},
(75,-20)*{\scriptstyle \circ_{i}},
\ar(0,-3);(0,-17)_{\cong}^{\text{assoc.}}
\ar(9,2);(23,9)
\ar(9,-22);(23,-29)
\ar(57,9);(80,-7)
\ar(57,-29);(80,-13)
\endxy$$

\end{enumerate}
These relations are illustrated by the trees in Figures \ref{asoc1} and  \ref{asoc2} below. 
Moreover, for all $1\leq i\leq n$ the following composite morphisms must be unit constraints  in $\C V$,
\begin{flushleft}
$\xymatrix@C=30pt@R=10pt{\quad\;\,\text{(3)}\qquad
\unit\otimes \mathcal{O}(n)\ar[r]^-{u\otimes\id{}}& \mathcal{O}(1)\otimes \mathcal{O}(n)\ar[r]^-{\circ_{1}}& \mathcal{O}(n),\\
\quad\;\,\text{(4)}\qquad\mathcal{O}(n)\otimes \unit\ar[r]^-{\id{}\otimes u}& \mathcal{O}(n)\otimes \mathcal{O}(1)
\ar[r]^-{\circ_{i}}& \mathcal{O}(n).
}$
\end{flushleft}

\end{rem}

\section{Trees}

The combinatorics of operads is that of trees with additional structure. In this section we recall some facts about trees that we need in order to prove our main theorems. We also give a different characterization of operads in terms of trees.

\begin{defn}
A \emph{planted tree} is a contractible finite $1$-dimensional simplicial complex $T$ with set of vertices $V(T)$, a non-empty set of edges $E(T)$, and a distinguished vertex $r(T)\in V(T)$ of degree $1$, called \emph{root}. Recall that the \emph{degree} of $v\in V(T)$ is the number of edges containing $v$. Nevertheless, we will mostly use the following number,
$$\val{T}{v}=(\text{degree of }v)-1.$$

The \emph{level} of a vertex $v\in V(T)$ is the distance to the root, $\level{v}=d(v,r(T))$, with respect to the usual metric $d$ such that the distance between two adjacent vertices $\{v,w\}\in E(T)$ is $d(v,w)=1$. The \emph{height} $\operatorname{ht}(T)$ of a planted  tree $T$ is $$\operatorname{ht}(T)=\max_{v\in V(T)}\level{v}.$$
\end{defn}

\begin{defn}
A \emph{planted planar tree} is a planted tree $T$ together with a total order $\leq$ in $V(T)$, called \emph{planar order}, such that:
\begin{itemize}
 \item If $\level{v}<\level{w}$ then $v< w$.

\item If $\{v_1,v_2\}, \{w_1,w_2\}\in E(T)$ are edges with $$\level{v_1}=\level{w_1}
=\level{v_2}-1=\level{w_2}-1,$$ and $v_1< w_1$, then $v_2< w_2$.
\end{itemize}

\begin{figure}[h]
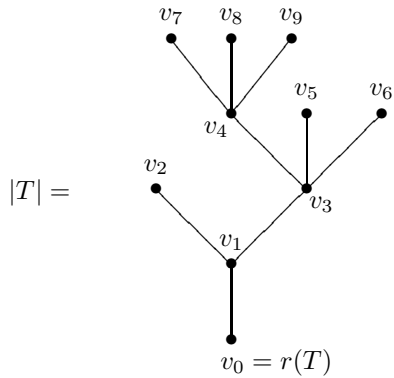

$$|T|=\qquad\begin{array}{c}
\xy
(0,0)*-{\bullet}="a",
(6,-3)*{v_{0}=r(T)},
(0,10)*-{\bullet}="b",
(0,13)*{v_1},
(-10,20)*-{\bullet}="c",
(-10,23)*{v_2},
(10,20)*-{\bullet}="d",
(12,18)*{v_3},
(0,30)*-{\bullet}="e",
(-2,28)*{v_4},
(10,30)*-{\bullet}="f",
(10,33)*{v_5},
(20,30)*-{\bullet}="g",
(20,33)*{v_6},
(-8,40)*-{\bullet}="h",
(-8,43)*{v_7},
(0,40)*-{\bullet}="i",
(0,43)*{v_8},
(8,40)*-{\bullet}="j",
(8,43)*{v_9},
\ar@{-}"a";"b",
\ar@{-}"b";"c",
\ar@{-}"b";"d",
\ar@{-}"d";"e",
\ar@{-}"d";"f",
\ar@{-}"d";"g",
\ar@{-}"e";"h",
\ar@{-}"e";"i",
\ar@{-}"e";"j",
\endxy
\end{array}$$
\caption{The geometric realization of a planted planar tree $T$ with vertices ordered by the subscript.}
\label{ppt}
\end{figure}

Given $e=\{v,w\}\in E(T)$ with $v<w$ we say that $e$ is \emph{an incoming edge} of $v$ and \emph{the outgoing edge} of $w$ (there is only one if $w\neq r(T)$ and none otherwise).

There is another useful order in $V(T)$ that we call the \emph{path order} $\preceq$. Given $v\in V(T)$, consider the shortest path from $r(T)$ to $v$ and let $r(T)=v_{0},\dots,v_{n}=v$ be the vertices within this path in order of appearance. We associate with $v$ the word $v_{0}\cdots v_{n}$ in $V(T)$. The path order in $V(T)$ is the order induced by the lexicographic order of words in $V(T)$ with respect to $\leq$.
\end{defn}

\begin{rem}
Notice that the path order $\preceq$ restricted to level sets $$\{v\in V(T)\, ;\, \level{v}=n\}, \quad n\geq 0,$$ coincides always with the planar order $\leq$. 

The words associated to the vertices  of the planted planar tree in Figure  \ref{ppt} are
$$\begin{array}{c|c}
\text{vertex}&\text{word}\\
\hline 
v_0&v_0\\
v_1&v_0v_1\\
v_2&v_0v_1v_2\\
v_3&v_0v_1v_3\\
v_4&v_0v_1v_3v_4\\
\end{array}\qquad \qquad \qquad
\begin{array}{c|c}
\text{vertex}&\text{word}\\
\hline 
v_5&v_0v_1v_3v_5\\
v_6&v_0v_1v_3v_6\\
v_7&v_0v_1v_3v_4v_7\\
v_8&v_0v_1v_3v_4v_8\\
v_9&v_0v_1v_3v_4v_9\\
\end{array}$$
hence the path order in $V(T)$ is $v_0\prec v_1\prec v_2\prec v_3\prec v_4\prec v_7\prec v_8\prec v_9\prec v_5\prec v_6$. 
\end{rem}

\begin{defn}
A \emph{planted planar tree with leaves} is a planted planar tree $T$ together with a fixed set of degree $1$ vertices $L(T)$, called \emph{leaves}, different from the root, $r(T)\notin L(T)$. An \emph{inner vertex} is a vertex which is neither a leaf nor the root. The set of inner vertices will be denoted by $I(T)$,
$$V(T)=\{r(T)\}\sqcup I(T)\sqcup L(T).$$
We denote $\norm{T}$ the open subspace of the geometric realization of $T$ obtained by removing the root and the leaves, see Figure \ref{pptl}, 
$$\norm{T}=|T|\setminus(\{r(T)\}\sqcup L(T)).$$

\begin{figure}[h]
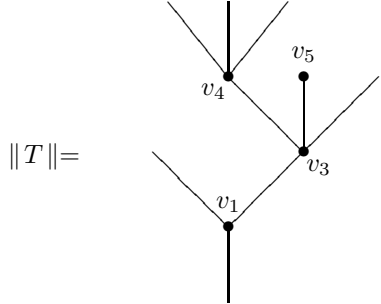

$$\norm{T}=\qquad\begin{array}{c}
\xy
(0,0)*{}="a",
(0,10)*-{\bullet}="b",
(0,13)*{v_1},
(-10,20)*{}="c",
(10,20)*-{\bullet}="d",
(12,18)*{v_3},
(0,30)*-{\bullet}="e",
(-2,28)*{v_4},
(10,30)*-{\bullet}="f",
(10,33)*{v_5},
(20,30)*{}="g",
(-8,40)*{}="h",
(0,40)*{}="i",
(8,40)*{}="j",
\ar@{-}"a";"b",
\ar@{-}"b";"c",
\ar@{-}"b";"d",
\ar@{-}"d";"e",
\ar@{-}"d";"f",
\ar@{-}"d";"g",
\ar@{-}"e";"h",
\ar@{-}"e";"i",
\ar@{-}"e";"j",
\endxy
\end{array}$$
\caption{The space $\norm{T}$ for the planted planar tree in Figure \ref{ppt} with set of leaves $L(T)=\{v_2,v_6,v_7,v_8,v_9\}$.}
\label{pptl}
\end{figure}

Abusing of terminology, we say that an edge is the \emph{root} or a \emph{leaf} if it contains the root or a leaf vertex, respectively. The rest of edges are called \emph{inner edges}. 

Given $n\geq 0$, the \emph{corolla} with $n$ leaves is a planted planar tree  $C_{n}$ with $n+2$ vertices and $n$ leaves, see Figure \ref{corollas}.
\begin{figure}[h]
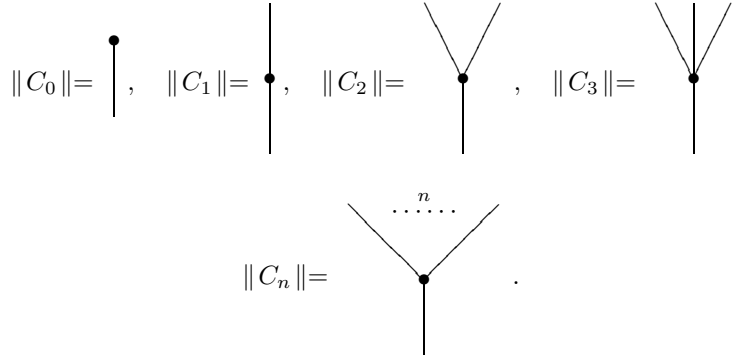

$$\norm{C_{0}}={}\begin{array}{c}
\xy
(0,0)*{}="a",
(0,10)*-{\bullet}="b",
\ar@{-}"a";"b",
\endxy
\end{array},\quad
\norm{C_{1}}={}\begin{array}{c}
\xy
(0,0)*{}="a",
(0,10)*-{\bullet}="b",
(0,20)*{}="c",
\ar@{-}"a";"b",
\ar@{-}"b";"c",
\endxy
\end{array},\quad
\norm{C_{2}}={}\begin{array}{c}
\xy
(0,0)*{}="a",
(0,10)*-{\bullet}="b",
(-5,20)*{}="c",
(5,20)*{}="d",
\ar@{-}"a";"b",
\ar@{-}"b";"c",
\ar@{-}"b";"d",
\endxy
\end{array},\quad
\norm{C_{3}}={}\begin{array}{c}
\xy
(0,0)*{}="a",
(0,10)*-{\bullet}="b",
(0,20)*{}="c",
(-5,20)*{}="d",
(5,20)*{}="e",
\ar@{-}"a";"b",
\ar@{-}"b";"c",
\ar@{-}"b";"d",
\ar@{-}"b";"e",
\endxy
\end{array},$$

$$\norm{C_{n}}={}\begin{array}{c}
\xy
(0,0)*{}="a",
(0,10)*-{\bullet}="b",
(-10,20)*{}="c",
(10,20)*{}="d",
(0,20)*{\st{n}{\cdots\cdots}},
\ar@{-}"a";"b",
\ar@{-}"b";"c",
\ar@{-}"b";"d",
\endxy
\end{array}.$$
\caption{A class of planted planar trees with leaves: the corollas~$C_{n}$, $n\geq 0$.}
\label{corollas}
\end{figure}

A \emph{morphism} of planted planar trees with leaves is a simplicial map $f\colon T\r T'$ such that:
\begin{itemize}
\item If $v\preceq w\in V(T)$ then $f(v)\preceq f(w)\in V(T')$.
\item 
$f^{-1}(\{r(T')\})=\{r(T)\}$.
\item $\card L(T)=\card L(T')$ and $f^{-1}(L(T'))=L(T)$.
\end{itemize}
We denote  ${\bf PPTL}$ the category of planted planar trees with leaves. Notice that this category has no non-trivial automorphism.
\end{defn}

\begin{rem}\label{contraction}
Any morphism $f\colon T\r T'$ is uniquely determined by the inner edges $e=\{v,w\}\in E(T)$ that $f$ contracts $f(v)=f(w)$. Moreover, given a planted planar tree with leaves $T$ and an inner edge $e=\{v,w\}\in E(T)$ the quotient tree $T/e$, obtained by contracting $e$ to a vertex $[e]\in V(T/e)$, carries a unique structure of planted planar tree with leaves such that the natural projection $p_{e}^T\colon T\r T/e$ is a morphism in ${\bf PPTL}$, see Figure \ref{pptl2}. This morphism induces identifications
\begin{align*}
V(T)\setminus\{v,w\}&= V(T/e)\setminus\{[e]\},&E(T)\setminus\{e\}&= E(T/e). 
\end{align*}

\begin{figure}[h]
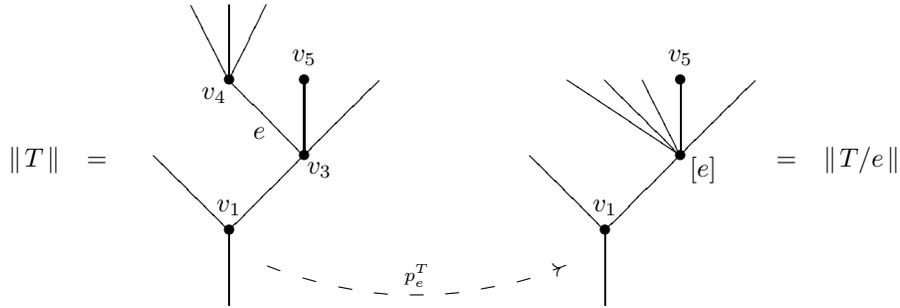

$$\norm{T}\quad=\quad\begin{array}{c}
\xy
(0,0)*{}="a",
(0,10)*-{\bullet}="b",
(0,13)*{v_1},
(-10,20)*{}="c",
(10,20)*-{\bullet}="d",
(12,18)*{v_3},
(0,30)*-{\bullet}="e",
(4,23)*{e},
(-2,28)*{v_4},
(10,30)*-{\bullet}="f",
(10,33)*{v_5},
(20,30)*{}="g",
(-5,40)*{}="h",
(0,40)*{}="i",
(5,40)*{}="j",
(50,0)*{}="aa",
(50,10)*-{\bullet}="bb",
(50,13)*{v_1},
(40,20)*{}="cc",
(60,20)*-{\bullet}="dd",
(63,18)*{[e]},
(60,30)*-{\bullet}="ff",
(60,33)*{v_5},
(70,30)*{}="gg",
(45,30)*{}="hh",
(50,30)*{}="ii",
(55,30)*{}="jj",
\ar@{-}"a";"b",
\ar@{-}"b";"c",
\ar@{-}"b";"d",
\ar@{-}"d";"e",
\ar@{-}"d";"f",
\ar@{-}"d";"g",
\ar@{-}"e";"h",
\ar@{-}"e";"i",
\ar@{-}"e";"j",
\ar@{-}"aa";"bb",
\ar@{-}"bb";"cc",
\ar@{-}"bb";"dd",
\ar@{-}"dd";"ff",
\ar@{-}"dd";"gg",
\ar@{-}"dd";"hh",
\ar@{-}"dd";"ii",
\ar@{-}"dd";"jj",
\ar@/_10pt/@{-->}(6,5);(44,5)^{p^{T}_{e}}
\endxy
\end{array}=\quad\norm{T/e}$$
\caption{The morphism $p_{e}^{T}\colon T\r T/e$ in ${\bf PPTL}$ contracting the inner edge $e=\{v_{3},v_{4}\}$.}
\label{pptl2}
\end{figure}

One can similarly define a morphism $p^T_K\colon T\r T/K$ in $\mathbf{PPTL}$ contracting the connected components of any subcomplex $K\subset T$ formed by inner edges, see Figure~\ref{dm2} below for a more complicated example.
\end{rem}

\begin{defn}
Given a planted planar tree  $T$ with $n$ leaves  and $n$ planted planar trees with leaves $T_{1},\dots, T_{n}$, we denote $T(T_{1},\dots,T_{n})$ the planted planar tree with the same root as $T$, the leaves are the disjoint union of the leaves of all $T_{i}$, and the space $\norm{T(T_{1},\dots,T_{n})}$ is obtained by \emph{grafting} the root edge of $\norm{T_{i}}$ in the $i^{\text{th}}$ leaf edge of $\norm{T}$ with respect to the path order $\preceq$ in $L(T)\subset V(T)$, $1\leq i\leq n$, see Figure \ref{graft}.

Grafting is associative, i.e.
\begin{align*}
&T(T_{1}(T_{1,1},\dots,T_{1,p_{1}}),\dots\dots,T_{n}(T_{n,1},\dots,T_{n,p_{n}}))\\
&\qquad=(T(T_{1},\dots,T_{n}))(T_{1,1},\dots,T_{1,p_{1}},\dots\dots,T_{n,1},\dots,T_{n,p_{n}}).
\end{align*}

The planted planar tree $U$ with only one edge and one leaf, $\norm{U}=|$, is a unit for the grafting operation, 
\begin{align*}
U(T)&=T,&T(U,\dots,U)=T.
\end{align*}

\begin{figure}[h]
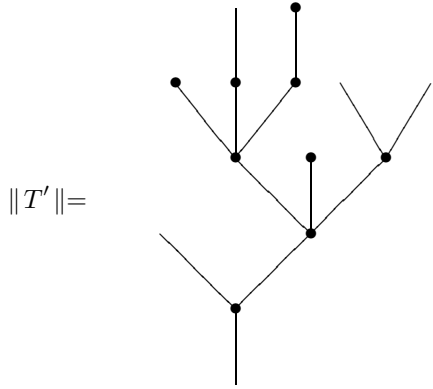

$$\norm{T'}=\qquad\begin{array}{c}
\xy
(0,0)*{}="a",
(0,10)*-{\bullet}="b",
(-10,20)*{}="c",
(10,20)*-{\bullet}="d",
(0,30)*-{\bullet}="e",
(10,30)*-{\bullet}="f",
(20,30)*-{\bullet}="g",
(-8,40)*-{\bullet}="h",
(0,40)*-{\bullet}="i",
(8,40)*-{\bullet}="j",
(0,50)*{}="k",
(8,50)*-{\bullet}="l",
(14,40)*{}="m",
(26,40)*{}="n",
\ar@{-}"a";"b",
\ar@{-}"b";"c",
\ar@{-}"b";"d",
\ar@{-}"d";"e",
\ar@{-}"d";"f",
\ar@{-}"d";"g",
\ar@{-}"e";"h",
\ar@{-}"e";"i",
\ar@{-}"e";"j",
\ar@{-}"i";"k",
\ar@{-}"j";"l",
\ar@{-}"g";"m",
\ar@{-}"g";"n",
\endxy
\end{array}$$
\caption{The grafting $T'=T(U,C_{0},C_{1},C_{1}(C_{0}),C_{2})$ for $T$ in Figure \ref{pptl}.}
\label{graft}
\end{figure}
\end{defn}

The category ${\bf PPTL}$ splits as the coproduct of the full subcategories ${\bf PPTL}(n)$ of trees with $n$ leaves, 
$${\bf PPTL}=\coprod_{n\geq 0}{\bf PPTL}(n).$$  
Notice that the grafting operation is functorial in ${\bf PPTL}$ in the sense of the following obvious lemma.

\begin{lem}
The sequence $\{{\bf PPTL}(n)\}_{n\geq0}$ with the grafting operation and the unit $U$ is an operad   in the cartesian closed category of small categories.
\end{lem}


\begin{lem}
All planted planar trees with leaves can be  obtained by grafting corollas  and $U$.
\end{lem}

\begin{proof}
By induction on the height of the  planted planar tree withe leaves $T$. 
On the one hand, if $\operatorname{ht}(T)=1$ then $T=U$ or $C_{0}$. On the other hand, any $T\neq U, C_{0}$ can be decomposed as $T=C_{n}(T_{1},\dots,T_{n})$ where $n+1$ is the degree of the unique level $1$ vertex of $T$ and $\operatorname{ht}(T_{i})<\operatorname{ht}(T)$,  $1\leq i\leq n$.
\end{proof}

For instance,  $T$ in Figure \ref{pptl} is $$T=C_{2}(U,C_{3}(C_{3},C_{0},U))=C_{2}\circ_{2}((C_{3}\circ_{2}C_{0})\circ_{1}C_{3}).$$


\begin{defn}
An \emph{operadic functor} with values in $\C V$ is  a functor $$\mathcal{G}\colon {\bf PPTL}\To\C{V}$$
equipped with a  \emph{unit morphism} $u\colon\unit\r \mathcal{G}(C_{1})$ and natural isomorphisms
$$\mathcal{G}(T({T_{1},\dots,T_{n}}))\cong \mathcal{G}(T)\otimes \mathcal{G}(T_{1})\otimes\cdots\otimes \mathcal{G}(T_{n}),$$
that we call \emph{grafting isomorphisms}, 
such that:
\begin{itemize}
\item $\mathcal{G}(U)=\unit$.

\item The following composition of grafting isomorphisms is a coherent composition of associativity and symmetry constraints in~$\C V$,
\begin{align*}
&\mathcal{G}(T)
\otimes
\mathcal{G}(T_{1})\otimes\mathcal{G}(T_{1,1})\otimes\dots\otimes\mathcal{G}(T_{1,p_{1}})\otimes\cdots\dots\otimes\mathcal{G}(T_{n})\otimes
\mathcal{G}(T_{n,1})\otimes\dots\otimes \mathcal{G}(T_{n,p_{n}})\\
&\cong
\mathcal{G}(T)
\otimes
\mathcal{G}(T_{1}(T_{1,1},\dots,T_{1,p_{1}}))\otimes\cdots\dots\otimes\mathcal{G}(T_{n}(T_{n,1},\dots,T_{n,p_{n}}))\\
&\cong\mathcal{G}(T(T_{1}(T_{1,1},\dots,T_{1,p_{1}}),\dots\dots,T_{n}(T_{n,1},\dots,T_{n,p_{n}})))\\
&=\mathcal{G}((T(T_{1},\dots,T_{n}))(T_{1,1},\dots,T_{1,p_{1}},\dots\dots,T_{n,1},\dots,T_{n,p_{n}}))\\
&\cong
\mathcal{G}(T(T_{1},\dots,T_{n}))
\otimes
\mathcal{G}(T_{1,1})\otimes\dots\otimes\mathcal{G}(T_{1,p_{1}})\otimes\cdots\cdots\otimes\mathcal{G}(T_{n,1})\otimes
\dots\otimes\mathcal{G}(T_{n,p_{n}})\\
&\cong\mathcal{G}(T)\otimes\mathcal{G}(T_{1})\otimes\cdots\otimes\mathcal{G}(T_{n})\otimes
\mathcal{G}(T_{1,1})\otimes\dots\otimes\mathcal{G}(T_{1,p_{1}})\otimes\cdots\cdots\otimes\mathcal{G}(T_{n,1})\otimes
\dots\otimes\mathcal{G}(T_{n,p_{n}}).
\end{align*}

\item The following grafting isomorphisms are (compositions of) unit constraints in~$\C V$,
$$\xymatrix@R=10pt{\unit\otimes\mathcal{G}(T)=\mathcal{G}(U)\otimes \mathcal{G}(T)\ar[r]^-{\cong}_-{\text{grafting}}& \mathcal{G}(U(T))=\mathcal{G}(T),\\
\mathcal{G}(T')\otimes \unit\otimes\cdots\otimes \unit=\mathcal{G}(T')
\otimes \mathcal{G}(U)\otimes\cdots\otimes \mathcal{G}(U)
\ar[r]_-{\text{grafting}}^-{\cong}& \mathcal{G}(T(U,\dots,U))=\mathcal{G}(T).}$$

\item Suppose $T'=C_{1}(T)$, see Figure \ref{1i}. Let $f\colon T'\r T$ be the morphism which contracts the incoming  edge of the level $1$ vertex of $T'$. 
Then the following morphism is the left unit constraint in $\C V$,
$$\xymatrix{\unit\otimes\mathcal{G}(T)\ar[r]^-{u\otimes\id{}}&\mathcal{G}(C_{1})\otimes \mathcal{G}(T)\ar[r]^-{\cong}_-{\text{grafting}}& \mathcal{G}(T')\ar[r]^-{\mathcal{G}(f)}& \mathcal{G}(T).}$$
\begin{figure}[h]
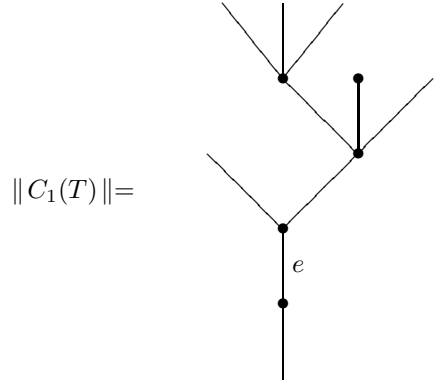

$$\norm{C_{1}(T)}=\qquad\begin{array}{c}
\xy
(0,-10)*{}="o",
(2,5)*{e},
(0,0)*-{\bullet}="a",
(0,10)*-{\bullet}="b",
(-10,20)*{}="c",
(10,20)*-{\bullet}="d",
(0,30)*-{\bullet}="e",
(10,30)*-{\bullet}="f",
(20,30)*{}="g",
(-8,40)*{}="h",
(0,40)*{}="i",
(8,40)*{}="j",
\ar@{-}"o";"a",
\ar@{-}"a";"b",
\ar@{-}"b";"c",
\ar@{-}"b";"d",
\ar@{-}"d";"e",
\ar@{-}"d";"f",
\ar@{-}"d";"g",
\ar@{-}"e";"h",
\ar@{-}"e";"i",
\ar@{-}"e";"j",
\endxy
\end{array}$$
\caption{The planted planar tree with leaves $T'=C_{1}({T})$ for $T$ as in Figure \ref{pptl}. Here we denote $e$ the incoming  edge of the level $1$ vertex of $T'$.}
\label{1i}
\end{figure}

\item Suppose $T'=T(C_{1},\dots,C_{1})$, see Figure \ref{1d}. Let $f\colon T'\r T$ be the morphism which contracts all the inner edges  adjacent to the leaf edges in $T'$. Then the following morphism is a composition of  right unit constraints in $\C V$,
$$\xymatrix@C=19pt{\mathcal{G}(T)\otimes \unit\otimes\cdots\otimes \unit\ar[rr]^-{\id{}\otimes u\otimes\cdots\otimes u}&& \mathcal{G}(T)
\otimes \mathcal{G}(C_{1})\otimes\cdots\otimes \mathcal{G}(C_{1})
\ar[r]^-{\cong}_-{\text{grafting}}& \mathcal{G}(T')\ar[r]^-{\mathcal{G}(f)}& \mathcal{G}(T).}$$
\begin{figure}[h]
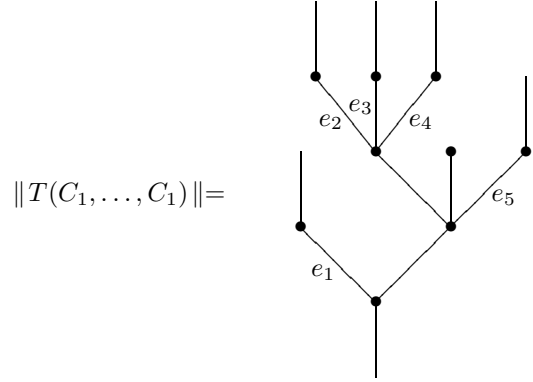

$$\norm{{T}(C_{1},\dots,C_{1})}=\qquad\begin{array}{c}
\xy
(0,0)*{}="a",
(-7,14)*{e_{1}},
(-6,34)*{e_{2}},
(-2,36)*{e_{3}},
(6,34)*{e_{4}},
(17,24)*{e_{5}},
(0,10)*-{\bullet}="b",
(-10,20)*-{\bullet}="c",
(-10,30)*{}="cc",
(10,20)*-{\bullet}="d",
(0,30)*-{\bullet}="e",
(10,30)*-{\bullet}="f",
(20,30)*-{\bullet}="g",
(20,40)*{}="gg",
(-8,40)*-{\bullet}="h",
(-8,50)*{}="hh",
(0,40)*-{\bullet}="i",
(0,50)*{}="ii",
(8,40)*-{\bullet}="j",
(8,50)*{}="jj",
\ar@{-}"a";"b",
\ar@{-}"b";"c",
\ar@{-}"b";"d",
\ar@{-}"d";"e",
\ar@{-}"d";"f",
\ar@{-}"d";"g",
\ar@{-}"e";"h",
\ar@{-}"e";"i",
\ar@{-}"e";"j",
\ar@{-}"c";"cc",
\ar@{-}"g";"gg",
\ar@{-}"h";"hh",
\ar@{-}"i";"ii",
\ar@{-}"j";"jj"
\endxy
\end{array}$$
\caption{The planted planar tree with leaves $T'={T}(C_{1},\dots,C_{1})$ for $T$ as in Figure \ref{pptl}. Here we denote $e_{i}$ the inner edges  adjacent to the leaf edges in $T'$.}
\label{1d}
\end{figure}

\end{itemize}

A \emph{morphism} of operadic functors $\varphi\colon\mathcal{G}\r\mathcal{H}$ is a natural transformation compatible with the grafting isomorphisms, with the unit morphism, and such that $\varphi(U)=\id{\unit}$.
\end{defn}

The following equivalence between operads and operadic functors was sketched by Ginzburg and Kapranov in the symmetric case \cite[1.2]{kdo}.

\begin{prop}\label{ofo}
There is an equivalence between the categories of operads in $\C V$ and operadic functors with values in $\C V$.
\end{prop}

\begin{proof}
Denote $\operadf{\C V}$ the category of operadic functors with values in $\C V$. We are going to define adjoint equivalences 
$$\xymatrix{\operad{\C V}\ar@<.5ex>[r]^-{L} &\operadf{\C V} \ar@<.5ex>[l]^-{R}.}$$

Given an operadic functor $\mathcal{G}$ we set $$R(\mathcal{G})(n)=\mathcal{G}(C_{n}),$$ the unit of the operad $R(\mathcal{G})$ is $u\colon \unit\r \mathcal{G}(C_{1})=R(\mathcal{G})(1)$, and multiplications in $R(\mathcal{G})$ are defined by the morphisms $$f_{n;p_{1},\dots,p_{m}}\colon C_{n}(C_{p_{1}},\dots,C_{p_{n}})\To C_{p_{1}+\dots + p_{n}}$$ which contract all inner edges,
$$\xymatrix{
R(\mathcal{G})(n)\otimes R(\mathcal{G})(p_{1})\otimes\cdots\otimes R(\mathcal{G})(p_{n})
\ar[dd]_{\mu_{n;p_{1},\dots,p_{n}}}\ar@{=}[r]&
\mathcal{G}(C_{n})\otimes \mathcal{G}(C_{p_{1}})\otimes\cdots\otimes \mathcal{G}(C_{p_{n}})
\ar[d]_{\cong}^{\text{grafting}}\\
&\mathcal{G}(C_{n}(C_{p_{1}},\dots,C_{p_{n}}))\ar[d]^{\mathcal{G}(f_{n;p_{1},\dots,p_{n}})}\\
R(\mathcal{G})({p_{1}+\dots + p_{n}})\ar@{=}[r]&\mathcal{G}(C_{p_{1}+\dots + p_{n}})}$$

Conversely, if $\mathcal{O}$ is an operad then the corresponding operadic functor $L(\mathcal{O})$ is defined on objects as
$$L(\mathcal{O})(T)=\bigotimes_{u\in I(T)} \mathcal{O}(\val{T}{u}),$$
see Figure \ref{lot}.
\begin{figure}[h]
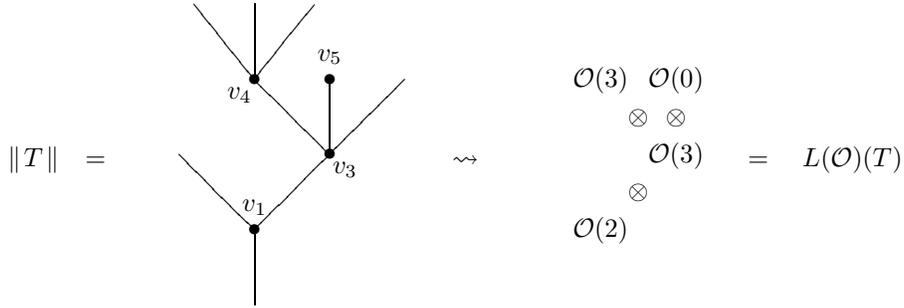

$$\norm{T}\quad=\qquad\begin{array}{c}
\xy
(0,0)*{}="a",
(0,10)*-{\bullet}="b",
(0,13)*{v_1},
(-10,20)*{}="c",
(10,20)*-{\bullet}="d",
(12,18)*{v_3},
(0,30)*-{\bullet}="e",
(-2,28)*{v_4},
(10,30)*-{\bullet}="f",
(10,33)*{v_5},
(20,30)*{}="g",
(-8,40)*{}="h",
(0,40)*{}="i",
(8,40)*{}="j",
\ar@{-}"a";"b",
\ar@{-}"b";"c",
\ar@{-}"b";"d",
\ar@{-}"d";"e",
\ar@{-}"d";"f",
\ar@{-}"d";"g",
\ar@{-}"e";"h",
\ar@{-}"e";"i",
\ar@{-}"e";"j"
\endxy
\end{array}\quad\leadsto\quad
\begin{array}{c}
\xy
(0,10)*-<2pt>{\mathcal{O}(2)}="b",
(-10,20)*{}="c",
(10,20)*-{ \mathcal{O}(3)}="d",
(0,30)*-<1pt>{\mathcal{O}(3)}="e",
(10,30)*-<2pt>{\mathcal{O}(0)}="f",
\ar@{}"b";"d"|{\displaystyle\otimes}
\ar@{}"d";"e"|{\displaystyle\otimes}
\ar@{}"d";"f"|{\displaystyle\otimes}
\endxy
\end{array}\quad=\quad L(\mathcal{O})(T)$$
\caption{The object $L(\mathcal{O})(T)$ associated to the planted planar tree with leaves $T$ in Figure \ref{pptl}.}
\label{lot}
\end{figure}
The morphism  $u\colon \unit\r \mathcal{O}(1)=L(\mathcal{O})(C_{1})$ is the unit of the operad. Grafting isomorphisms are coherent compositions of associativity and symmetry constraints in~$\C V$. Moreover, let $T$ be a planted planar tree with leaves and $e=\{v,w\}\in E(T)$ an inner edge which is the $i^{\text{th}}$ incoming edge of $v$. The morphism induced by the natural projection $p_{e}^{T}\colon T\r T/e$ in Remark \ref{contraction} is 
$$\xy
(0,0)*{L(\mathcal{O})(T)},
(0,-40)*{L(\mathcal{O})(T/e)},
(58,-1.5)*{\mathcal{O}(\val{T}{v})\otimes \mathcal{O}(\val{T}{w})\otimes \hspace{-20pt} \bigotimes\limits_{u\in I(T)\setminus\{v,w\}}\!\!\!\!\!\!\!\!\!\! \mathcal{O}(\val{T}{u}) },
(63,-41.5)*{ \mathcal{O}(\val{T/e}{[e]})\otimes \hspace{-20pt}\bigotimes\limits_{u\in I(T/e)\setminus\{[e]\}}\!\!\!\!\!\!\!\!\!\!\! \mathcal{O}(\val{T/e}{u})},
\ar@{->}(8,0);(38,0)^{\cong}_{\text{symmetry}}
\ar@{->}(10,-40);(48,-40)^{\cong}_{\text{symmetry}}
\ar(0,-3);(0,-37)_{L(\mathcal{O})(p_{e}^{T})}
\ar(58,-5);(58,-36)^{\circ_{i}\otimes\id{}}
\endxy$$
Here we use that 
$
\val{T/e}{[e]}
=\val{T}{v}+\val{T}{w}-1
$ 
,
see Figure \ref{lot2}. 
\begin{figure}[h]
$$
\begin{array}{c}
\xy
(5,40)*{L(\mathcal{O})(T)},
(5,35)*[right]{=},
(0,10)*-<2pt>{\mathcal{O}(2)}="b",
(10,20)*-{ \mathcal{O}(3)}="d",
(0,30)*-<1pt>{\mathcal{O}(3)}="e",
(10,30)*-<2pt>{\mathcal{O}(0)}="f",
(-5,34);(5,25),{\ellipse<,11pt>{--}}
\ar@{}"b";"d"|{\displaystyle\otimes}
\ar@{}"d";"e"|{\displaystyle\otimes}
\ar@{}"d";"f"|{\displaystyle\otimes}
\ar(16,20);(70,20)^{\circ_{1}}
\ar(13,40);(60,40)^{L(\mathcal{O})(p_{e}^{T})}
\endxy
\end{array}\hspace{-33pt}
\begin{array}{c}
\vspace{-2pt}
\\
\xy
(5,40)*{L(\mathcal{O})(T/e)},
(5,35)*[right]{=},
(0,10)*-<2pt>{\mathcal{O}(2)}="b",
(10,20)*-{ \mathcal{O}(5)}="d",
(0,30)*{}="e",
(10,30)*-<2pt>{\mathcal{O}(0)}="f",
\ar@{}"b";"d"|{\displaystyle\otimes}
\ar@{}"d";"f"|{\displaystyle\otimes}
\endxy
\end{array}
$$
\caption{The morphism $L(\mathcal{O})(p_{e}^{T})$ for $T$ and $e=\{v_{3},v_{4}\}$ as in Figure \ref{pptl2}, see also Figure \ref{lot}.}
\label{lot2}
\end{figure}

The unit natural transformation $\mathcal{O}\r RL(\mathcal{O})$ is the identity morphism, and the counit $\varepsilon\colon LR(\mathcal{G})\r \mathcal{G}$ is defined by grafting isomorphisms,
$$\xy
(-12,-1.5)*{\varepsilon(T)\colon LR(\mathcal{G})(T)=\!\!\!\!\bigotimes\limits_{u\in I(T)}\!\!\!\!\mathcal{G}(C_{\val{T}{u}})},
(40,0)*{\mathcal{G}(T).}
\ar(11,0);(34,0)^-{\cong}_-{\text{grafting}}
\endxy$$
Here we use that any planted planar tree with leaves $T$ can be obtained by grafting appropriately the corollas $C_{\val{T}{u}}$, $u\in I(T)$, compare the previous lemma.
\end{proof}

Examples of planted planar trees with leaves illustrating relations~(1) and~(2) in 
Remark \ref{circi} are depicted in Figures \ref{asoc1} and  \ref{asoc2}, 
respectively.

\begin{figure}[h]
$$\xy
(0,0)*{}="a",
(0,10)*-<1pt>{\bullet}="b",
(-19,20)*-{\bullet}="c",
(0,20)*-<1pt>{\bullet}="d",
(19,20)*{}="e",
(-27,30)*{}="f",
(-23,30)*{}="g",
(-19,30)*{}="h",
(-15,30)*{}="i",
(-11,30)*{}="j",
(-6,30)*{}="k",
(-2,30)*{}="l",
(2,30)*{}="m",
(6,30)*{}="n",
\ar@{-}"a";"b",
\ar@{-}"b";"c",
\ar@{-}"b";"d",
\ar@{-}"b";"e",
\ar@{-}"c";"f",
\ar@{-}"c";"g",
\ar@{-}"c";"h",
\ar@{-}"c";"i",
\ar@{-}"c";"j",
\ar@{-}"d";"k",
\ar@{-}"d";"l",
\ar@{-}"d";"m",
\ar@{-}"d";"n",
\endxy$$
\caption{The planted planar tree with leaves illustrating the associativity relation $(C_l\circ_iC_m)\circ_jC_n=(C_l\circ_jC_n)\circ_{i+n-1}C_m$ in Remark \ref{circi} (1) for $l=3$, $m=4$, $n=5$, and $j=1<i=2$.} 
\label{asoc1}
\end{figure}

\begin{figure}[h]
$$\xy
(0,0)*{}="a",
(0,10)*-{\bullet}="b",
(-10,20)*{}="c",
(0,20)*-<1pt>{\bullet}="d",
(10,20)*{}="e",
(-10,40)*{}="f",
(-6,40)*{}="g",
(-2,40)*{}="h",
(2,40)*{}="i",
(6,40)*{}="j",
(-6,30)*{}="k",
(-2,30)*-<1pt>{\bullet}="l",
(2,30)*{}="m",
(6,30)*{}="n",
\ar@{-}"a";"b",
\ar@{-}"b";"c",
\ar@{-}"b";"d",
\ar@{-}"b";"e",
\ar@{-}"l";"f",
\ar@{-}"l";"g",
\ar@{-}"l";"h",
\ar@{-}"l";"i",
\ar@{-}"l";"j",
\ar@{-}"d";"k",
\ar@{-}"d";"l",
\ar@{-}"d";"m",
\ar@{-}"d";"n",
\endxy$$
\caption{The planted planar tree with leaves illustrating the associativity relation $(C_l\circ_iC_m)\circ_jC_n=C_l\circ_i(C_m\circ_{j-i+1}C_n)$ in Remark \ref{circi} (2) for $l=3$, $m=4$, $n=5$, and $i=2\leq j=3<m+i=6$.} 
\label{asoc2}
\end{figure}


\section{The monoidal category of morphisms}\label{mcm}

The \emph{category $\mor{\C{C}}$ of morphisms in $\C C$} can be regarded as the category of functors $\dos\r\C{C}$, where $\dos$ is the category with two objects, $0$ and $1$, and only one non-identity morphism $0\r 1$, i.e. it is the poset $\{0<1\}$. A morphism $f\colon U\r V$ in $\C C$ is identified with the functor $f\colon\dos\r\C{C}$ defined by $f(0)=U$, $f(1)=V$ and $f(0\r 1)=f$.

The category $\mor{\C{C}}$ carries a biclosed monoidal structure given by the $\odot$ product of morphisms $f\odot g$,
$$\xy
(0,0)*+{U\otimes X}="a",
(35,0)*+{V\otimes X}="b",
(0,-18.2)*+{U\otimes Y}="c",
(35,-20)*+{U\otimes Y\!\bigcup\limits_{U\otimes X}\! V\otimes X}="d",
(70,-35)*+{V\otimes Y}="e"
\ar"a";"b"^{f\otimes \id{X}}
\ar@{}"a";"d"|{\text{push}}\ar"a";"c"_{\id{U}\otimes g}
\ar"b";"d"
\ar"c";(20,-18.2)
\ar@/^15pt/"b";"e"^{\id{V}\otimes g}
\ar@/_15pt/"c";"e"_{f\otimes\id{Y}}
\ar(40,-23);"e"^-{f\odot g}
\endxy$$
This monoidal structure is symmetric provided  $\otimes$ is. If $0$ denotes the initial object of $\C C$, the functor 
\begin{align*}
\C{C}&\To\mor{\C C},\\
X&\;\mapsto\; (0\r X),
\end{align*}
is strong (symmetric) monoidal. We regard $\C C$ as a full subcategory of $\mor{\C C}$ through this functor.

Notice that push-outs in $\C C$ are a special kind of morphism in $\mor{\C C}$. The following lemma asserts that the $\odot$ product preserves push-outs in $\C C$.

\begin{lem}\label{podot}
Given two push-out diagrams in $\C{C}$, $i=1,2$,
$$\xymatrix{U_i\ar[r]^{f_i}\ar[d]_{g_i}\ar@{}[rd]|{\text{push}}&V_i\ar[d]^{g'_i}\\X_i\ar[r]_{f'_i}&Y_i}$$
the following diagram in $\C{C}$ is also a push-out,
$$\xy
(0,-2)*{U_1\otimes V_2\!\!\!\!\bigcup\limits_{U_1\otimes U_2}\!\!\!\!V_1\otimes U_2},
(0,-22)*{X_1\otimes Y_2\!\!\!\!\bigcup\limits_{X_1\otimes X_2}\!\!\!\!Y_1\otimes X_2},
(40,0)*{V_1\otimes V_2},
(40,-20)*{Y_1\otimes Y_2},
(20,-10)*{\text{\scriptsize push}}
\ar(16,0);(33,0)^-{f_1\odot f_2}
\ar(40,-3);(40,-17)^-{g_1'\otimes g_2'}
\ar(16,-20);(33,-20)_-{f'_1\odot f'_2}
\ar(0,-6);(00,-17)_-{g_1\otimes g_2' \!\!\!\!\bigcup\limits_{g_1\otimes g_2}\!\!\!\! g_1'\otimes g_2}
\endxy$$
\end{lem}

This lemma follows straightforwardly from the very definition of $\odot$ together with the fact that $\otimes$ is biclosed, and hence it preserves colimits in both variables.

\begin{defn}\label{defm}
The category $\C C$ is a \emph{monoidal model category} if it is endowed with a model structure satisfying the \emph{push-out product axiom}: 
\begin{itemize}
\item Let  $f$ and $g$ be cofibrations in $\C{C}$. The morphism $f\odot g$ is also a cofibration. If in addition $f$ or $g$ is a weak equivalence, then so is $f\odot g$. 
\end{itemize} 
This axiom was considered in \cite[Definition 3.1]{ammmc} for $\C{C}$ symmetric, but it also makes sense in the non-symmetric case.
\end{defn}

\begin{rem}
The push-out product axiom implies that the tensor product of cofibrant objects is cofibrant. Moreover, if $X$ is a cofibrant object and $f$ is a (trivial) cofibration in  $\C C$ then  $X\otimes f$ and $f\otimes X$ are (trivial) cofibrations. In particular, by Ken Brown's lemma \cite[Lemma 1.1.12]{hmc}, for $X$ cofibrant the functors $X\otimes -$ and $-\otimes X$ preserve weak equivalences between cofibrant objects. Furthermore, if~$f$ and $g$ are (trivial) cofibrations with cofibrant source, then so is $f\odot g$.
\end{rem}

\begin{lem}\label{wedot}
Let $\C{C}$ be a left proper monoidal model category. Consider two commutative squares in $\mor{\C C}$ where the rows are cofibrations and the columns are weak equivalences between cofibrant objects,  $i=1,2$,
$$\xymatrix{U_i\ar@{ >->}[r]^{f_i}\ar[d]_{g_i}^{\sim}    
&V_i\ar[d]^{g'_i}_{\sim}\\X_i\ar@{ >->}[r]_{f'_i}&Y_i}$$
Then in the following diagram the rows are also cofibrations and the columns are weak equivalences between cofibrant objects,
$$\xy
(0,-2)*{U_1\otimes V_2\!\!\!\!\bigcup\limits_{U_1\otimes U_2}\!\!\!\!V_1\otimes U_2},
(0,-22)*{X_1\otimes Y_2\!\!\!\!\bigcup\limits_{X_1\otimes X_2}\!\!\!\!Y_1\otimes X_2},
(40,0)*{V_1\otimes V_2},
(40,-20)*{Y_1\otimes Y_2},
(20,-10)*{\text{\scriptsize push}}
\ar@{ >->}(16,0);(33,0)^-{f_1\odot f_2}
\ar(40,-3);(40,-17)^-{g_1'\otimes g_2'}_{\sim}
\ar@{ >->}(16,-20);(33,-20)_-{f'_1\odot f'_2}
\ar(0,-6);(00,-17)_-{g_1\otimes g_2' \!\!\!\!\bigcup\limits_{g_1\otimes g_2}\!\!\!\! g_1'\otimes g_2}^{\sim}
\endxy$$
\end{lem}

\begin{proof}
Looking at Definition \ref{defm} and the remark afterwards we notice that it is only left to check that the left column is a weak equivalence. This follows easily from the  gluing property in left proper model  categories \cite[Proposition 13.5.4]{hirschhorn}. 
\end{proof}

Given morphisms $f_{i}\colon U_{i}\r V_{i}$ in $\C C$, $1\leq i\leq n$, the target of $f_{1}\odot\cdots\odot f_{n}$ is the iterated tensor product of the targets $V_{1}\otimes\cdots\otimes V_{n}$. This object is the colimit of the diagram
$$f_{1}\otimes\cdots\otimes f_{n}\colon \dos^{n}\To\C{C},$$
since $\dos^{n}$ has a final object $(1,\st{n}\dots,1)$. The source of $f_{1}\odot\cdots\odot f_{n}$ is the colimit of the restriction of this diagram to the full subcategory of $\dos^{n}$ obtained by removing the final object.  For simplicity, we  denote it by $s(f_{1}\odot\cdots\odot f_{n})$,
$$f_{1}\odot\cdots\odot f_{n}\colon s(f_{1}\odot\cdots\odot f_{n})\To V_{1}\otimes\cdots\otimes V_{n}.$$
The universal property of $s(f_{1}\odot\cdots\odot f_{n})$ in $\C C$ refers to canonical morphisms
$$\kappa_{i}\colon V_{1}\otimes\cdots\otimes V_{i-1}\otimes U_{i}\otimes V_{i+1}\otimes\cdots\otimes V_{n}\To
s(f_{1}\odot\cdots\odot f_{n}),\quad 1\leq i\leq n,$$
with $(f_{1}\odot\cdots\odot f_{n})\kappa_{i}=\id{}^{\otimes(i-1)}\otimes f_{i}\otimes\id{}^{\otimes(n-i)}$.
Any collection of morphisms
$$g_{i}\colon V_{1}\otimes\cdots\otimes V_{i-1}\otimes U_{i}\otimes V_{i+1}\otimes\cdots\otimes V_{n}\To
X,\quad 1\leq i\leq n,$$
such that the following squares commute, $1\leq i<j\leq n$,
$$\xymatrix@C=50pt{
V_{1}\otimes\cdots\otimes U_{i}\otimes \cdots\otimes U_{j}\otimes \cdots\otimes V_{n}
\ar[r]^-{
\begin{array}{c}
\scriptstyle
\id{}\otimes\cdots\otimes f_{i}\otimes\cdots\otimes\id{}
\\
\vspace{-10pt}
\end{array}
}
\ar[d]_-{
\id{}\otimes\cdots\otimes f_{j}\otimes\cdots\otimes\id{}
}&
V_{1}\otimes\cdots\otimes V_{i}\otimes \cdots\otimes U_{j}\otimes \cdots\otimes V_{n}\ar[d]^-{g_{j}}\\
V_{1}\otimes\cdots\otimes U_{i}\otimes \cdots\otimes V_{j}\otimes \cdots\otimes V_{n}\ar[r]_-{g_{i}}&
X
}$$
induces a unique morphism $g\colon s(f_{1}\odot\cdots\odot f_{n}) \r X$ such that $g_{i}=g\kappa_{i}$, $1\leq i\leq n$.

\section{The relevant operad push-out}\label{ropo}

The forgetful functor from operads to sequences $\operad{\C V}\r\C{V}^{\mathbb{N}}$ has a left adjoint $\mathcal{F}\colon\C{V}^{\mathbb{N}}\r\operad{\C V}$, the \emph{free operad} functor, explicitly constructed for example in \cite[Appendix B]{cmmc}. An alternative construction using trees is as follows,
$$\mathcal{F}(V)(n)=\coprod_{T}\bigotimes_{v\in I(T)}V(\val{T}{v}),$$
where $T$ runs over a set of isomorphism classes of trees with $n$ leaves in $\mathbf{PPTL}$. The product $\circ_{i}$, $1\leq i\leq m$, 
$$\xy
(0,0)*{ \mathcal{F}(V)(m)\otimes \mathcal{F}(V)(n)=\coprod\limits_{T'}\bigotimes\limits_{u\in I(T')}V(\val{T'}{u})\;\otimes\; \coprod\limits_{T}\bigotimes\limits_{v\in I(T)}V(\val{T}{v})},
(60,1.5)*{\card L(T')=m},
(60,-4.5)*{\card L(T)=n},
(11.4,-10)*{\cong\coprod\limits_{T',T}\bigotimes\limits_{u\in I(T')}\hspace{-7pt}V(\val{T'}{u})\otimes\hspace{-6pt}\bigotimes\limits_{v\in I(T)}\hspace{-6pt}V(\val{T}{v})},
(-10,-30)*{\mathcal{F}(V)(m+n-1)=
\coprod\limits_{T''}\bigotimes\limits_{w\in I(T'')}V(\val{T''}{w})},
(53,-28.5)*{\card L(T'')=m+n-1},
\ar
(-26.5,-2);(-26.5,-25)_-{\circ_{i}}
\endxy$$
sends the factor corresponding to the trees $T$ and $T'$ in the source to the factor of $T''=T'\circ_{i}T$ in the target,
\begin{align*}
I(T'\circ_{i}T)&=I(T')\sqcup I(T),\\
\bigotimes\limits_{u\in I(T')}\hspace{-7pt}V(\val{T'}{u})\otimes\hspace{-5pt}\bigotimes\limits_{v\in I(T)}\hspace{-5pt}V(\val{T}{v})&=
\hspace{-10pt}\bigotimes\limits_{w\in I(T'\circ_{i}T)}\hspace{-10pt}V(\val{T'\circ_{i}T}{w}).
\end{align*}
The unit
$u\colon\unit\r\mathcal{F}(V)(1)$ is the inclusion of the factor of the coproduct correspoding to the tree with one leaf a no inner vertex, i.e. the unit of the grafting operation.

The unit of the adjunction $V\r \mathcal{F}(V)$ in $\C{V}^{\mathbb{N}}$ is given by the following morphisms in $\C V$, $n\geq 0$,
$$\xymatrix{V(n)\ar[rrr]^-{
\begin{array}{c}
\scriptstyle \text{inclusion of the factor}\vspace{-5pt}\\
\scriptstyle \text{corresponding to }C_n
\end{array}
}&&&\mathcal{F}(V)(n).}$$
Given an operad $\mathcal{O}$ with associated operadic functor $L(\mathcal{O})$, if we denote $p_T\colon T\r C_n$ the morphism in $\mathbf{PPTL}$ collapsing all inner edges of a tree $T$ with $n$ leaves, then the counit $\mathcal{F}(\mathcal{O})\r\mathcal{O}$ is defined by the following morphisms, $n\geq 0$,
$$\xy
(5,-1.5)*{\mathcal{F}(\mathcal{O})(n)=\coprod\limits_{T}\bigotimes\limits_{v\in I(T)}\mathcal{O}(\val{T}{v})=\coprod\limits_{T}L(\mathcal{O})(T)}, (70,0)*{L(\mathcal{O})(C_n)=\mathcal{O}(n).}
\ar(36,0);(55,0)^-{(L(\mathcal{O})(p_T))_T}
\endxy$$
An analogous construction for symmetric operads was considered by Ginzburg and Kapranov in \cite[2.1]{kdo}.

In this section we give an explicit construction of the push-out of a diagram in $\operad{\C{V}}$ as follows,
\begin{equation}\label{po}
\xymatrix{\mathcal{F}(U)\ar[d]_{g}\ar[r]^{\mathcal{F}(f)}&\mathcal{F}(V)\\
\mathcal{O}&}
\end{equation}

Consider the adjoint diagram in $\C{V}^{\mathbb N}$,
$$\xymatrix{U\ar[d]_{\bar g}\ar[r]^{f}&V\\
\mathcal{O}&}$$
The push-out of \eqref{po} is an operad $\mathcal{P}$ together with morphisms $f'\colon \mathcal{O}\r \mathcal{P}$ in $\operad{\C V}$ and $\bar{g}'\colon V\r \mathcal{P}$ in $\C{V}^{\mathbb N}$ such that
$f'\bar{g}=\bar{g}'f$ in $\C{V}^{\mathbb N}$. Moreover,  given an operad $\mathcal{P}'$ and morphisms $f''\colon \mathcal{O}\r \mathcal{P}'$ in $\operad{\C V}$ and $\bar{g}''\colon V\r \mathcal{P}'$ in $\C{V}^{\mathbb N}$ with $f''\bar{g}=\bar{g}''f$ in~$\C{V}^{\mathbb N}$, 
there is a unique morphism $h\colon \mathcal{P}\r \mathcal{P}'$ in $\operad{\C V}$ such that $f''=hf'$ and $\bar{g}''=h\bar{g}'$ in $\C{V}^{\mathbb N}$. 

Given a planted planar tree with leaves $T$ we denote
\begin{align*}
V^e(T)&=\{v\in V(T)\,;\,\level{v}\text{ is even}\},&V^o(T)&=V(T)\setminus V^e(T),\\
I^e(T)&=I(T)\cap V^e(T),&I^o(T)&=I(T)\cap V^o(T),
\end{align*}
see Figure \ref{eo}. From now on, we will only consider one tree in each isomorphism class of objects in ${\bf PPTL}$. 
\begin{figure}[h]
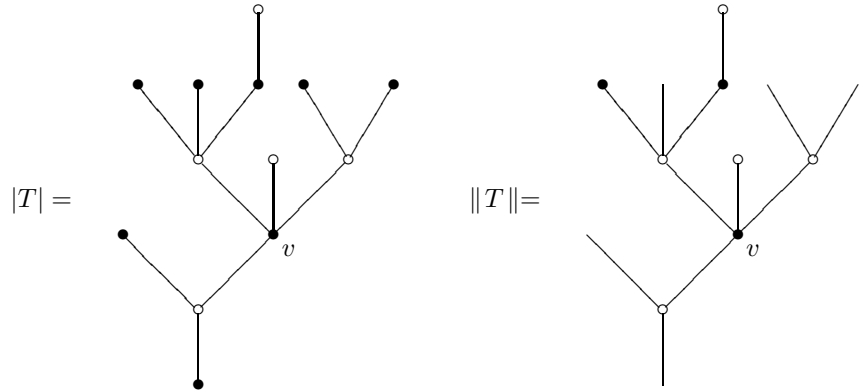

$$
|T|=\quad\begin{array}{c}
\xy
(0,0)*{\bullet}="a",
(0,10)*-<2pt>{\circ}="b",
(-10,20)*-<2pt>{\bullet}="c",
(10,20)*-{\bullet}="d",
(12,18)*{v},
(0,30)*-<1pt>{\circ}="e",
(10,30)*-<2pt>{\circ}="f",
(20,30)*-<2pt>{\circ}="g",
(-8,40)*-{\bullet}="h",
(0,40)*-{\bullet}="i",
(8,40)*-{\bullet}="j",
(0,50)*{}="k",
(8,50)*-<2pt>{\circ}="l",
(14,40)*-{\bullet}="m",
(26,40)*-{\bullet}="n",
\ar@{-}"a";"b",
\ar@{-}"b";"c",
\ar@{-}"b";"d",
\ar@{-}"d";"e",
\ar@{-}"d";"f",
\ar@{-}"d";"g",
\ar@{-}"e";"h",
\ar@{-}"e";"i",
\ar@{-}"e";"j",
\ar@{-}"j";"l",
\ar@{-}"g";"m",
\ar@{-}"g";"n",
\endxy
\end{array}
\qquad
\norm{T}=\quad\begin{array}{c}
\xy
(0,0)*{}="a",
(0,10)*-<2pt>{\circ}="b",
(-10,20)*{}="c",
(10,20)*-{\bullet}="d",
(12,18)*{v},
(0,30)*-<1pt>{\circ}="e",
(10,30)*-<2pt>{\circ}="f",
(20,30)*-<2pt>{\circ}="g",
(-8,40)*-{\bullet}="h",
(0,40)*{}="i",
(8,40)*-{\bullet}="j",
(0,50)*{}="k",
(8,50)*-<2pt>{\circ}="l",
(14,40)*{}="m",
(26,40)*{}="n",
\ar@{-}"a";"b",
\ar@{-}"b";"c",
\ar@{-}"b";"d",
\ar@{-}"d";"e",
\ar@{-}"d";"f",
\ar@{-}"d";"g",
\ar@{-}"e";"h",
\ar@{-}"e";"i",
\ar@{-}"e";"j",
\ar@{-}"j";"l",
\ar@{-}"g";"m",
\ar@{-}"g";"n",
\endxy
\end{array}
$$
\caption{For a  planted planar tree with leaves $T$, on the left (resp. right) we denote $\circ$ the vertices in $V^o(T)$ (resp. $I^o(T)$) and~$\bullet$ the vertices in $V^e(T)$ (resp. in $I^e(T)$).}
\label{eo}
\end{figure}

The idea behind our construction of the push-out of \eqref{po} is as follows. For any
planted planar tree with leaves concentrated in even levels, such as $T$ in Figure~\ref{eo}, we replace any inner even (resp. odd) vertex $v$ with the piece of $V$ (resp. $\mathcal{O}$) in
degree $\val{T}{v}$, and transform adjacency relations into tensor products.
\begin{center}
\includegraphics{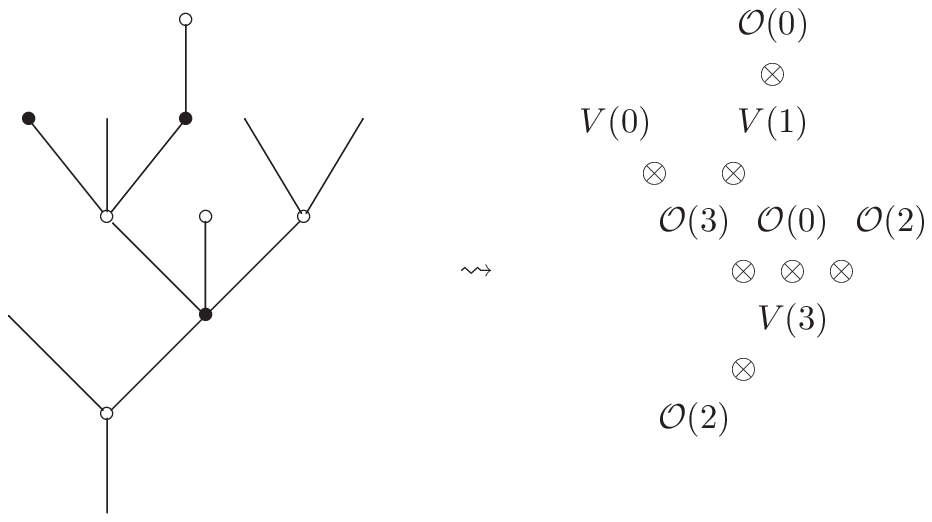}
\end{center}
In order to simplify the exposition of this intuitive idea, let us allow ourselves to
talk about elements of this object in $\C{V}$. We want to attach to $\mathcal{O}$ the product
of these elements in a coherent way. More precisely, if $T$ has $n$ leaves, we attach
these elements to $\mathcal{O}(n)$. For this, we must proceed by induction on the number of inner even vertices and require that, for any even inner vertex $v$, the image of the
morphism induced by $f(\val{T}{v})$,
\begin{center}
\includegraphics{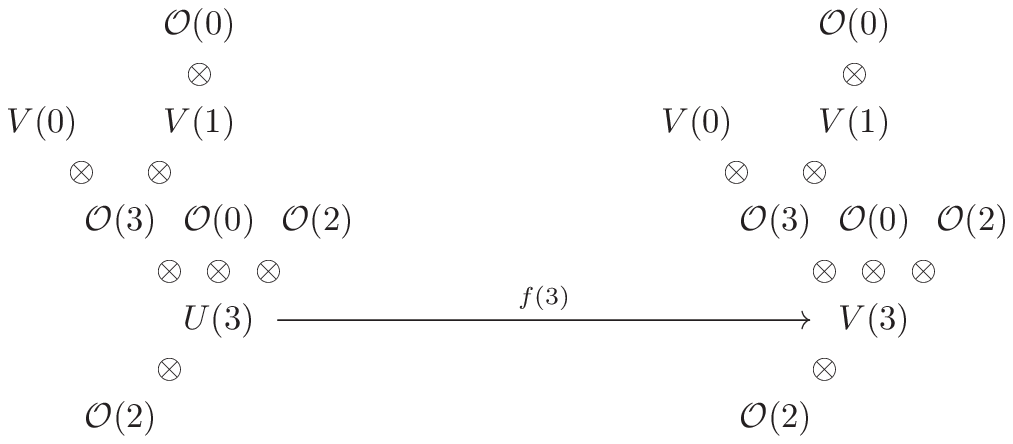}
\end{center}
is attached according to the attachment of the tree $T'$ with less even inner vertices
obtained from $T$ by contracting the edges surrounding $v$,
\begin{equation}\label{surround}
\includegraphics{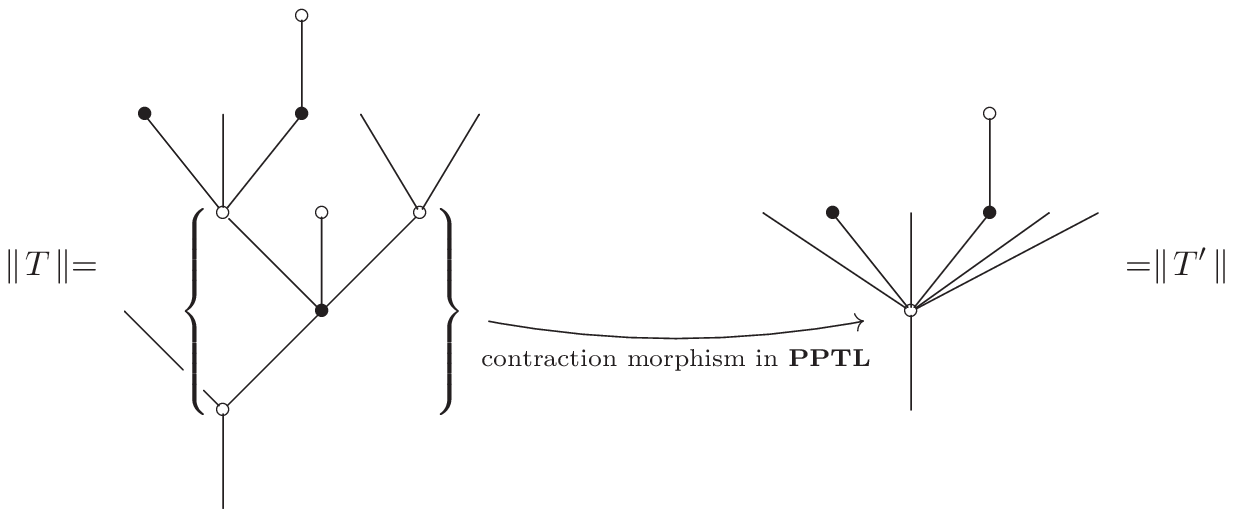}
\end{equation}
\begin{center}
\includegraphics{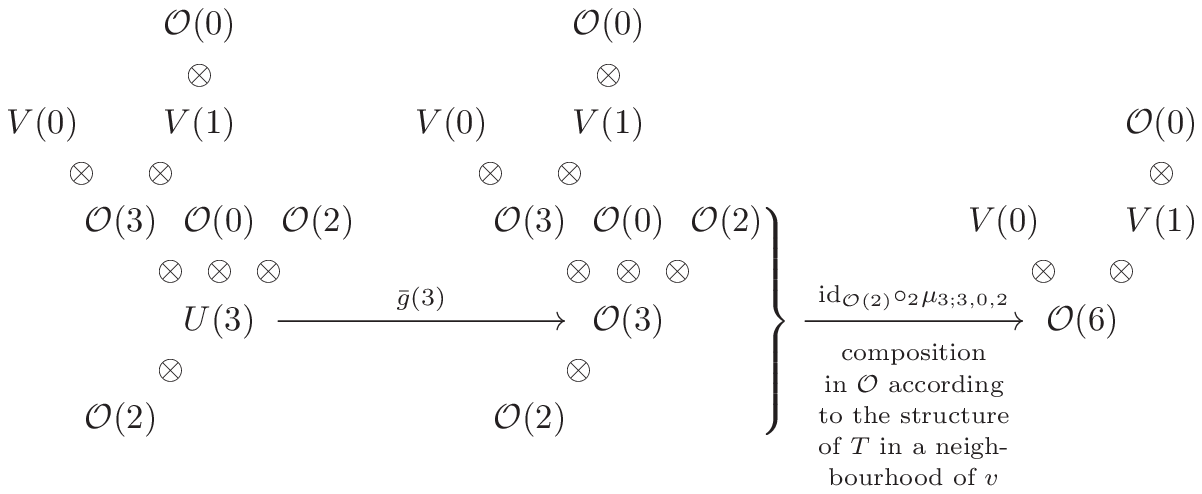}
\end{center}
This inductive construction is carried out in the following lemma. In order to state
it we need to introduce some terminology.

The \emph{star} of a vertex $v\in V(T)$ is the subtree $\star(v)\subset T$ formed by the edges containing $v$, and the link $\link(v)\subset V(T)$ consists of the vertices adjacent to $v$, see Figure \ref{starlink}.
\begin{figure}[h]
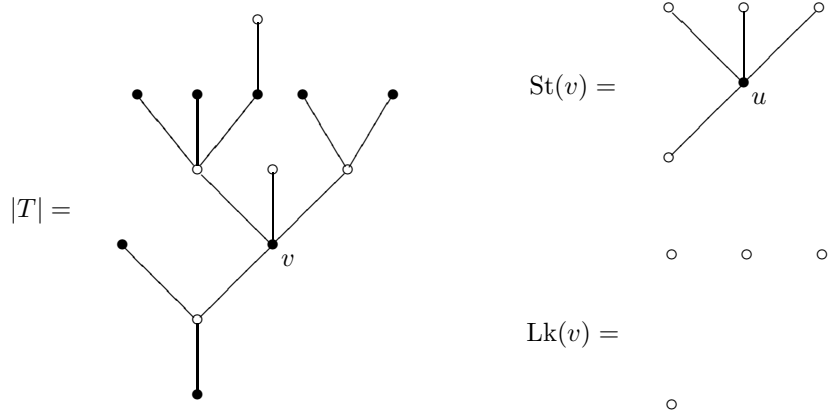

$$
|T|=\quad\begin{array}{c}
\xy
(0,0)*{\bullet}="a",
(0,10)*-<2pt>{\circ}="b",
(-10,20)*-<2pt>{\bullet}="c",
(10,20)*-{\bullet}="d",
(12,18)*{v},
(0,30)*-<1pt>{\circ}="e",
(10,30)*-<2pt>{\circ}="f",
(20,30)*-<2pt>{\circ}="g",
(-8,40)*-{\bullet}="h",
(0,40)*-{\bullet}="i",
(8,40)*-{\bullet}="j",
(0,50)*{}="k",
(8,50)*-<2pt>{\circ}="l",
(14,40)*-{\bullet}="m",
(26,40)*-{\bullet}="n",
\ar@{-}"a";"b",
\ar@{-}"b";"c",
\ar@{-}"b";"d",
\ar@{-}"d";"e",
\ar@{-}"d";"f",
\ar@{-}"d";"g",
\ar@{-}"e";"h",
\ar@{-}"e";"i",
\ar@{-}"e";"j",
\ar@{-}"j";"l",
\ar@{-}"g";"m",
\ar@{-}"g";"n",
\endxy
\end{array}\qquad\qquad
\begin{array}{c}
\star(v)=\quad\begin{array}{c}
\xy
(0,10)*-<2pt>{\circ}="b",
(10,20)*-{\bullet}="d",
(12,18)*{u},
(0,30)*-<1pt>{\circ}="e",
(10,30)*-<2pt>{\circ}="f",
(20,30)*-<2pt>{\circ}="g",
\ar@{-}"b";"d",
\ar@{-}"d";"e",
\ar@{-}"d";"f",
\ar@{-}"d";"g",
\endxy
\end{array}\vspace{30pt}\\
\link(v)=\quad\begin{array}{c}
\xy
(0,10)*-<2pt>{\circ}="b",
(0,30)*-<1pt>{\circ}="e",
(10,30)*-<2pt>{\circ}="f",
(20,30)*-<2pt>{\circ}="g",
\endxy
\end{array}
\end{array}
$$
\caption{The star and the link of the vertex $v$ of the tree $T$ in Figure \ref{eo}.}
\label{starlink}
\end{figure}
When the star is formed by inner edges,  the natural projection $$p_{\star(v)}^{T}\colon T\To T/\star(v)$$ is a morphism in ${\bf PPTL}$, see \eqref{surround}. This is the case if $v\in I^{e}(T)$ and $L(T)\subset V^{e}(T)$. Moreover, in this case $p_{\star(v)}^T$ induces identifications
\begin{align*}
I^{e}(T)\setminus\{v\}&= I^e(T/\star(v)),&
I^o(T)\setminus\link(v)&= I^o(T/\star(v))\setminus\{[\star(v)]\}.
\end{align*}
Furthermore, we will also consider the \emph{extended star} $\overline{\star}(v)\subset T$, which is the planted planar tree with leaves  whose inner part is $\star(v)$, the root edge is the outgoing edge of the minimum vertex $u\in\link(v)$, the leaves are the incoming edges of the vertices in $\link(v)$ except from $\{u,v\}$, and the planar order is the restriction of the planar order in $T$, see Figure \ref{extendedstar}.  Notice that $\overline{\star}(v)/{\star}(v)=C_{r_{v}}$, where
\begin{equation}\label{rsub}
r_{v}=\val{T/\star(v)}{[\star(v)]}=\card L(\overline{\star}(v))=\val{T}{u}-1+\hspace{-15pt}\sum_{w\in\link(v)\setminus\{u\}}\hspace{-15pt}
\val{T}{w}
=\left(\sum_{w\in\link(v)}\hspace{-5pt}\val{T}{w}\right)
-1.
\end{equation}
\begin{figure}[h]
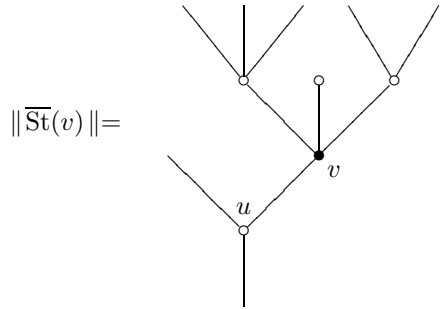

$$
\norm{\overline{\star}(v)}=\quad\begin{array}{c}
\xy
(0,0)*{}="a",
(0,10)*-<1pt>{\circ}="b",
(0,13)*{u},
(-10,20)*{}="c",
(10,20)*-{\bullet}="d",
(12,18)*{v},
(0,30)*-<1pt>{\circ}="e",
(10,30)*-<1pt>{\circ}="f",
(20,30)*-<1pt>{\circ}="g",
(-8,40)*{}="h",
(0,40)*{}="i",
(8,40)*{}="j",
(0,50)*{}="k",
(14,40)*{}="m",
(26,40)*{}="n",
\ar@{-}"a";"b",
\ar@{-}"b";"c",
\ar@{-}"b";"d",
\ar@{-}"d";"e",
\ar@{-}"d";"f",
\ar@{-}"d";"g",
\ar@{-}"e";"h",
\ar@{-}"e";"i",
\ar@{-}"e";"j",
\ar@{-}"g";"m",
\ar@{-}"g";"n",
\endxy
\end{array}
$$
\caption{The extended star of the vertex $v$ of the planted planar tree  with leaves $T$ in Figures \ref{eo} and \ref{starlink}.}
\label{extendedstar}
\end{figure}

The inductive construction of the push-out of \eqref{po} is the in following scaring lemma, whose statement is actually more complicated than its proof. For the sake of simplicity, from now on we use the same notation for an operad and for its associated operadic functor.

\begin{lem}\label{pind}
There is a sequence of morphisms in $\C{V}^{\mathbb N}$,
\begin{equation*}
\mathcal{O}=P_{0}\st{\varphi_1}\To P_{1}\r \cdots\r P_{t-1}\st{\varphi_t}\To P_t\r \cdots,
\end{equation*}
such that, for all $n\geq 0$, the morphism $\varphi_{t}(n)\colon P_{t-1}(n)\r P_{t}(n)$ is the push-out of the following coproduct of morphisms indexed by the set of planted trees with $n$ leaves concentrated in even levels and $t$ inner even vertices, i.e. $\card L(T)=n$, $L(T)\subset V^e(T)$, and $\card I^e(T)=t$, 
\begin{equation}\label{monstruo}
\coprod_{T}\bigodot_{v\in I^{e}(T)}f(\val{T}{v})\;\;\otimes\bigotimes_{w\in I^{o}(T)}\mathcal{O}(\val{T}{w}),
\end{equation}
along the unique morphism
\begin{equation}\label{monstruo2}
(\psi_{t}^{T})_{T}\colon \coprod_{T}s\left(\bigodot_{v\in I^{e}(T)}f(\val{T}{v})\right)\otimes\bigotimes_{w\in I^{o}(T)}\mathcal{O}(\val{T}{w})\To P_{t-1}(n)
\end{equation}
such that, given $u\in I^{e}(T)$, for $t=1$ the morphism $\psi_{1}^{T}$ is 
$$\xy
(0,0)*{U(\val{T}{u})\otimes 
\hspace{-7pt}\bigotimes\limits_{w\in I^o(T)}\hspace{-7pt}\mathcal{O}(\val{T}{w})},
(60,0)*{\mathcal{O}(\val{T}{u})\otimes 
\hspace{-7pt}\bigotimes\limits_{w\in I^o(T)}\hspace{-7pt}\mathcal{O}(\val{T}{w})=\mathcal{O}(T)},
(100,1.5)*{\mathcal{O}(n),},
\ar@/^20pt/(3,5);(96,5)^{\psi_{1}^{T}}
\ar@/_15pt/(3,-4);(49,-4)_{\bar{g}(\val{T}{u})\otimes\id{}}
\ar@/_10pt/(77,-1);(100,-1)_{\mathcal{O}(p_{T})}
\endxy$$
and for $t>1$ the composite morphism
$$
\xy
(0,0)*{U(\val{T}{u})\otimes 
\hspace{-15pt}\bigotimes\limits_{v\in I^{e}(T)\setminus\{u\}}\hspace{-15pt} V(\val{T}{v})\otimes
\hspace{-7pt}\bigotimes\limits_{w\in I^o(T)}\hspace{-7pt}\mathcal{O}(\val{T}{w})},
(60,0)*{s(\hspace{-8pt}\bigodot\limits_{v\in I^{e}(T)}\!\! f(\val{T}{v}))\otimes\hspace{-8pt}\bigotimes\limits_{w\in I^{o}(T)}\hspace{-8pt}\mathcal{O}(\val{T}{w})},
(95,1.5)*{P_{t-1}(n)}
\ar(24,1.5);(40,1.5)^-{\kappa_{u}\otimes\id{}}
\ar(79,1.5);(88,1.5)^-{\psi_t^T}
\endxy
$$
coincides with the following composition, that we call $\psi_{t,u}^T$, 
$$\xy
(0,0)*{U(\val{T}{u})\otimes 
\hspace{-15pt}\bigotimes\limits_{v\in I^{e}(T)\setminus\{u\}}\hspace{-15pt} V(\val{T}{v})\otimes
\hspace{-7pt}\bigotimes\limits_{w\in I^o(T)}\hspace{-7pt}\mathcal{O}(\val{T}{w})},
(0,-20)*{\mathcal{O}(\val{T}{u})\otimes 
\hspace{-15pt}\bigotimes\limits_{v\in I^{e}(T)\setminus\{u\}}\hspace{-15pt} V(\val{T}{v})\otimes
\hspace{-7pt}\bigotimes\limits_{w\in I^o(T)}\hspace{-7pt}\mathcal{O}(\val{T}{w})},
(0,-40)*{
\hspace{-15pt}\bigotimes\limits_{v\in I^{e}(T)\setminus\{u\}}\hspace{-15pt} V(\val{T}{v})\;\;\otimes
\hspace{-15pt}\bigotimes\limits_{w\in I^o(T)\setminus\link(u )}\hspace{-20pt}\mathcal{O}(\val{T}{w})\;\;\otimes \;\;\mathcal{O}(\overline{\star}(u))},
(0,-60)*{
\hspace{-15pt}\bigotimes\limits_{v\in I^{e}(T/{\star}(u))}\hspace{-15pt} V(\val{T}{v})\;\;\;\;\otimes
\hspace{-25pt}\bigotimes\limits_{w\in I^o(T/{\star}(u))\setminus\{[\star(u )]\}}\hspace{-30pt}\mathcal{O}(\val{T}{w})\;\;\;\;\otimes\;\;\;\; \mathcal{O}(r_{u})},
(0,-78)*{P_{t-1}(n)},
\ar(0,-4);(0,-15)_{\bar{g}(\val{T}{u })\otimes\id{}}
\ar(0,-24);(0,-35)_{\cong}^{\text{symmetry}}
\ar(0,-45);(0,-55)_{\id{}\otimes\mathcal{O}(p_{\overline{\star}(u)})}
\ar(0,-65);(0,-75)_{\bar{\psi}_{t-1}^{T/\star(u)}}
\endxy$$
Here $(\bar{\psi}_{t-1}^{T'})_{T'}$ denotes the push-out of $({\psi}_{t-1}^{T'})_{T'}$, i.e. \eqref{monstruo2} for $t-1$, along  \eqref{monstruo}.
\end{lem}

\begin{proof}
The proof is by induction on $t\geq 0$. Notice that there is nothing to check for $t=0,1$. Let $t>1$ and assume everything works up to $t-1$. By the universal property of the source of an iterated $\odot$ product, described in Section \ref{mcm}, we only have to check the following compatibility condition: given two different vertices $u,u'\in I^e(T)$, the following square commutes,
\begin{equation*}\tag{a}
\xy
(0,0)*{U(\val{T}{u})\otimes 
U(\val{T}{u'})\otimes 
\hspace{-17pt}\bigotimes\limits_{v\in I^e(T)\setminus\{u,u'\}}\hspace{-17pt} V(\val{T}{v})\otimes
\hspace{-7pt}\bigotimes\limits_{w\in I^o(T)}\hspace{-7pt}\mathcal{O}(\val{T}{w})},
(20,20)*{V(\val{T}{u})\otimes 
U(\val{T}{u'})\otimes 
\hspace{-17pt}\bigotimes\limits_{v\in I^e(T)\setminus\{u,u'\}}\hspace{-17pt} V(\val{T}{v})\otimes
\hspace{-7pt}\bigotimes\limits_{w\in I^o(T)}\hspace{-7pt}\mathcal{O}(\val{T}{w})},
(20,-20)*{U(\val{T}{u})\otimes 
V(\val{T}{u'})\otimes 
\hspace{-17pt}\bigotimes\limits_{v\in I^e(T)\setminus\{u,u'\}}\hspace{-17pt} V(\val{T}{v})\otimes
\hspace{-7pt}\bigotimes\limits_{w\in I^o(T)}\hspace{-7pt}\mathcal{O}(\val{T}{w})},
(70,1)*{P_{t-1}(n)}
\ar(0,4);(15,16)^{f(\val{T}{u})\otimes\id{}\;}
\ar(0,-4);(15,-16)_{\id{}\otimes f(\val{T}{u'})\otimes\id{}\;}
\ar(50,16);(68,4)^{\psi^T_{t,u'}}
\ar(50,-16);(68,-2)_{\psi^T_{t,u}}
\endxy
\end{equation*}
Here, for simplicity, we omit some symmetry isomorphisms in $\C V$.

Denote $\star(u,u')=\star(u)\cup \star(u')$ and $\link(u,u')=\link(u)\cup \link(u')$. Suppose that $d(u,u')>2$. Then $\star(u)\cap \star(u')=\emptyset$, see Figure \ref{dm2}. Moreover, in this case $t>2$.  
By induction hypotesis, in this case both compositions coincide with
\begin{equation}\label{ultima1}
\xy
(0,0)*{U(\val{T}{u})\otimes 
U(\val{T}{u'})\otimes 
\hspace{-17pt}\bigotimes\limits_{v\in I^e(T)\setminus\{u,u'\}}\hspace{-17pt} V(\val{T}{v})\otimes
\hspace{-7pt}\bigotimes\limits_{w\in I^o(T)}\hspace{-7pt}\mathcal{O}(\val{T}{w})},
(0,-20)*{\mathcal{O}(\val{T}{u})\otimes 
\mathcal{O}(\val{T}{u'})\otimes 
\hspace{-17pt}\bigotimes\limits_{v\in I^e(T)\setminus\{u,u'\}}\hspace{-17pt} V(\val{T}{v})\otimes
\hspace{-7pt}\bigotimes\limits_{w\in I^o(T)}\hspace{-7pt}\mathcal{O}(\val{T}{w})},
(0,-40)*{\bigotimes\limits_{v\in I^e(T)\setminus\{u,u'\}}\hspace{-20pt} V(\val{T}{v})\;\;\;\;\;\otimes
\hspace{-15pt}\bigotimes\limits_{w\in I^o(T)\setminus\link(u,u')}\hspace{-24pt}\mathcal{O}(\val{T}{w})
\;\;\;\;\otimes\;\;\;\;\mathcal{O}(\overline{\star}(u))\otimes 
\mathcal{O}(\overline{\star}(u'))},
(-9,-60)*{\bigotimes\limits_{v\in I^e(T/\star(u,u'))}\hspace{-25pt} V(\val{T}{v})\hspace{35pt}\otimes
\hspace{-35pt}\bigotimes\limits_{w\in I^o(T/\star(u,u'))\setminus\{[\star(u)],[\star(u')]\}}\hspace{-55pt}\mathcal{O}(\val{T}{w})\hspace{15pt}
\otimes\hspace{15pt}\mathcal{O}(r_{u})\otimes 
\mathcal{O}(r_{u'})},
(0,-80)*{P_{t-2}(n)},
(0,-100)*{P_{t-1}(n)}
\ar(-1,-5);(-1,-15)_{\bar{g}(\val{T}{u})\otimes \bar{g}(\val{T}{u'})\otimes\id{}}
\ar(-1,-25);(-1,-35)_\cong^{\text{symmetry}}
\ar(-1,-45);(-1,-55)^{\id{}\otimes\mathcal{O}(p_{\overline{\star}(u)})\otimes \mathcal{O}(p_{\overline{\star}(u')})}
\ar(-1,-65);(-1,-77)_{\bar{\psi}_{t-2}^{T/\star(u,u')}}
\ar(-1,-83);(-1,-97)_{\varphi_{t-1}(n)}
\endxy
\end{equation}
See Figures \ref{dm2} and \ref{dm22}.
\begin{figure}[h]
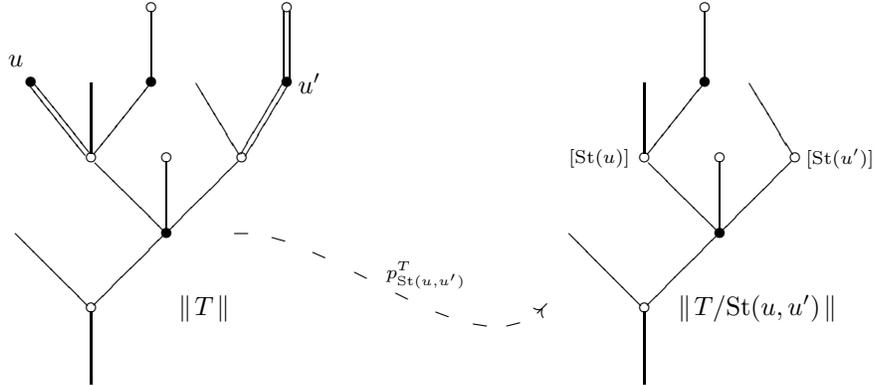

$$
\quad\begin{array}{c}
\xy
(0,0)*{}="a",
(15,10)*{\norm{T}},
(0,10)*-<2pt>{\circ}="b",
(-10,20)*{}="c",
(10,20)*-{\bullet}="d",
(29,40)*{u'},
(0,30)*-<1pt>{\circ}="e",
(10,30)*-<2pt>{\circ}="f",
(20,30)*-<2pt>{\circ}="g",
(-8,40)*-{\bullet}="h",
(-10,43)*{u},
(0,40)*{}="i",
(8,40)*-{\bullet}="j",
(0,50)*{}="k",
(8,50)*-<2pt>{\circ}="l",
(14,40)*{}="m",
(26,40)*-{\bullet}="n",
(26,50)*-<2pt>{\circ}="o",
\ar@{-}"a";"b",
\ar@{-}"b";"c",
\ar@{-}"b";"d",
\ar@{-}"d";"e",
\ar@{-}"d";"f",
\ar@{-}"d";"g",
\ar@{=}"e";"h",
\ar@{-}"e";"i",
\ar@{-}"e";"j",
\ar@{-}"j";"l",
\ar@{-}"g";"m",
\ar@{=}"g";"n",
\ar@{=}"n";"o",
\ar@(r,ld)@{-->}(20,20);(60,10)^{p^{T}_{\star(u,u')}}
\endxy
\end{array}
\begin{array}{c}
\xy
(0,0)*{}="a",
(15,10)*{\norm{T/\star(u,u')}},
(0,10)*-<2pt>{\circ}="b",
(-10,20)*{}="c",
(10,20)*-{\bullet}="d",
(26,30)*{\scriptstyle [\star(u')]},
(0,30)*-<1pt>{\circ}="e",
(10,30)*-<2pt>{\circ}="f",
(20,30)*-<2pt>{\circ}="g",
(-6,30)*{\scriptstyle [\star(u)]},
(0,40)*{}="i",
(8,40)*-{\bullet}="j",
(0,50)*{}="k",
(8,50)*-<2pt>{\circ}="l",
(14,40)*{}="m",
\ar@{-}"a";"b",
\ar@{-}"b";"c",
\ar@{-}"b";"d",
\ar@{-}"d";"e",
\ar@{-}"d";"f",
\ar@{-}"d";"g",
\ar@{-}"e";"i",
\ar@{-}"e";"j",
\ar@{-}"j";"l",
\ar@{-}"g";"m",
\endxy
\end{array}
$$
\caption{A planted planar tree $T$ with leaves  in even levels and two even inner vertices $u$ and $u'$ with $d(u,u')>2$. The disconnected subcomplex $\star(u,u')\subset T$ is in double lines. We illustrate the morphism~$p^{T}_{\star(u,u')}$. 
}
\label{dm2}
\end{figure}

\begin{figure}[h]
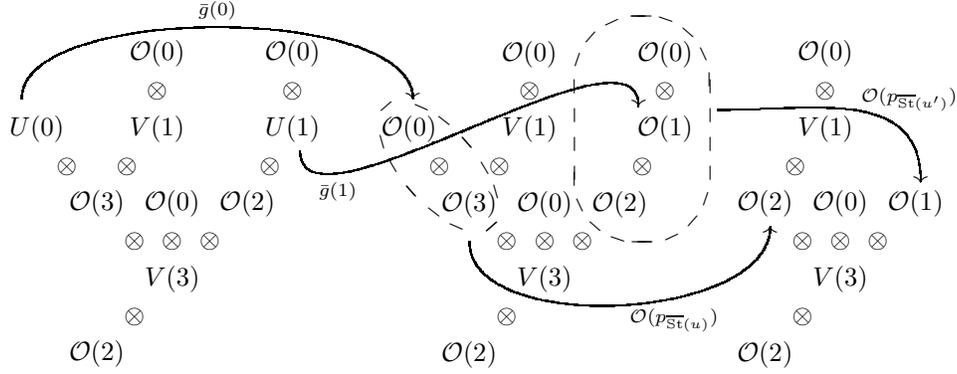

$$
\begin{array}{c}
\xy
(0,0)*{}="a",
(0,10)*-<2pt>{\mathcal{O}(2)}="b",
(-10,20)*{}="c",
(10,20)*-{ V(3)}="d",
(0,30)*-<1pt>{\mathcal{O}(3)}="e",
(10,30)*-<2pt>{\mathcal{O}(0)}="f",
(20,30)*-<2pt>{\mathcal{O}(2)}="g",
(-8,40)*-{ U(0)}="h",
(0,40)*{}="i",
(8,40)*-{ V(1)}="j",
(0,50)*{}="k",
(8,50)*-<2pt>{\mathcal{O}(0)}="l",
(14,40)*{}="m",
(26,40)*-{U(1)}="n",
(26,50)*-<2pt>{\mathcal{O}(0)}="o"
\ar@{}"a";"b"
\ar@{}"b";"c"
\ar@{}"b";"d"|{\displaystyle\otimes}
\ar@{}"d";"e"|{\displaystyle\otimes}
\ar@{}"d";"f"|{\displaystyle\otimes}
\ar@{}"d";"g"|{\displaystyle\otimes}
\ar@{}"e";"h"|{\displaystyle\otimes}
\ar@{}"e";"i"
\ar@{}"e";"j"|{\displaystyle\otimes}
\ar@{}"j";"l"|{\displaystyle\otimes}
\ar@{}"g";"m"
\ar@{}"g";"n"|{\displaystyle\otimes}
\ar@{}"n";"o"|{\displaystyle\otimes}
\ar@(d,u)(27,37);(72,43)_<(.2){\bar{g}(1)}
\ar@(u,u)(-10,44);(42,44)^-{\bar{g}(0)}
\endxy
\end{array}\hspace{-109pt}
\begin{array}{c}
\vspace{-3pt}
\\
\xy
(0,0)*{}="a",
(0,10)*-<2pt>{\mathcal{O}(2)}="b",
(-10,20)*{}="c",
(10,20)*-{ V(3)}="d",
(0,30)*-<1pt>{\mathcal{O}(3)}="e",
(10,30)*-<2pt>{\mathcal{O}(0)}="f",
(20,30)*-<2pt>{\mathcal{O}(2)}="g",
(-8,40)*-{ \mathcal{O}(0)}="h",
(0,40)*{}="i",
(8,40)*-{ V(1)}="j",
(0,50)*{}="k",
(8,50)*-<2pt>{\mathcal{O}(0)}="l",
(14,40)*{}="m",
(26,40)*-{\mathcal{O}(1)}="n",
(26,50)*-<2pt>{\mathcal{O}(0)}="o",
(23,40)*=<1.8cm,3cm>{}*\frm<30pt>{--},
(-11,43);(-4,35),{\ellipse<,14pt>{--}}
\ar@{}"a";"b"
\ar@{}"b";"c"
\ar@{}"b";"d"|{\displaystyle\otimes}
\ar@{}"d";"e"|{\displaystyle\otimes}
\ar@{}"d";"f"|{\displaystyle\otimes}
\ar@{}"d";"g"|{\displaystyle\otimes}
\ar@{}"e";"h"|{\displaystyle\otimes}
\ar@{}"e";"i"
\ar@{}"e";"j"|{\displaystyle\otimes}
\ar@{}"j";"l"|{\displaystyle\otimes}
\ar@{}"g";"m"
\ar@{}"g";"n"|{\displaystyle\otimes}
\ar@{}"n";"o"|{\displaystyle\otimes}
\ar@(r,u)(33,42.5);(60,33)^<(.64){\mathcal{O}(p_{\overline{\star}(u')})}
\ar@(d,d)(0,25);(40,27)_<(.6){\mathcal{O}(p_{\overline{\star}(u)})}
\endxy
\end{array}
\hspace{-114pt}
\begin{array}{c}
\vspace{7pt}
\\
\xy
(0,0)*{}="a",
(0,10)*-<2pt>{\mathcal{O}(2)}="b",
(-10,20)*{}="c",
(10,20)*-{ V(3)}="d",
(0,30)*-<1pt>{\mathcal{O}(2)}="e",
(10,30)*-<2pt>{\mathcal{O}(0)}="f",
(20,30)*-<2pt>{\mathcal{O}(1)}="g",
(-8,40)*-{}="h",
(0,40)*{}="i",
(8,40)*-{ V(1)}="j",
(0,50)*{}="k",
(8,50)*-<2pt>{\mathcal{O}(0)}="l",
(14,40)*{}="m",
(26,40)*-{}="n",
(26,50)*-<2pt>{}="o"
\ar@{}"a";"b"
\ar@{}"b";"c"
\ar@{}"b";"d"|{\displaystyle\otimes}
\ar@{}"d";"e"|{\displaystyle\otimes}
\ar@{}"d";"f"|{\displaystyle\otimes}
\ar@{}"d";"g"|{\displaystyle\otimes}
\ar@{}"e";"i"
\ar@{}"e";"j"|{\displaystyle\otimes}
\ar@{}"j";"l"|{\displaystyle\otimes}
\ar@{}"g";"m"
\endxy
\end{array}
$$
\caption{A sketch of \eqref{ultima1} for the  planted planar tree $T$ with leaves  in even levels and the two even inner vertices $u$ and $u'$ in Figure \ref{dm2}.}
\label{dm22}
\end{figure}

Suppose now that $d(u,u')=2$. Then the subcomplex $\star(u,u')\subset T$ is connected. Both factors share the unique vertex which is one step away from both $u$ and $u'$, see Figure \ref{configs}. 
\begin{figure}
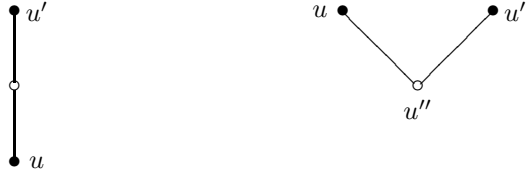

$$
\xy
(0,0)*-{\bullet}="a",
(0,-10)*-<1pt>{\circ}="b",
(0,-20)*-{\bullet}="c",
(3,0)*{u'},
(3,-20)*{u},
\ar@{-}"a";"b"
\ar@{-}"b";"c"
\endxy\qquad \qquad \qquad \qquad \qquad
\xy
(10,0)*-{\bullet}="a",
(0,-10)*-<2pt>{\circ}="b",
(-10,0)*-{\bullet}="c",
(13,0)*{u'},
(-13,0)*{u},
(0,-13)*{u''},
\ar@{-}"a";"b"
\ar@{-}"b";"c"
\endxy
$$
\caption{The only two possible relative positions of $u$ and $u'$, $u<u'$, within the planted planar tree with leaves $T$ if $d(u,u')=2$.}\label{configs}
\end{figure}

Let $T'\subset T$ be in this case the planted planar tree with leaves whose inner part is $\star(u,u')$, the root edge is the outgoing edge of the minimun vertex $u''\in\link(u,u')$, the leaves are the incoming edges of the vertices in $\link(u,u')$ not containing $u$ or~$u'$, and the planar order is the restriction of the planar order in~$T$. This planted planar tree 
has $m$ leaves, where
$$m=r_u+r_{u'}-1,$$ 
when the relative position of $u$ and $u'$ is as in the first diagram of Figure \ref{configs}, see also Figure \ref{tprima1}. If the relative position is as in the second diagram of Figure \ref{configs}, then 
$$m=\val{T}{u''}+r_u+r_{u'}-2,$$
see Figure \ref{tprima2}.


In this case, by induction hypothesis, the two possible compositions in the square~(a) coincide with the following morphism, see Figure \ref{tprima3} for an illustration,
\begin{equation}\label{ultima2}
\xy
(0,0)*{U(\val{T}{u})\otimes 
U(\val{T}{u'})\otimes 
\hspace{-17pt}\bigotimes\limits_{v\in I^e(T)\setminus\{u,u'\}}\hspace{-17pt} V(\val{T}{v})\otimes
\hspace{-7pt}\bigotimes\limits_{w\in I^o(T)}\hspace{-7pt}\mathcal{O}(\val{T}{w})},
(0,-20)*{\mathcal{O}(\val{T}{u})\otimes 
\mathcal{O}(\val{T}{u'})\otimes 
\hspace{-17pt}\bigotimes\limits_{v\in I^e(T)\setminus\{u,u'\}}\hspace{-17pt} V(\val{T}{v})\otimes
\hspace{-7pt}\bigotimes\limits_{w\in I^o(T)}\hspace{-7pt}\mathcal{O}(\val{T}{w})},
(-4,-40)*{\bigotimes\limits_{v\in I^e(T)\setminus\{u,u'\}}\hspace{-20pt} V(\val{T}{v})\;\;\;\;\otimes
\hspace{-18pt}\bigotimes\limits_{w\in I^o(T)\setminus\link(u,u')}\hspace{-27pt}\mathcal{O}(\val{T}{w})
\;\;\otimes\;\;\mathcal{O}(T')},
(-7,-60)*{\bigotimes\limits_{v\in I^e(T/\star(u,u'))}\hspace{-20pt} V(\val{T}{v})\;\;\;\;\;\;\;\;\otimes
\hspace{-25pt}\bigotimes\limits_{w\in I^o(T/\star(u,u'))\setminus\{[\star(u,u')]\}}\hspace{-45pt}\mathcal{O}(\val{T}{w})
\;\;\;\;\otimes\;\;\;\;\mathcal{O}(m)},
(0,-80)*{P_{t-2}(n)},
(0,-100)*{P_{t-1}(n)}
\ar(-1,-5);(-1,-15)_{\bar{g}(\val{T}{u})\otimes \bar{g}(\val{T}{u'})\otimes\id{}}
\ar(-1,-25);(-1,-35)_\cong^{\text{symmetry}}
\ar(-1,-45);(-1,-55)^{\id{}\otimes\mathcal{O}(p_{T'})}
\ar(-1,-65);(-1,-77)_{\bar{\psi}_{t-2}^{T/\star(u,u')}}^{\text{ or the identity if }t=2}
\ar(-1,-84);(-1,-97)_{\varphi_{t-1}(n)}
\endxy
\end{equation}

\begin{figure}
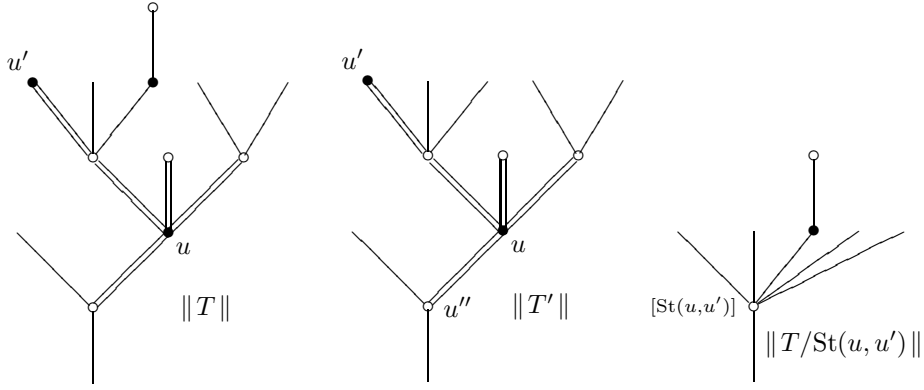

$$
\begin{array}{c}
\xy
(0,0)*{}="a",
(0,10)*-<2pt>{\circ}="b",
(-10,20)*{}="c",
(10,20)*-{\bullet}="d",
(12,18)*{u},
(0,30)*-<1pt>{\circ}="e",
(10,30)*-<2pt>{\circ}="f",
(20,30)*-<2pt>{\circ}="g",
(-8,40)*-{\bullet}="h",
(-10,43)*{u'},
(15,10)*{\norm{T}},
(0,40)*{}="i",
(8,40)*-{\bullet}="j",
(0,50)*{}="k",
(8,50)*-<2pt>{\circ}="l",
(14,40)*{}="m",
(26,40)*{}="n",
\ar@{-}"a";"b",
\ar@{-}"b";"c",
\ar@{=}"b";"d",
\ar@{=}"d";"e",
\ar@{=}"d";"f",
\ar@{=}"d";"g",
\ar@{=}"e";"h",
\ar@{-}"e";"i",
\ar@{-}"e";"j",
\ar@{-}"j";"l",
\ar@{-}"g";"m",
\ar@{-}"g";"n",
\endxy
\end{array}
\quad\begin{array}{c}
\xy
(0,0)*{}="a",
(0,10)*-<2pt>{\circ}="b",
(4,10)*{u''},
(-10,20)*{}="c",
(10,20)*-{\bullet}="d",
(12,18)*{u},
(0,30)*-<1pt>{\circ}="e",
(10,30)*-<2pt>{\circ}="f",
(20,30)*-<2pt>{\circ}="g",
(-8,40)*-{\bullet}="h",
(-10,43)*{u'},
(15,10)*{\norm{T'}},
(0,40)*{}="i",
(8,40)*{}="j",
(0,50)*{}="k",
(14,40)*{}="m",
(26,40)*{}="n",
\ar@{-}"a";"b",
\ar@{-}"b";"c",
\ar@{=}"b";"d",
\ar@{=}"d";"e",
\ar@{=}"d";"f",
\ar@{=}"d";"g",
\ar@{=}"e";"h",
\ar@{-}"e";"i",
\ar@{-}"e";"j",
\ar@{-}"g";"m",
\ar@{-}"g";"n",
\endxy
\end{array}
\begin{array}{c}
\xy
(0,0)*{}="a",
(0,10)*-<1pt>{\circ}="b",
(-8,10)*{\scriptstyle [\star(u,u')]},
(-10,20)*{}="c",
(0,30)*-<1pt>{}="e",
(20,30)*-<2pt>{}="g",
(12,5)*{\norm{T/\star(u,u')}},
(0,20)*{}="i",
(8,20)*-{\bullet}="j",
(0,50)*{}="k",
(8,30)*-<2pt>{\circ}="l",
(14,20)*{}="m",
(20,20)*{}="n",
\ar@{-}"a";"b",
\ar@{-}"b";"c",
\ar@{-}"b";"i",
\ar@{-}"b";"j",
\ar@{-}"j";"l",
\ar@{-}"b";"m",
\ar@{-}"b";"n",
\endxy
\end{array}
$$
\caption{An example of the planted planar tree with leaves~$T'$ for the relative position of the vertices $u$ and $u'$ as in the first digram of Figure \ref{configs}. The subcomplex $\star(u,u')\subset T$ is in double lines. We also depict $T/\star(u,u')$.}
\label{tprima1}
\end{figure}

\begin{figure}
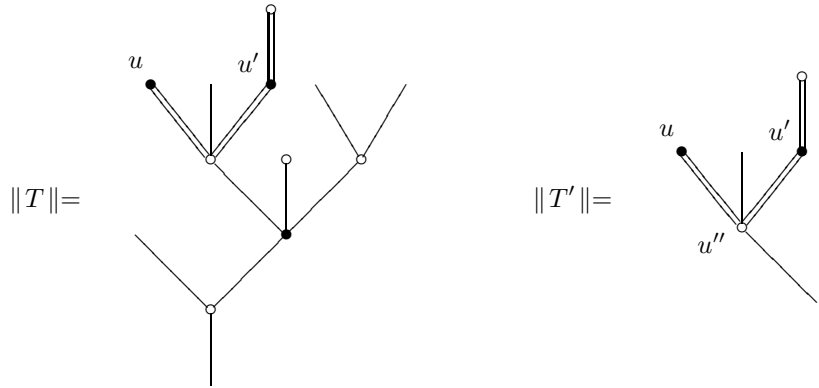

$$
\norm{T}=\quad\begin{array}{c}
\xy
(0,0)*{}="a",
(0,10)*-<2pt>{\circ}="b",
(-10,20)*{}="c",
(10,20)*-{\bullet}="d",
(12,18)*{},
(0,30)*-<1pt>{\circ}="e",
(10,30)*-<2pt>{\circ}="f",
(20,30)*-<2pt>{\circ}="g",
(-8,40)*-{\bullet}="h",
(-10,43)*{u},
(0,40)*{}="i",
(8,40)*-{\bullet}="j",
(5,43)*{u'},
(0,50)*{}="k",
(8,50)*-<2pt>{\circ}="l",
(14,40)*{}="m",
(26,40)*{}="n",
\ar@{-}"a";"b",
\ar@{-}"b";"c",
\ar@{-}"b";"d",
\ar@{-}"d";"e",
\ar@{-}"d";"f",
\ar@{-}"d";"g",
\ar@{=}"e";"h",
\ar@{-}"e";"i",
\ar@{=}"e";"j",
\ar@{=}"j";"l",
\ar@{-}"g";"m",
\ar@{-}"g";"n",
\endxy
\end{array}\qquad\qquad
\norm{T'}=\quad\begin{array}{c}
\xy
(10,20)*{}="d",
(12,18)*{},
(0,30)*-<1pt>{\circ}="e",
(-4,28)*{u''},
(-8,40)*-{\bullet}="h",
(-10,43)*{u},
(0,40)*{}="i",
(8,40)*-{\bullet}="j",
(5,43)*{u'},
(0,50)*{}="k",
(8,50)*-<2pt>{\circ}="l",
(14,40)*{}="m",
(26,40)*{}="n",
\ar@{-}"d";"e",
\ar@{=}"e";"h",
\ar@{-}"e";"i",
\ar@{=}"e";"j",
\ar@{=}"j";"l",
\endxy
\end{array}
$$
\caption{An example of the planted planar tree with leaves $T'$ for the relative position of the vertices $u$ and $u'$ as in the second digram of Figure \ref{configs}. The subcomplex $\star(u,u')\subset T$ is in double lines.}
\label{tprima2}
\end{figure}

\begin{figure}[h]
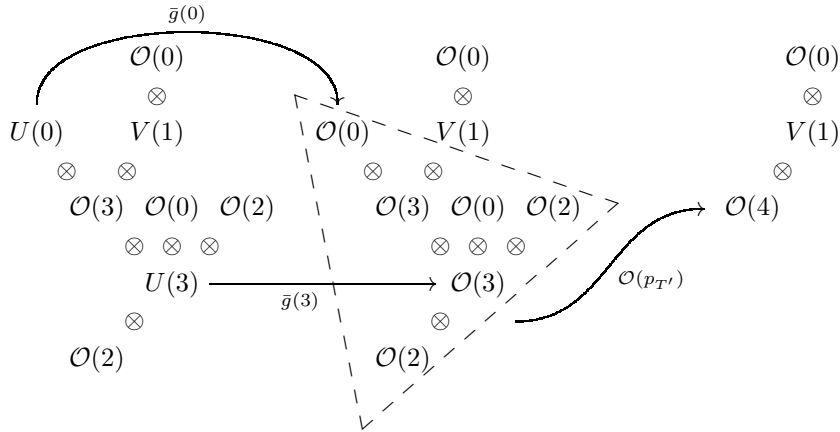

$$
\begin{array}{c}
\xy
(0,0)*{}="a",
(0,10)*-<2pt>{\mathcal{O}(2)}="b",
(-10,20)*{}="c",
(10,20)*-{ U(3)}="d",
(0,30)*-<1pt>{\mathcal{O}(3)}="e",
(10,30)*-<2pt>{\mathcal{O}(0)}="f",
(20,30)*-<2pt>{\mathcal{O}(2)}="g",
(-8,40)*-{ U(0)}="h",
(0,40)*{}="i",
(8,40)*-{ V(1)}="j",
(0,50)*{}="k",
(8,50)*-<2pt>{\mathcal{O}(0)}="l",
(14,40)*{}="m",
(26,40)*{}="n"
\ar@{}"a";"b"
\ar@{}"b";"c"
\ar@{}"b";"d"|{\displaystyle\otimes}
\ar@{}"d";"e"|{\displaystyle\otimes}
\ar@{}"d";"f"|{\displaystyle\otimes}
\ar@{}"d";"g"|{\displaystyle\otimes}
\ar@{}"e";"h"|{\displaystyle\otimes}
\ar@{}"e";"i"
\ar@{}"e";"j"|{\displaystyle\otimes}
\ar@{}"j";"l"|{\displaystyle\otimes}
\ar@{}"g";"m"
\ar@{}"g";"n"
\ar(15,20);(45,20)_<(.4){\bar{g}(3)}
\ar@(u,u)(-8,44);(32,44)^{\bar{g}(0)}
\endxy
\end{array}\hspace{-63pt}
\begin{array}{c}
\vspace{7pt}
\\
\xy
(0,0)*{}="a",
(0,10)*-<2pt>{\mathcal{O}(2)}="b",
(-10,20)*{}="c",
(10,20)*-{ \mathcal{O}(3)}="d",
(0,30)*-<1pt>{\mathcal{O}(3)}="e",
(10,30)*-<2pt>{\mathcal{O}(0)}="f",
(20,30)*-<2pt>{\mathcal{O}(2)}="g",
(-8,40)*-{ \mathcal{O}(0)}="h",
(0,40)*{}="i",
(8,40)*-{ V(1)}="j",
(0,50)*{}="k",
(8,50)*-<2pt>{\mathcal{O}(0)}="l",
(14,40)*{}="m",
(26,40)*{}="n",
(3,25.5),{\xypolygon3{~={-32}~:{(19,18):}~>{{--}}}}
\ar@{}"a";"b"
\ar@{}"b";"c"
\ar@{}"b";"d"|{\displaystyle\otimes}
\ar@{}"d";"e"|{\displaystyle\otimes}
\ar@{}"d";"f"|{\displaystyle\otimes}
\ar@{}"d";"g"|{\displaystyle\otimes}
\ar@{}"e";"h"|{\displaystyle\otimes}
\ar@{}"e";"i"
\ar@{}"e";"j"|{\displaystyle\otimes}
\ar@{}"j";"l"|{\displaystyle\otimes}
\ar@{}"g";"m"
\ar@{}"g";"n"
\ar@(r,l)(15,15);(40,30)_{\mathcal{O}(p_{T'})}
\endxy
\end{array}\hspace{-20pt}
\begin{array}{c}
\vspace{7pt}
\\
\xy
(0,0)*{}="a",
(0,10)*-<2pt>{}="b",
(-10,20)*{}="c",
(10,20)*-{}="d",
(0,30)*-<1pt>{\mathcal{O}(4)}="e",
(10,30)*-<2pt>{}="f",
(20,30)*-<2pt>{}="g",
(-8,40)*-{}="h",
(0,40)*{}="i",
(8,40)*-{ V(1)}="j",
(0,50)*{}="k",
(8,50)*-<2pt>{\mathcal{O}(0)}="l",
(14,40)*{}="m",
(26,40)*{}="n"
\ar@{}"e";"j"|{\displaystyle\otimes}
\ar@{}"j";"l"|{\displaystyle\otimes}
\endxy
\end{array}
$$
\caption{An illustration of \eqref{ultima2}  for~$T$, $u$ and $u'$ as in Figure \ref{tprima1}.}
\label{tprima3}
\end{figure}

\end{proof}

In the following lemma, we inductively construct an operad structure on the
colimit of the sequence defined in the former. Roughly speaking, we need
to define the multiplications $\circ_{i}$ of elements attached to $\mathcal{O}$ through planted planar
trees with leaves concentrated in even levels $T$ and $T'$, where $i$ is less than or equal to
the number of leaves of $T$. Consider for instance 
\begin{center}
\includegraphics{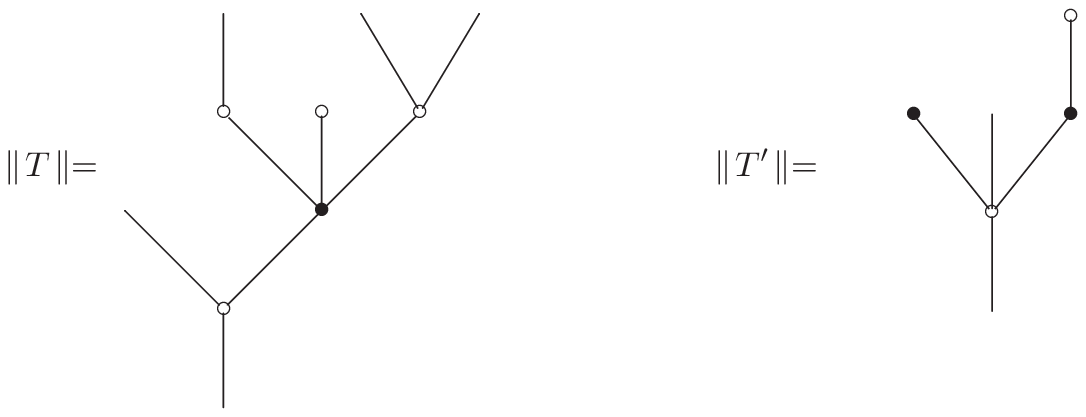}
\end{center}
In this case, in order to define $\circ_{2}$ we take $T\circ_{2}T'$ and the following associated tensor
product of objects in $V$ and $\mathcal{O}$,
\begin{center}
\includegraphics{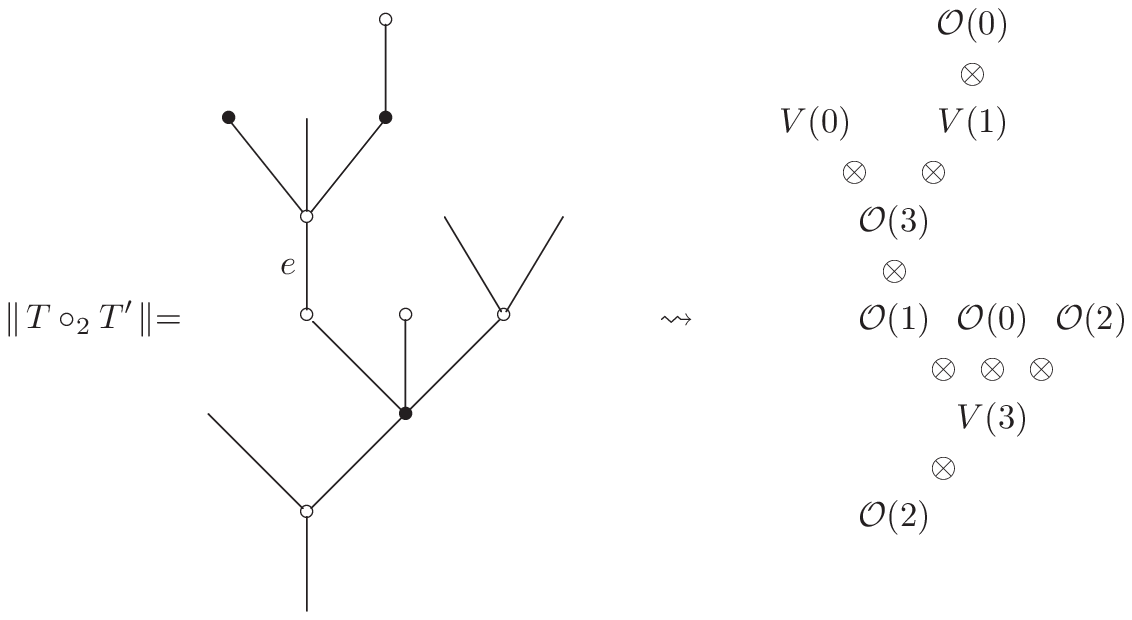}
\end{center}
Notice that this object in $\C V$ is just the tensor product of the objects associated to $T$
and $T'$. Then we contract the root edge $e$ of $T'$, which is identified with the second
leaf of $T$, and we get a planted planar tree with leaves in even levels $(T\circ_{2}T')/e$ .
\begin{center}
\includegraphics{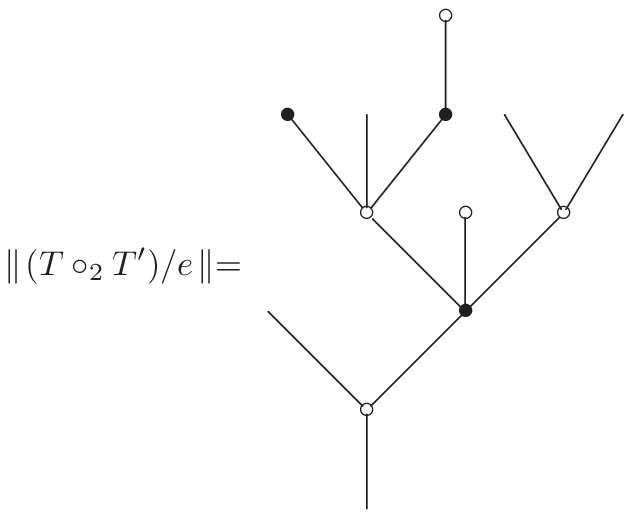}
\end{center}
This can be algebraically mimicked on the associated tensor product by means of
multiplication in $\mathcal{O}$ according to the local structure of $T\circ_{2}T'$ in a neighbourhood
of $e$, e.g. $e$ is the second leaf of $T$ but it is the first (and the only) one attached to its inner vertex in $T$,
\begin{center}
\includegraphics{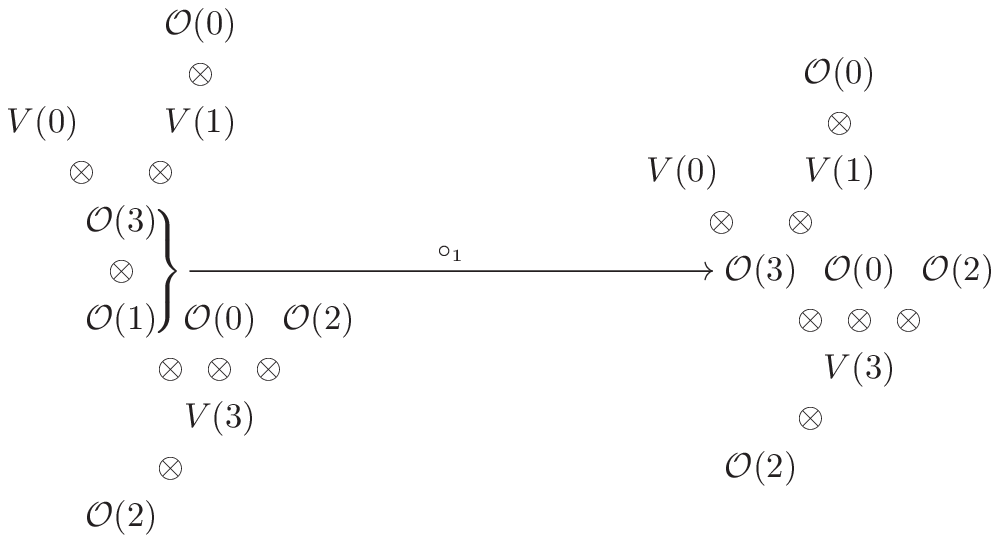}
\end{center}
We can define the $\circ_{2}$ multiplication of elements associated to $T$ and $T'$ via this
morphism and the attaching of elements associated to  $(T\circ_{2}T')/e$. We now formalize
this idea.

\begin{lem}\label{pond}
There are unique morphisms in $\C{V}$, $n,s,t\geq 0$, $1\leq i\leq m$,
$$c_{i}^{s,t}(m,n)\colon P_{s}(m)\otimes P_{t}(n)\To P_{s+t}(m+n-1),$$
such that $$c_i^{0,0}=\circ_i\colon\mathcal{O}(m)\otimes \mathcal{O}(n)\To \mathcal{O}(m+n-1)$$
is  the operad composition law, 
\begin{align*}
c_{i}^{s,t}(m,n)(\varphi_{s}(m)\otimes\id{})&=\varphi_{s+t}(m+n-1)c_{i}^{s-1,t}(m,n),\\
c_{i}^{s,t}(m,n)(\id{}\otimes\varphi_{t}(n))&=\varphi_{s+t}(m+n-1)c_{i}^{s,t-1}(m,n),
\end{align*}
and given planted planar trees $T$ and $T'$ with leaves concentrated in even levels, $\card L(T)=m$, $\card L(T')=n$, $\card I^{e}(T)=s$, and $\card I^{e}(T')=t$, 
if
$u'\in I^o(T')$ is the unique level $1$ vertex, $u\in I^{o}(T)$ belongs to the $i^{\text{th}}$ leaf edge (with respect to the path order),  the $i^{\text{th}}$ leaf edge occupies the  $k^{\text{th}}$ place among all incomming edges of $u$, and $e=\{u,u'\}\in E(T\circ_{i} T')$, then the the morphism
$c_{i}^{s,t}(m,n)(\bar{\psi}_{s}^{T}\otimes \bar{\psi}_{t}^{T'})$ coincides with the following morphism, that we call $d^{s,t}_{i}(T,T')$,
$$\xy
(0,0)*{\bigotimes\limits_{v\in I^{e}(T)}\hspace{-8pt} V(\val{T}{v})\otimes
\hspace{-7pt}\bigotimes\limits_{w\in I^o(T)}\hspace{-7pt}\mathcal{O}(\val{T}{w})
\otimes\hspace{-10pt} 
\bigotimes\limits_{v'\in I^{e}(T')}\hspace{-10pt} V(\val{T'}{v'})\otimes
\hspace{-10pt}\bigotimes\limits_{w'\in I^o(T')}\hspace{-10pt}\mathcal{O}(\val{T'}{w'})},
(-1,-20)*{
\mathcal{O}(\val{T}{u})\otimes \mathcal{O}(\val{T'}{u'})
\;\;\;\;\otimes\hspace{-10pt}
\bigotimes\limits_{v\in I^{e}(T)\cup I^{e}(T')}\hspace{-20pt} V(\val{}{v})\;\;\;\;\;\;\;\otimes
\hspace{-25pt}\bigotimes\limits_{w\in (I^o(T)\setminus\{u\})\cup (I^o(T')\setminus\{u'\})}\hspace{-45pt}\mathcal{O}(\val{}{w})
},
(-1,-41)*{
\mathcal{O}(\underbrace{\val{T}{u}+\val{T'}{u'}-1}_{=\;\val{(T\circ_{i}T')/e}{[e]}})
\;\;\;\otimes\hspace{-15pt}
\bigotimes\limits_{v\in I^{e}((T\circ_{i}T')/e)}\hspace{-22pt} V(\val{}{v})\;\;\;\;\;\otimes
\hspace{-15pt}\bigotimes\limits_{w\in I^o((T\circ_{i}T')/e)\setminus\{[e]\}}\hspace{-33pt}\mathcal{O}(\val{}{w})
},
(0,-60)*{P_{s+t}(m+n-1)}
\ar(0,-5);(0,-15)_{\cong}^{\text{symmetry}}
\ar(0,-25);(0,-35)_{\circ_{k}\otimes\id{}}
\ar(0,-45);(0,-57)_{\bar{\psi}_{s+t}^{(T\circ_i T')/e}}
\endxy$$
Here we use the convention that $\bar{\psi}_{0}^{T}=\id{\mathcal{O}(m)}$ and  $\bar{\psi}_{0}^{T'}=\id{\mathcal{O}(n)}$.
\end{lem}

\begin{proof}
The map $c_{i}^{s,t}(m,n)$ is defined from $c_{i}^{s-1,t}(m,n)$, $c_{i}^{s,t-1}(m,n)$ and $d_{i}^{s,t}(T,T')$ by using the universal property of the push-out definition of
$P_{s}(m)\otimes P_{t}(n)$ arising from Lemmas \ref{podot} and \ref{pind}, by induction on $(s,t)\in \mathbb{N}\times \mathbb{N}$, $\mathbb{N}=\{0,1,2\dots\}$, with respect to the graded lexicographic order, 
$$(s,t)\leq (s',t')\Leftrightarrow\left\{
\begin{array}{l}
\text{either }s+t<s'+t',\\
\text{or }s+t = s'+t'\text{ and } s\leq s'. 
\end{array}
\right.$$

There is nothing to check for the first three elements $(0,0)$, $(0,1)$, $(1,0)$. Assume that everything holds up to the predecessor of $(s,t)$ with $s+t>1$. We have to show that, for any $x\in I^e(T)$ and $x'\in I^e(T')$, the following compatibility conditions hold:
\begin{align*}
\tag{a} d^{s,t}_i(T,T')(f(\val{T}{x})\otimes\id{})&=\varphi_{s+t}(m+n-1)c_{i}^{s-1,t}(m,n)(\psi_{s,x}^T\otimes\bar{\psi}_t^{T'}),\\
\tag{b}
d^{s,t}_i(T,T')(f(\val{T'}{x'})\otimes\id{})&=\varphi_{s+t}(m+n-1)c_{i}^{s,t-1}(m,n)(\bar{\psi}_s^T\otimes \psi_{t,x'}^{T'}).
\end{align*}
Since (a) and (b) are very similar to each other, we here just check (a). 
We must distinguish two cases: $\{x,u\}\in E(T)$ and $\{x,u\}\notin E(T)$, see Figure \ref{xu1}.

\begin{figure}
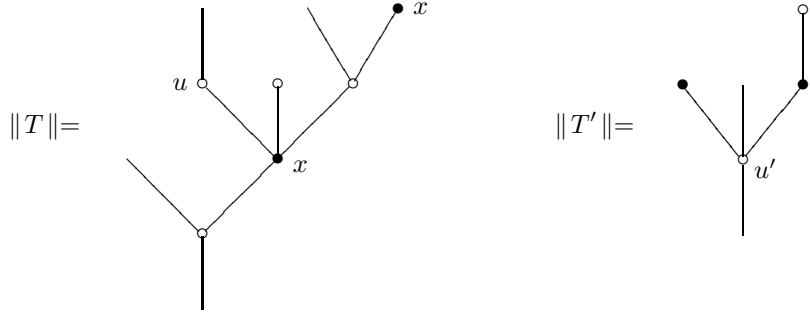

$$
\norm{T}=\quad\begin{array}{c}
\xy
(0,0)*{}="a",
(0,10)*-<2pt>{\circ}="b",
(-10,20)*{}="c",
(10,20)*-{\bullet}="d",
(13,19)*{x},
(29,40)*{x},
(0,30)*-<1pt>{\circ}="e",
(-3,30)*{u},
(10,30)*-<2pt>{\circ}="f",
(20,30)*-<2pt>{\circ}="g",
(-8,40)*-{}="h",
(0,40)*{}="i",
(8,40)*-{}="j",
(0,50)*{}="k",
(14,40)*{}="m",
(26,40)*-{\bullet}="n",
\ar@{-}"a";"b",
\ar@{-}"b";"c",
\ar@{-}"b";"d",
\ar@{-}"d";"e",
\ar@{-}"d";"f",
\ar@{-}"d";"g",
\ar@{-}"e";"i",
\ar@{-}"g";"m",
\ar@{-}"g";"n",
\endxy
\end{array}\qquad\qquad\norm{T'}=\quad
\begin{array}{c}
\xy
(0,20)*{}="d",
(0,30)*-<1pt>{\circ}="e",
(3,29)*{u'},
(-8,40)*-{\bullet}="h",
(0,40)*{}="i",
(8,40)*-{\bullet}="j",
(0,50)*{}="k",
(8,50)*-<2pt>{\circ}="l",
(14,40)*{}="m",
(26,40)*{}="n",
\ar@{-}"d";"e",
\ar@{-}"e";"h",
\ar@{-}"e";"i",
\ar@{-}"e";"j",
\ar@{-}"j";"l",
\endxy
\end{array}$$
\caption{For the trees $T$ and $T'$ and $i=2$ we depict $u$, $u'$ and two possible choices of $x$, one with $\{x,u\}\in E(T)$ and the other one with $\{x,u\}\notin E(T)$.}
\label{xu1}
\end{figure}

Suppose $\{x,u\}\notin E(T)$. Then $u\notin \link(x)$. Using the definition of $d_i^{s,t}(T,T')$ in the statement of this lemma and the definition of $\bar{\psi}_{s+t}^{(T\circ_iT')/e}$ in Lemma \ref{pind} we deduce that, in this case, the left hand side of (a) is the following composite morphism, see~Figure~\ref{xu11},
\begin{equation}\label{ultima3}
\xy
(0,0)*{U(\val{T}{x})\otimes\hspace{-17pt}\bigotimes\limits_{v\in I^{e}(T)\setminus\{x\}}\hspace{-17pt} V(\val{T}{v})\otimes
\hspace{-9pt}\bigotimes\limits_{w\in I^o(T)}\hspace{-9pt}\mathcal{O}(\val{T}{w})
\otimes\hspace{-12pt} 
\bigotimes\limits_{v'\in I^{e}(T')}\hspace{-12pt} V(\val{T'}{v'})\otimes
\hspace{-12pt}\bigotimes\limits_{w'\in I^o(T')}\hspace{-12pt}\mathcal{O}(\val{T'}{w'})},
(0,-20)*{
\mathcal{O}(\val{T}{x})\otimes\hspace{-17pt}\bigotimes\limits_{v\in I^{e}(T)\setminus\{x\}}\hspace{-17pt} V(\val{T}{v})\otimes
\hspace{-9pt}\bigotimes\limits_{w\in I^o(T)}\hspace{-9pt}\mathcal{O}(\val{T}{w})
\otimes\hspace{-12pt} 
\bigotimes\limits_{v'\in I^{e}(T')}\hspace{-12pt} V(\val{T'}{v'})\otimes
\hspace{-12pt}\bigotimes\limits_{w'\in I^o(T')}\hspace{-12pt}\mathcal{O}(\val{T'}{w'})
},
(-3,-40)*{
\begin{array}{c}
\hspace{39pt}\mathcal{O}(\overline{\star}(x))\otimes\mathcal{O}(\val{T}{u})\otimes\mathcal{O}(\val{T'}{u'})\vspace{-10pt}\\{}\\
\otimes\;\;\;\;\hspace{-17pt}\bigotimes\limits_{v\in I^{e}(T)\setminus\{x\}}\hspace{-17pt} V(\val{T}{v})\;\;\;\;\otimes
\hspace{-18pt}\bigotimes\limits_{w\in I^o(T)\setminus(\link(x)\cup\{u\})}\hspace{-30pt}\mathcal{O}(\val{T}{w})
\;\;\;\;\otimes\hspace{-0pt} 
\bigotimes\limits_{v'\in I^{e}(T')}\hspace{-10pt} V(\val{T'}{v'})\;\;\;\;\otimes
\hspace{-10pt}\bigotimes\limits_{w'\in I^o(T')\setminus\{u'\}}\hspace{-20pt}\mathcal{O}(\val{T'}{w'})
\end{array}
},
(-3,-67)*{
\begin{array}{c}
\hspace{53pt}\mathcal{O}(r_x)\otimes\mathcal{O}(\val{T}{u}+\val{T'}{u'}-1)\vspace{-10pt}\\{}\\
\otimes\;\;\;\;\hspace{-17pt}\bigotimes\limits_{v\in I^{e}(T)\setminus\{x\}}\hspace{-17pt} V(\val{T}{v})\;\;\;\;\otimes
\hspace{-18pt}\bigotimes\limits_{w\in I^o(T)\setminus(\link(x)\cup\{u\})}\hspace{-30pt}\mathcal{O}(\val{T}{w})
\;\;\;\;\otimes\hspace{-0pt} 
\bigotimes\limits_{v'\in I^{e}(T')}\hspace{-10pt} V(\val{T'}{v'})\;\;\;\;\otimes
\hspace{-10pt}\bigotimes\limits_{w'\in I^o(T')\setminus\{u'\}}\hspace{-20pt}\mathcal{O}(\val{T'}{w'})
\end{array}
},
(-6,-89)*{
\bigotimes\limits_{v\in I^{e}(((T/\star(x))\circ_i T')/e)}\hspace{-30pt} V(\val{((T/\star(x))\circ_i T')/e}{v})\;\;\;\;\;\;\;\;\otimes
\hspace{-15pt}\bigotimes\limits_{w\in I^o(((T/\star(x))\circ_i T')/e)}\hspace{-30pt}\mathcal{O}(\val{((T/\star(x))\circ_i T')/e}{w})
},
(0,-105)*{P_{s+t-1}(m+n-1)},
(0,-120)*{P_{s+t}(m+n-1)}
\ar(0,-2);(0,-15)_{\bar{g}(\val{T}{x})\otimes\id{}}
\ar(0,-22);(0,-32)_{\cong}^{\text{symmetry}}
\ar(0,-46);(0,-59)_{\mathcal{O}(p_{\overline{\star}(x)})\otimes\circ_k  \otimes\id{}}
\ar(0,-73);(0,-85)_{\cong}^{\text{symmetry}}
\ar(0,-93);(0,-103)_{\bar{\psi}_{s+t-1}^{((T/\star(x))\circ_i T')/e}}
\ar(0,-108);(0,-117)_{\varphi_{s+t}(m+n-1)}
\endxy
\end{equation}
Moreover, by induction, since $(s-1,t)<(s,t)$ one can easily check that this is also the right hand side of~(a). 
\begin{figure}
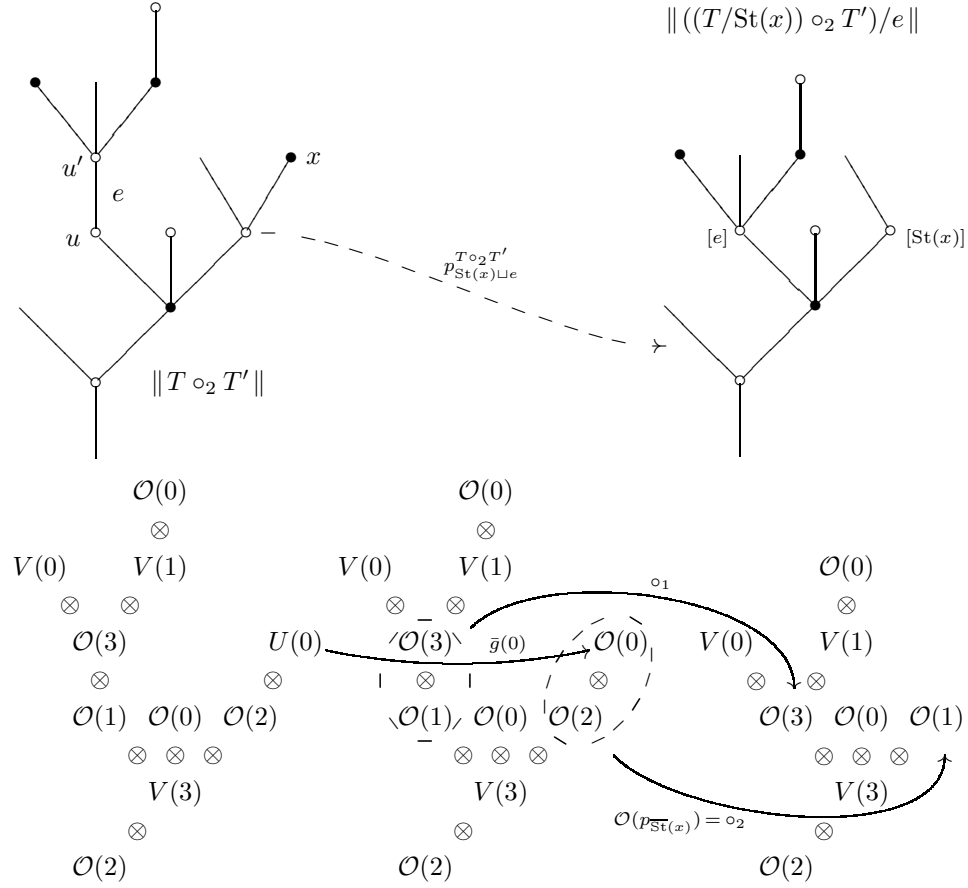

$$
\begin{array}{c}
\xy
(0,0)*{}="a",
(0,10)*-<2pt>{\circ}="b",
(-10,20)*{}="c",
(10,20)*-{\bullet}="d",
(29,40)*{x},
(0,30)*-<1pt>{\circ}="e",
(-3,29)*{u},
(-3,39)*{u'},
(3,35)*{e},
(10,30)*-<2pt>{\circ}="f",
(20,30)*-<2pt>{\circ}="g",
(-8,40)*-{}="h",
(0,40)*-<1pt>{\circ}="i",
(-8,50)*-{\bullet}="hh",
(0,50)*{}="ii",
(8,50)*-{\bullet}="jj",
(8,60)*-<2pt>{\circ}="ll",
(8,40)*-{}="j",
(0,50)*{}="k",
(14,40)*{}="m",
(26,40)*-{\bullet}="n",
(15,10)*{\norm{T\circ_{2}T'}},
\ar@{-}"a";"b",
\ar@{-}"b";"c",
\ar@{-}"b";"d",
\ar@{-}"d";"e",
\ar@{-}"d";"f",
\ar@{-}"d";"g",
\ar@{-}"e";"i",
\ar@{-}"g";"m",
\ar@{-}"g";"n",
\ar@{-}"i";"hh",
\ar@{-}"i";"ii",
\ar@{-}"i";"jj",
\ar@{-}"jj";"ll"
\ar@(r,l)@{-->}(23,30);(75,15)^{p^{T\circ_{2}T'}_{\star(x) \sqcup e}}
\endxy
\end{array}\hspace{-8pt}
\begin{array}{c}
\xy
(0,0)*{}="a",
(0,10)*-<2pt>{\circ}="b",
(-10,20)*{}="c",
(10,20)*-{\bullet}="d",
(0,30)*-<1pt>{\circ}="e",
(-3,29)*{\scriptstyle [e]},
(10,30)*-<2pt>{\circ}="f",
(20,30)*-<2pt>{\circ}="g",
(26,29)*{\scriptstyle [\star(x)]},
(-8,40)*-{}="h",
(0,40)*{}="i",
(-8,40)*-{\bullet}="hh",
(0,40)*{}="ii",
(8,40)*-{\bullet}="jj",
(8,50)*-<2pt>{\circ}="ll",
(8,40)*-{}="j",
(0,50)*{}="k",
(14,40)*{}="m",
(7,58)*{\norm{((T/\star(x))\circ_{2}T')/e}},
\ar@{-}"a";"b",
\ar@{-}"b";"c",
\ar@{-}"b";"d",
\ar@{-}"d";"e",
\ar@{-}"d";"f",
\ar@{-}"d";"g",
\ar@{-}"g";"m",
\ar@{-}"e";"hh",
\ar@{-}"e";"ii",
\ar@{-}"e";"jj",
\ar@{-}"jj";"ll"
\endxy
\end{array}$$
$$\begin{array}{c}
\xy
(0,0)*{}="a",
(0,10)*-<2pt>{\mathcal{O}(2)}="b",
(-10,20)*{}="c",
(10,20)*-{ V(3)}="d",
(0,30)*-<1pt>{\mathcal{O}(1)}="ee",
(0,40)*-<1pt>{\mathcal{O}(3)}="e",
(10,30)*-<2pt>{\mathcal{O}(0)}="f",
(20,30)*-<2pt>{\mathcal{O}(2)}="g",
(-8,50)*-{ V(0)}="h",
(0,50)*{}="i",
(8,50)*-{ V(1)}="j",
(0,60)*{}="k",
(8,60)*-<2pt>{\mathcal{O}(0)}="l",
(14,40)*{}="m",
(26,40)*{U(0)}="n"
\ar@{}"a";"b"
\ar@{}"b";"c"
\ar@{}"b";"d"|{\displaystyle\otimes}
\ar@{}"d";"ee"|{\displaystyle\otimes}
\ar@{}"ee";"e"|{\displaystyle\otimes}
\ar@{}"d";"f"|{\displaystyle\otimes}
\ar@{}"d";"g"|{\displaystyle\otimes}
\ar@{}"e";"h"|{\displaystyle\otimes}
\ar@{}"e";"i"
\ar@{}"e";"j"|{\displaystyle\otimes}
\ar@{}"j";"l"|{\displaystyle\otimes}
\ar@{}"g";"m"
\ar@{}"g";"n"|{\displaystyle\otimes}
\ar@/_5pt/(30,39);(65,39)^<(.7){\bar{g}(0)}
\endxy
\end{array}\hspace{-102pt}
\begin{array}{c}
\xy
(0,0)*{}="a",
(0,10)*-<2pt>{\mathcal{O}(2)}="b",
(-10,20)*{}="c",
(10,20)*-{ V(3)}="d",
(0,30)*-<1pt>{\mathcal{O}(1)}="ee",
(0,40)*-<1pt>{\mathcal{O}(3)}="e",
(10,30)*-<2pt>{\mathcal{O}(0)}="f",
(20,30)*-<2pt>{\mathcal{O}(2)}="g",
(-8,50)*-{ V(0)}="h",
(0,50)*{}="i",
(8,50)*-{ V(1)}="j",
(0,60)*{}="k",
(8,60)*-<2pt>{\mathcal{O}(0)}="l",
(14,40)*{}="m",
(26,40)*{\mathcal{O}(0)}="n",
(0,43);(0,35),{\ellipse<,17pt>{--}},
(28,43);(23,35),{\ellipse<,17pt>{--}}
\ar@{}"a";"b"
\ar@{}"b";"c"
\ar@{}"b";"d"|{\displaystyle\otimes}
\ar@{}"d";"ee"|{\displaystyle\otimes}
\ar@{}"ee";"e"|{\displaystyle\otimes}
\ar@{}"d";"f"|{\displaystyle\otimes}
\ar@{}"d";"g"|{\displaystyle\otimes}
\ar@{}"e";"h"|{\displaystyle\otimes}
\ar@{}"e";"i"
\ar@{}"e";"j"|{\displaystyle\otimes}
\ar@{}"j";"l"|{\displaystyle\otimes}
\ar@{}"g";"m"
\ar@{}"g";"n"|{\displaystyle\otimes}
\ar@(ru,u)(6,42);(49,34)^{\circ_{1}}
\ar@(rd,d)(25,25);(69,25)_<(.3){\mathcal{O}(p_{\overline{\star}(x)})\,=\, \circ_{2}\quad}
\endxy
\end{array}\hspace{-100pt}
\begin{array}{c}
\vspace{-13.5pt}
\\
\xy
(0,0)*{}="a",
(0,10)*-<2pt>{\mathcal{O}(2)}="b",
(-10,20)*{}="c",
(10,20)*-{ V(3)}="d",
(0,30)*-<1pt>{\mathcal{O}(3)}="e",
(10,30)*-<2pt>{\mathcal{O}(0)}="f",
(20,30)*-<2pt>{\mathcal{O}(1)}="g",
(-8,40)*-{ V(0)}="h",
(0,40)*{}="i",
(8,40)*-{ V(1)}="j",
(0,62)*{}="k",
(8,50)*-<2pt>{\mathcal{O}(0)}="l",
(14,40)*{}="m",
(26,40)*{}="n"
\ar@{}"a";"b"
\ar@{}"b";"c"
\ar@{}"b";"d"|{\displaystyle\otimes}
\ar@{}"d";"e"|{\displaystyle\otimes}
\ar@{}"d";"f"|{\displaystyle\otimes}
\ar@{}"d";"g"|{\displaystyle\otimes}
\ar@{}"e";"h"|{\displaystyle\otimes}
\ar@{}"e";"i"
\ar@{}"e";"j"|{\displaystyle\otimes}
\ar@{}"j";"l"|{\displaystyle\otimes}
\ar@{}"g";"m"
\ar@{}"g";"n"
\endxy
\end{array}$$
\caption{An illustration of \eqref{ultima3} for $T$ and $T'$ as in Figure \ref{xu1} in case $\{x,u\}\notin E(T)$. 
}
\label{xu11}
\end{figure}

Suppose now that $\{x,u\}\in E(T)$. Then $u\in\link(x)$. 
Assume that $e$ is the $l^{\text{th}}$ leaf of $\overline{\star}(x)$. 
We denote $T''$ the planted planar tree with leaves $T''=\overline{\star}(x)\circ_l C_{\val{T'}{u'}}$. 
The inner part of $T''$ is identified with the subtree $T'''\subset T\circ_i T'$ formed by adjoining the edge $e$ to $\star(x)$, see Figure \ref{xu2}.  
\begin{figure}
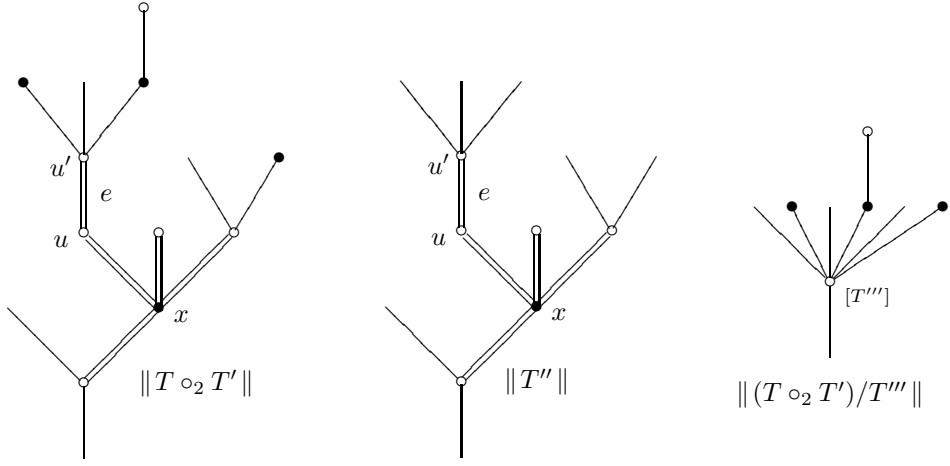

$$
\begin{array}{c}
\xy
(0,0)*{}="a",
(15,10)*{\norm{T\circ_{2}T'}},
(0,10)*-<2pt>{\circ}="b",
(-10,20)*{}="c",
(10,20)*-{\bullet}="d",
(13,19)*{x},
(0,30)*-<1pt>{\circ}="e",
(-3,29)*{u},
(-3,39)*{u'},
(3,35)*{e},
(10,30)*-<2pt>{\circ}="f",
(20,30)*-<2pt>{\circ}="g",
(-8,40)*-{}="h",
(0,40)*-<2pt>{\circ}="i",
(-8,50)*-{\bullet}="hh",
(0,50)*{}="ii",
(8,50)*-{\bullet}="jj",
(8,60)*-<2pt>{\circ}="ll",
(8,40)*-{}="j",
(0,50)*{}="k",
(14,40)*{}="m",
(26,40)*-{\bullet}="n",
\ar@{-}"a";"b",
\ar@{-}"b";"c",
\ar@{=}"b";"d",
\ar@{=}"d";"e",
\ar@{=}"d";"f",
\ar@{=}"d";"g",
\ar@{=}"e";"i",
\ar@{-}"g";"m",
\ar@{-}"g";"n",
\ar@{-}"i";"hh",
\ar@{-}"i";"ii",
\ar@{-}"i";"jj",
\ar@{-}"jj";"ll"
\endxy
\end{array}\qquad
\quad\begin{array}{c}
\xy
(0,0)*{}="a",
(10,10)*{\norm{T''}},
(0,10)*-<2pt>{\circ}="b",
(-10,20)*{}="c",
(10,20)*-{\bullet}="d",
(13,19)*{x},
(0,30)*-<1pt>{\circ}="e",
(-3,29)*{u},
(-3,39)*{u'},
(3,35)*{e},
(10,30)*-<2pt>{\circ}="f",
(20,30)*-<2pt>{\circ}="g",
(-8,40)*-{}="h",
(0,40)*-<2pt>{\circ}="i",
(-8,50)*{}="hh",
(0,50)*{}="ii",
(8,50)*{}="jj",
(8,60)*{}="ll",
(8,40)*-{}="j",
(0,50)*{}="k",
(14,40)*{}="m",
(26,40)*{}="n",
\ar@{-}"a";"b",
\ar@{-}"b";"c",
\ar@{=}"b";"d",
\ar@{=}"d";"e",
\ar@{=}"d";"f",
\ar@{=}"d";"g",
\ar@{=}"e";"i",
\ar@{-}"g";"m",
\ar@{-}"g";"n",
\ar@{-}"i";"hh",
\ar@{-}"i";"ii",
\ar@{-}"i";"jj",
\endxy
\end{array}\qquad
\begin{array}{c}
\xy
(0,0)*{}="a",
(0,-5)*{\norm{(T\circ_{2}T')/T'''}},
(0,10)*-<1pt>{\circ}="b",
(-10,20)*{}="c",
(-5,20)*-{\bullet}="hh",
(0,20)*{}="ii",
(5,20)*-{\bullet}="jj",
(5,30)*-<2pt>{\circ}="ll",
(8,40)*-{}="j",
(0,20)*{}="k",
(10,20)*{}="m",
(15,20)*-{\bullet}="n",
(5,8)*{\scriptstyle [T''']},
\ar@{-}"a";"b",
\ar@{-}"b";"c",
\ar@{-}"b";"m",
\ar@{-}"b";"n",
\ar@{-}"b";"hh",
\ar@{-}"b";"ii",
\ar@{-}"b";"jj",
\ar@{-}"jj";"ll"
\endxy
\end{array}$$
\caption{For the choice of $x$ in Figure \ref{xu1} with $\{x,u\}\in E(T)$  we here depict $T''$. The subtree $T'''$ is indicated with double lines.}
\label{xu2}
\end{figure}
Using the definition of $d_i^{s,t}(T,T')$ in the statement,  the definition of $\bar{\psi}_{s+t}^{(T\circ_iT')/e}$ in Lemma \ref{pind}, and relation (2) in Remark \ref{circi} for $\mathcal{O}$, we deduce that, in this case, the left hand side of (a) is the following composite morphism, see Figure \ref{xu22},
\begin{equation}\label{ultima4}
\xy
(0,0)*{U(\val{T}{x})\otimes\hspace{-17pt}\bigotimes\limits_{v\in I^{e}(T)\setminus\{x\}}\hspace{-17pt} V(\val{T}{v})\otimes
\hspace{-9pt}\bigotimes\limits_{w\in I^o(T)}\hspace{-9pt}\mathcal{O}(\val{T}{w})
\otimes\hspace{-12pt} 
\bigotimes\limits_{v'\in I^{e}(T')}\hspace{-12pt} V(\val{T'}{v'})\otimes
\hspace{-12pt}\bigotimes\limits_{w'\in I^o(T')}\hspace{-12pt}\mathcal{O}(\val{T'}{w'})},
(0,-20)*{
\mathcal{O}(\val{T}{x})\otimes\hspace{-17pt}\bigotimes\limits_{v\in I^{e}(T)\setminus\{x\}}\hspace{-17pt} V(\val{T}{v})\otimes
\hspace{-9pt}\bigotimes\limits_{w\in I^o(T)}\hspace{-9pt}\mathcal{O}(\val{T}{w})
\otimes\hspace{-12pt} 
\bigotimes\limits_{v'\in I^{e}(T')}\hspace{-12pt} V(\val{T'}{v'})\otimes
\hspace{-12pt}\bigotimes\limits_{w'\in I^o(T')}\hspace{-12pt}\mathcal{O}(\val{T'}{w'})
},
(0,-40)*{
\mathcal{O}(T'')
\otimes\hspace{-10pt}\bigotimes\limits_{v\in I^{e}(T)\setminus\{x\}}\hspace{-17pt} V(\val{T}{v})\;\;\;\otimes
\hspace{-15pt}\bigotimes\limits_{w\in I^o(T)\setminus\link(x)}\hspace{-20pt}\mathcal{O}(\val{T}{w})
\;\;\otimes\hspace{-5pt} 
\bigotimes\limits_{v'\in I^{e}(T')}\hspace{-12pt} V(\val{T'}{v'})\;\;\;\otimes
\hspace{-15pt}\bigotimes\limits_{w'\in I^o(T')\setminus\{u'\}}\hspace{-20pt}\mathcal{O}(\val{T'}{w'})
},
(-6,-60)*{
\mathcal{O}(r_x+\val{T'}{u'}-1)
\otimes\hspace{-10pt}\bigotimes\limits_{v\in I^{e}(T)\setminus\{x\}}\hspace{-17pt} V(\val{T}{v})\;\;\;\otimes
\hspace{-15pt}\bigotimes\limits_{w\in I^o(T)\setminus\link(x)}\hspace{-20pt}\mathcal{O}(\val{T}{w})
\;\;\otimes\hspace{-5pt} 
\bigotimes\limits_{v'\in I^{e}(T')}\hspace{-12pt} V(\val{T'}{v'})\;\;\;\otimes
\hspace{-15pt}\bigotimes\limits_{w'\in I^o(T')\setminus\{u'\}}\hspace{-20pt}\mathcal{O}(\val{T'}{w'})
},
(-4,-80)*{
\bigotimes\limits_{v\in I^{e}((T\circ_i T')/T''')}\hspace{-28pt} V(\val{(T\circ_i T')/T'''}{v})\qquad\otimes
\hspace{-8pt}\bigotimes\limits_{w\in I^o((T\circ_i T')/T''')}\hspace{-28pt}\mathcal{O}(\val{(T\circ_i T')/T'''}{w})
},
(0,-100)*{P_{s+t-1}(m+n-1)},
(0,-120)*{P_{s+t}(m+n-1)}
\ar(0,-2);(0,-15)_{\bar{g}(\val{T}{x})\otimes\id{}}
\ar(0,-22);(0,-35)_{\cong}^{\text{symmetry}}
\ar(0,-45);(0,-55)_{\mathcal{O}(p_{T''}) \otimes\id{}}
\ar(0,-65);(0,-75)_{\cong}^{\text{symmetry}}
\ar(0,-84);(0,-97)_{\bar{\psi}_{s+t-1}^{(T\circ_i T')/T'''}}
\ar(0,-103);(0,-117)_{\varphi_{s+t}(m+n-1)}
\endxy
\end{equation}
\begin{figure}
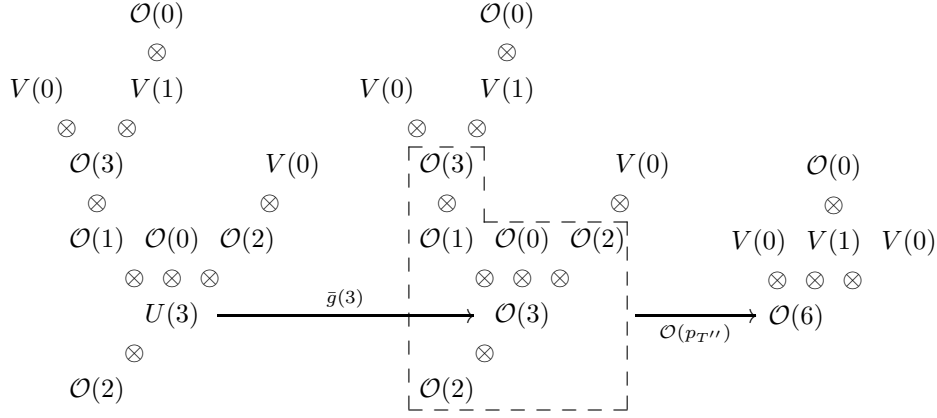

$$\begin{array}{c}
\xy
(0,0)*{}="a",
(0,10)*-<2pt>{\mathcal{O}(2)}="b",
(-10,20)*{}="c",
(10,20)*-{ U(3)}="d",
(0,30)*-<1pt>{\mathcal{O}(1)}="ee",
(0,40)*-<1pt>{\mathcal{O}(3)}="e",
(10,30)*-<2pt>{\mathcal{O}(0)}="f",
(20,30)*-<2pt>{\mathcal{O}(2)}="g",
(-8,50)*-{ V(0)}="h",
(0,50)*{}="i",
(8,50)*-{ V(1)}="j",
(0,60)*{}="k",
(8,60)*-<2pt>{\mathcal{O}(0)}="l",
(14,40)*{}="m",
(26,40)*{V(0)}="n"
\ar@{}"a";"b"
\ar@{}"b";"c"
\ar@{}"b";"d"|{\displaystyle\otimes}
\ar@{}"d";"ee"|{\displaystyle\otimes}
\ar@{}"ee";"e"|{\displaystyle\otimes}
\ar@{}"d";"f"|{\displaystyle\otimes}
\ar@{}"d";"g"|{\displaystyle\otimes}
\ar@{}"e";"h"|{\displaystyle\otimes}
\ar@{}"e";"i"
\ar@{}"e";"j"|{\displaystyle\otimes}
\ar@{}"j";"l"|{\displaystyle\otimes}
\ar@{}"g";"m"
\ar@{}"g";"n"|{\displaystyle\otimes}
\ar(16,20);(50,20)^-{\bar{g}(3)}
\endxy
\end{array}\hspace{-50pt}
\begin{array}{c}
\xy
(0,0)*{}="a",
(0,10)*-<2pt>{\mathcal{O}(2)}="b",
(-10,20)*{}="c",
(10,20)*-{ \mathcal{O}(3)}="d",
(0,30)*-<1pt>{\mathcal{O}(1)}="ee",
(0,40)*-<1pt>{\mathcal{O}(3)}="e",
(10,30)*-<2pt>{\mathcal{O}(0)}="f",
(20,30)*-<2pt>{\mathcal{O}(2)}="g",
(-8,50)*-{ V(0)}="h",
(0,50)*{}="i",
(8,50)*-{ V(1)}="j",
(0,60)*{}="k",
(8,60)*-<2pt>{\mathcal{O}(0)}="l",
(14,40)*{}="m",
(26,40)*{V(0)}="n"
\ar@{}"a";"b"
\ar@{}"b";"c"
\ar@{}"b";"d"|{\displaystyle\otimes}
\ar@{}"d";"ee"|{\displaystyle\otimes}
\ar@{}"ee";"e"|{\displaystyle\otimes}
\ar@{}"d";"f"|{\displaystyle\otimes}
\ar@{}"d";"g"|{\displaystyle\otimes}
\ar@{}"e";"h"|{\displaystyle\otimes}
\ar@{}"e";"i"
\ar@{}"e";"j"|{\displaystyle\otimes}
\ar@{}"j";"l"|{\displaystyle\otimes}
\ar@{}"g";"m"
\ar@{}"g";"n"|{\displaystyle\otimes}
\ar(25,20);(41,20)_-{\mathcal{O}(p_{T''})}
\ar@{--}(-5,7.5);(-5,42.5)
\ar@{--}(-5,7.5);(24,7.5)
\ar@{--}(-5,42.5);(5,42.5)
\ar@{--}(5,42.5);(5,32.5)
\ar@{--}(5,32.5);(24,32.5)
\ar@{--}(24,32.5);(24,7.5)
\endxy
\end{array}\hspace{-23pt}
\begin{array}{c}
\xy
(0,0)*{}="a",
(0,10)*-<1pt>{\mathcal{O}(6)}="b",
(-10,20)*{}="c",
(-5,20)*-{V(0)}="hh",
(0,20)*{}="ii",
(5,20)*-{V(1)}="jj",
(5,30)*-<2pt>{\mathcal{O}(0)}="ll",
(8,43)*-{}="j",
(0,20)*{}="k",
(10,20)*{}="m",
(15,20)*-{V(0)}="n",
\ar@{}"b";"m",
\ar@{}"b";"n",|{\displaystyle\otimes}
\ar@{}"b";"hh",|{\displaystyle\otimes}
\ar@{}"b";"ii",
\ar@{}"b";"jj",|{\displaystyle\otimes}
\ar@{}"jj";"ll"|{\displaystyle\otimes}
\endxy
\end{array}$$
\caption{An illustration of \eqref{ultima4} for $T$ and $T'$ as in Figure \ref{xu1} in case $\{x,u\}\in E(T)$, see Figure \ref{xu2}.}
\label{xu22}
\end{figure}

\noindent Moreover, by induction one can easily check that this is also the right hand side of~(a), hence we are done with this proof.
\end{proof}

Let $\mathcal{P}$ be the sequence defined as
$$\mathcal{P}(n)=\colim\limits_{t\geq 0}P_{t}(n).$$

By the previous lemma, the morphisms $c_{i}^{s,t}(m,n)$ induce composition laws in the colimit,
\begin{equation}\label{compush}
 \circ_{i}\colon \mathcal{P}(m)\otimes \mathcal{P}(n)\To \mathcal{P}(m+n-1),\quad 1\leq i\leq m,\; n\geq0.
\end{equation}

Consider the morphism
\begin{equation}\label{upush}
\xy
(-5,.4)*{\unit},
(17,0)*{\mathcal{O}(1)=P_{0}(1)},
(56,-1)*{\colim\limits_{t\geq 0}P_{t}(1)=\mathcal{P}(1).},
\ar(28,0);(40,0)^-{\text{canonical}}
\ar(-3,0);(6,0)^-{u}
\endxy
\end{equation}

\begin{prop}
The sequence $\mathcal{P}$, the unit \eqref{upush} and the composition laws \eqref{compush} define an operad.
\end{prop}

\begin{proof}
We must check that relations (1)--(4) in Remark \ref{circi} hold for $\mathcal{P}$. Each of these relations for $\mathcal{P}$ can be derived from the corresponding relation for $\mathcal{O}$. As relations (1) and (2) 
are very similar to each other, just as (3) and (4), we here check (2) and (3).

In order to prove relation (2) for $\mathcal{P}$ it is enough to check that the following two morphisms $P_r(l)\otimes P_s(m)\otimes P_t(n)\r P_{r+s+t}(l+m+n-2)$ coincide,
$$c^{r+s,t}_j(l+m-1,n)(c^{r,s}_i(l,m)\otimes\id{P_t(n)})=
c^{r,s+t}_i(l,m+n-1)(\id{P_r(l)}\otimes c^{s,t}_{j-i+1}(m,n)).$$
We check this by induction on $(r,s,t)\in\mathbb{N}^3$ with respect to the graded lexicographic order. For $r=s=t=0$ this is just relation (2) for the operad $\mathcal{O}$. If we assume that the relation holds up to the predecessor of $(r,s,t)$, then by using the universal property of the push-out definition of
$P_r(m)\otimes P_s(n)\otimes P_t(p)$ arising from Lemmas \ref{podot} and \ref{pind}, we only have to check that, with the notation of Lemma \ref{pond}, given planted planar trees with leaves concentrated in even levels $T, T', T''$ with $\card L(T)=l$, $\card L(T')=m$, $\card L(T'')=n$, $\card I^e(T)=r$, $\card I^e(T')=s$, and $\card I^e(T'')=t$, then
\begin{equation*}\tag{a}
c^{r+s,t}_j(l+m-1,n)(d_i^{r,s}(T,T')\otimes\bar{\psi}^{T''}_t)=
c^{r,s+t}_i(l,m+n-1)(\bar{\psi}^{T}_r\otimes d^{s,t}_{j-i+1}(T',T'')).
\end{equation*}

Let $u\in I^o(T)$ be the inner vertex of the $i^{\text{th}}$ leaf edge of $T$, $u'_1\in I^o(T')$ the unique level $1$ vertex of $T'$, $u'_2\in I^o(T')$ the inner vertex of the $(j-i+1)^{\text{th}}$ leaf edge of $T'$, and $u''\in I^o(T'')$ the unique level $1$ vertex of $T''$. Suppose that the $i^{\text{th}}$ leaf edge of $T$ is the $k_1^{\text{th}}$ incomming edge of $u$, and that the $(j-i+1)^{\text{th}}$ leaf edge of $T'$ is the $k_2^{th}$ incomming edge of $u_2'$. The most complicated case is when $u_1'=u_2'$, and even this case is easy, although somewhat tedious. 

\begin{figure}
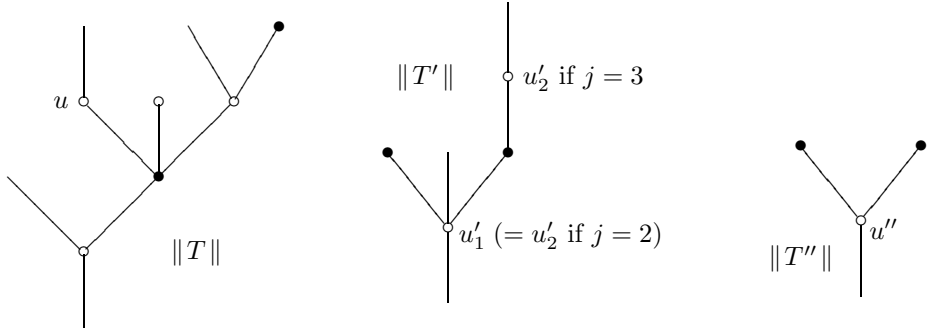

$$
\begin{array}{c}
\xy
(15,10)*{\norm{T}},
(0,0)*{}="a",
(0,10)*-<2pt>{\circ}="b",
(-10,20)*{}="c",
(10,20)*-{\bullet}="d",
(0,30)*-<1pt>{\circ}="e",
(-3,30)*{u},
(10,30)*-<2pt>{\circ}="f",
(20,30)*-<2pt>{\circ}="g",
(-8,40)*-{}="h",
(0,40)*{}="i",
(8,40)*-{}="j",
(0,50)*{}="k",
(14,40)*{}="m",
(26,40)*-{\bullet}="n",
\ar@{-}"a";"b",
\ar@{-}"b";"c",
\ar@{-}"b";"d",
\ar@{-}"d";"e",
\ar@{-}"d";"f",
\ar@{-}"d";"g",
\ar@{-}"e";"i",
\ar@{-}"g";"m",
\ar@{-}"g";"n",
\endxy
\end{array}\qquad \quad
\begin{array}{ll}
\xy
(-3,50)*{\norm{T'}},
(0,20)*{}="d",
(0,30)*-<1pt>{\circ}="e",
(15,29)*{u'_{1}\;(=u'_{2}\text{ if }j=2)},
(18,50)*{u'_{2}\text{ if }j=3},
(-8,40)*-{\bullet}="h",
(0,40)*{}="i",
(8,40)*-{\bullet}="j",
(0,50)*{}="k",
(8,50)*-<1pt>{\circ}="l",
(8,60)*{}="ll",
(14,40)*{}="m",
(26,40)*{}="n",
\ar@{-}"d";"e",
\ar@{-}"e";"h",
\ar@{-}"e";"i",
\ar@{-}"e";"j",
\ar@{-}"j";"l",
\ar@{-}"l";"ll",
\endxy\vspace{-10pt}
&\qquad\quad
\xy
(-8,25)*{\norm{T''}},
(0,20)*{}="d",
(0,30)*-<1pt>{\circ}="e",
(3,29)*{u''},
(-8,40)*-{\bullet}="h",
(8,40)*-{\bullet}="j",
(0,50)*{}="k",
(14,40)*{}="m",
(26,40)*{}="n",
\ar@{-}"d";"e",
\ar@{-}"e";"h",
\ar@{-}"e";"j",
\endxy
\end{array}$$
\caption{For the planted planar trees with leaves $T$, $T'$ and $T''$ we depict $u$, $u'_{1}$, $u'_{2}$ and $u''$ for $i=2$ and $j=2,3$.}
\label{uprima12}
\end{figure}

Assume $u_1'=u_2'$ and denote this vertex simply by $u'$. Notice that $(T\circ_iT')\circ_jT''
= T\circ_i(T'\circ_{j-i+1}T'')$, compare Figure \ref{asoc2}. Let $K\subset (T\circ_iT')\circ_jT''$ be the subtree with  $V(K)=\{u,u',u''\}$ and $E(K)=\{\{u,u'\},\{u',u''\}\}$, see Figure \ref{yoqueseya}. Then by Lemma \ref{pond} and relation (2) for $\mathcal{O}$, both sides of (a) coincide with
$$\xy
(0,0)*{
\bigotimes\limits_{v\in I^{e}(T)}\hspace{-8pt} V(\val{T}{v})\;\otimes
\hspace{-7pt}\bigotimes\limits_{w\in I^o(T)}\hspace{-7pt}\mathcal{O}(\val{T}{w})
\;\otimes 
\hspace{-7pt} 
\bigotimes\limits_{v'\in I^{e}(T')}\hspace{-10pt} V(\val{T'}{v'})\,\,\,\otimes
\hspace{-7pt}\bigotimes\limits_{w'\in I^o(T')}\hspace{-10pt}\mathcal{O}(\val{T'}{w'}) 
\;\otimes 
\hspace{-8pt} 
\bigotimes\limits_{v''\in I^{e}(T'')}\hspace{-12pt} V(\val{T''}{v''})\;\otimes
\hspace{-12pt}\bigotimes\limits_{w''\in I^o(T'')}\hspace{-12pt}\mathcal{O}(\val{T''}{w''})
},
(-5,-20)*{
\mathcal{O}(\val{T}{u})\otimes \mathcal{O}(\val{T'}{u'})\otimes \mathcal{O}(\val{T''}{u''})
\qquad\otimes
\hspace{-20pt}
\bigotimes\limits_{v\in I^{e}(T)\cup I^{e}(T')\cup I^{e}(T'')}\hspace{-35pt} V(\val{}{v})
\qquad\qquad\otimes
\hspace{-27pt}\bigotimes\limits_{w\in (I^o(T)\cup I^o(T')\cup I^o(T''))\setminus\{u,u',u''\}}\hspace{-57pt}\mathcal{O}(\val{}{w})
},
(-5,-40)*{
\mathcal{O}(\val{T}{u}+\val{T'}{u'}+\val{T''}{u''}-2)\\
\qquad\otimes
\hspace{-15pt}
\bigotimes\limits_{v\in I^e((T\circ_i T')\circ_jT'')/K}\hspace{-35pt} V(\val{}{v})\qquad\quad\otimes
\hspace{-23pt}\bigotimes\limits_{w\in I^o(((T\circ_i T')\circ_jT'')/K)\setminus\{[K]\}}\hspace{-53pt}\mathcal{O}(\val{}{w})
},
(0,-60)*{P_{s+t}(m+n+p-2)}
\ar(0,-5);(0,-15)_{\cong}^{\text{symmetry}}
\ar(0,-25);(0,-35)_{(\circ_{k_1}(\id{}\otimes\circ_{k_2}))\otimes\id{}}
\ar(0,-45);(0,-57)_{\bar{\psi}_{r+s+t}^{((T\circ_i T')\circ_jT'')/K}}
\endxy$$

\begin{figure}
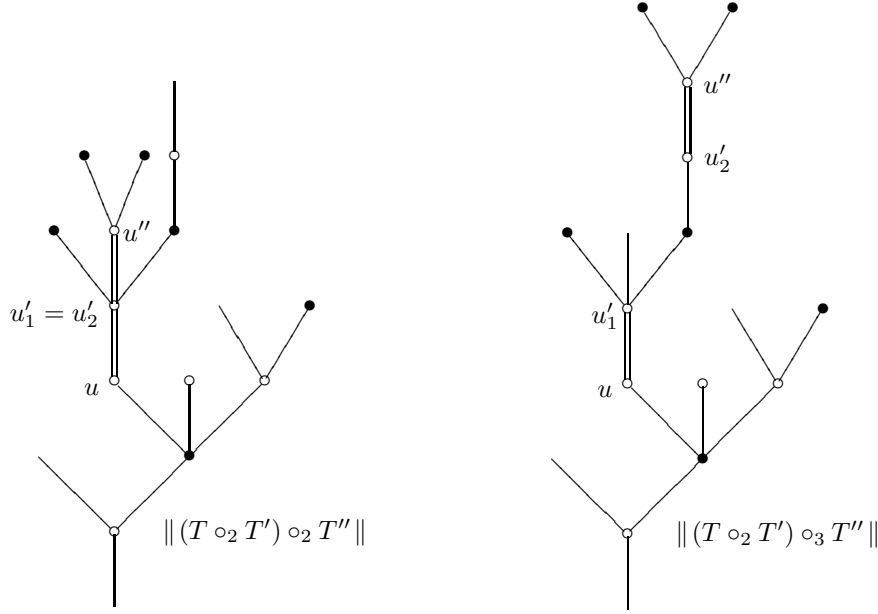

$$
\begin{array}{c}
\xy
(0,0)*{}="a",
(20,10)*{\norm{(T\circ_{2}T')\circ_{2}T''}},
(0,10)*-<2pt>{\circ}="b",
(-10,20)*{}="c",
(10,20)*-{\bullet}="d",
(0,30)*-<1pt>{\circ}="e",
(-3,29)*{u},
(3,50)*{u''},
(-8,39)*{u'_{1}=u'_{2}},
(10,30)*-<2pt>{\circ}="f",
(20,30)*-<2pt>{\circ}="g",
(-8,40)*-{}="h",
(0,40)*-<2pt>{\circ}="i",
(-8,50)*-{\bullet}="hh",
(0,50)*{}="ii",
(8,50)*-{\bullet}="jj",
(8,60)*-<1pt>{\circ}="ll",
(0,50)*-<1pt>{\circ}="lll",
(8,70)*{}="llll",
(-4,60)*-<1pt>{\bullet}="r",
(4,60)*-<1pt>{\bullet}="s",
(8,40)*-{}="j",
(0,50)*{}="k",
(14,40)*{}="m",
(26,40)*-{\bullet}="n",
(14,80)*{}="t",
\ar@{-}"a";"b",
\ar@{-}"b";"c",
\ar@{-}"b";"d",
\ar@{-}"d";"e",
\ar@{-}"d";"f",
\ar@{-}"d";"g",
\ar@{=}"e";"i",
\ar@{-}"g";"m",
\ar@{-}"g";"n",
\ar@{-}"i";"hh",
\ar@{-}"i";"jj",
\ar@{-}"jj";"ll",
\ar@{=}"i";"lll",
\ar@{-}"ll";"llll",
\ar@{-}"lll";"r",
\ar@{-}"lll";"s"
\endxy
\end{array}\qquad\qquad\qquad
\begin{array}{c}
\xy
(0,0)*{}="a",
(20,10)*{\norm{(T\circ_{2}T')\circ_{3}T''}},
(0,10)*-<2pt>{\circ}="b",
(-10,20)*{}="c",
(10,20)*-{\bullet}="d",
(0,30)*-<1pt>{\circ}="e",
(-3,29)*{u},
(-3,39)*{u'_{1}},
(12,60)*{u'_{2}},
(12,70)*{u''},
(10,30)*-<2pt>{\circ}="f",
(20,30)*-<2pt>{\circ}="g",
(-8,40)*-{}="h",
(0,40)*-<1pt>{\circ}="i",
(-8,50)*-{\bullet}="hh",
(0,50)*{}="ii",
(8,50)*-{\bullet}="jj",
(8,60)*-<1pt>{\circ}="ll",
(8,70)*-<1pt>{\circ}="lll",
(2,80)*-<1pt>{\bullet}="r",
(14,80)*-<1pt>{\bullet}="s",
(8,40)*-{}="j",
(0,50)*{}="k",
(14,40)*{}="m",
(26,40)*-{\bullet}="n",
\ar@{-}"a";"b",
\ar@{-}"b";"c",
\ar@{-}"b";"d",
\ar@{-}"d";"e",
\ar@{-}"d";"f",
\ar@{-}"d";"g",
\ar@{=}"e";"i",
\ar@{-}"g";"m",
\ar@{-}"g";"n",
\ar@{-}"i";"hh",
\ar@{-}"i";"ii",
\ar@{-}"i";"jj",
\ar@{-}"jj";"ll",
\ar@{=}"ll";"lll",
\ar@{-}"lll";"r",
\ar@{-}"lll";"s"
\endxy
\end{array}$$
\caption{Here we depict the subtree $K\subset (T\circ_{i}T')\circ_{j}T'' $ in double lines for the trees in Figure \ref{uprima12}, $i=2$ and $j=2,3$. }
\label{yoqueseya}
\end{figure}

Assume now that $u_1'\neq u_2'$. In this case it is not even necessary to use any of the relations in Remark \ref{circi} for~$\mathcal{O}$.  Actually, by Lemma \ref{pond}, if $K\subset (T\circ_i T')\circ_j T''$ is the (disjoint) union of the edges $e_1=\{u,u_1'\}$ and $e_2=\{u_2',u''\}$, see Figure \ref{yoqueseya}, then both sides of (a) coincide with
$$\xy
(0,0)*{
\bigotimes\limits_{v\in I^{e}(T)}\hspace{-8pt} V(\val{T}{v})\;\otimes
\hspace{-7pt}\bigotimes\limits_{w\in I^o(T)}\hspace{-7pt}\mathcal{O}(\val{T}{w})
\;\otimes 
\hspace{-7pt} 
\bigotimes\limits_{v'\in I^{e}(T')}\hspace{-10pt} V(\val{T'}{v'})\,\,\,\otimes
\hspace{-7pt}\bigotimes\limits_{w'\in I^o(T')}\hspace{-10pt}\mathcal{O}(\val{T'}{w'}) 
\;\otimes 
\hspace{-8pt} 
\bigotimes\limits_{v''\in I^{e}(T'')}\hspace{-12pt} V(\val{T''}{v''})\;\otimes
\hspace{-12pt}\bigotimes\limits_{w''\in I^o(T'')}\hspace{-12pt}\mathcal{O}(\val{T''}{w''})
},
(-8,-20)*{
\mathcal{O}(\val{T}{u})\otimes \mathcal{O}(\val{T'}{u'_1})\otimes \mathcal{O}(\val{T'}{u'_2})\otimes \mathcal{O}(\val{T''}{u''})
\otimes
\hspace{-35pt}
\bigotimes\limits_{v\in I^{e}(T)\cup I^{e}(T')\cup I^{e}(T'')}\hspace{-35pt} V(\val{}{v})
\qquad\qquad\otimes
\hspace{-30pt}\bigotimes\limits_{w\in (I^o(T)\cup I^o(T')\cup I^o(T''))\setminus\{u,u'_1,u'_2,u''\}}\hspace{-65pt}\mathcal{O}(\val{}{w})
},
(-8,-40)*{
\mathcal{O}(\val{T}{u}+\val{T'}{u'_1}-1)
\otimes \mathcal{O}(\val{T'}{u'_2}+\val{T''}{u''}-1)
\otimes
\hspace{-35pt}
\bigotimes\limits_{v\in I^e((T\circ_i T')\circ_jT'')/K}\hspace{-35pt} V(\val{}{v})\qquad\qquad\otimes
\hspace{-25pt}\bigotimes\limits_{w\in I^o(((T\circ_i T')\circ_jT'')/K)\setminus\{[e_1],[e_2]\}}\hspace{-60pt}\mathcal{O}(\val{}{w})
},
(0,-60)*{P_{s+t}(m+n+p-2)}
\ar(0,-5);(0,-15)_{\cong}^{\text{symmetry}}
\ar(0,-25);(0,-35)_{\circ_{k_1}\otimes \circ_{k_2}\otimes\id{}}
\ar(0,-45);(0,-57)_{\bar{\psi}_{r+s+t}^{((T\circ_i T')\circ_jT'')/K}}
\endxy$$

Relation (3) is a consequence of the fact that the following composite morphism is a right unit constraint in $\C{V}$,
$$\xymatrix@C=35pt{P_r(l)\otimes\unit\ar[r]^-{\id{}\otimes u}& P_r(l)\otimes \mathcal{O}(1)=P_r(l)\otimes P_0(1)\ar[r]^-{c_i^{r,0}(l,1)}& P_r(l).}$$
This follows by induction on $r$. For $r=0$ this is just relation (3) for $\mathcal{O}$. Assume this holds up to $r-1$. By Lemma~\ref{pind} and the induction hypothesis, we only have to check that 
the morphism $c^{r,0}_{i}(l,1)(\bar{\psi}_{r}^{T}\otimes u)$ coincides with the composition of the right unit isomorphism and $\bar{\psi}_{r}^{T}$. By Lemma \ref{pond}, $$c^{r,0}_{i}(l,1)(\bar{\psi}_{r}^{T}\otimes \id{\mathcal{O}(1)})=d^{r,0}_{i}(T,C_{1}).$$
Let 
$u'\in I^{o}(C_{1})$ be now the unique inner vertex of $C_{1}$, and $e=\{u,u'\}\in E(T\circ_{i}C_{1})$. In this case $(T\circ_{i}C_{1})/e=T$. Moreover, by the very definition of $d^{r,0}_{i}(T,C_{1})$ in the statement of Lemma \ref{pond} and by relation (3) for $\mathcal{O}$, the morphism $d^{r,0}_{i}(T,C_{1})(\id{}\otimes u)$ is the composition of the right unit isomorphism and $\bar{\psi}_{r}^{T}$, hence we are done. 
\end{proof}

Consider the morphisms of sequences $f'\colon\mathcal{O}\r\mathcal{P}$ and $\bar{g}'\colon V\r\mathcal{P}$ defined as
$$\xy
(10,0)*{f'(n)\colon \mathcal{O}(n)=P_{0}(n)},
(58,-1)*{\colim\limits_{t\geq 0}P_{t}(n)=\mathcal{P}(n),},
\ar(28,0);(40,0)^-{\text{canonical}}
\endxy$$
$$\xy
(-17,20)*{ V(n)\otimes \unit^{\otimes (n+1)}\cong V(n)},
(-16,0)*{V(n)\otimes \mathcal{O}(1)^{\otimes (n+1)}},
(25,0)*{P_1(n)},
(60,-1.2)*{\colim\limits_{t\geq 0} P_{t}(n)=\mathcal{P}(n)},
\ar(-25,17);(-25,3)^{\id{}\otimes u^{\otimes (n+1)}}
\ar(0,0);(20,0)^{\bar{\psi}_1^{C_1(C_n(C_1,\dots, C_1))}}
\ar(30,0);(44,0)^-{\text{canonical}}
\ar@/^12pt/(2,20);(70,3)^{\bar{g}'(n)}
\endxy$$

\begin{figure}
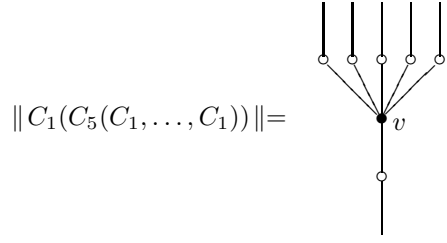

$$\xy/r2.2pt/:
(-40,30)*{\norm{C_{1}(C_{5}(C_{1},\dots,C_{1}))}=},
(0,20)*-<1pt>{\circ}="d",
(3,29)*{v},
(0,10)*{}="c",
(0,30)*-<1pt>{\bullet}="e",
(-5,40)*-<1pt>{\circ}="h",
(-10,40)*-<1pt>{\circ}="h1",
(10,40)*-<1pt>{\circ}="h5",
(0,40)*-<1pt>{\circ}="i",
(5,40)*-<1pt>{\circ}="j",
(0,50)*{}="k",
(5,50)*{}="l",
(-10,50)*{}="l1",
(-5,50)*{}="l2",
(0,50)*{}="l3",
(10,50)*{}="l5",
(14,40)*{}="m",
(26,40)*{}="n",
\ar@{-}"c";"d",
\ar@{-}"d";"e",
\ar@{-}"e";"h",
\ar@{-}"e";"h1",
\ar@{-}"h1";"l1",
\ar@{-}"h";"l2",
\ar@{-}"i";"l3",
\ar@{-}"h5";"l5",
\ar@{-}"e";"h5",
\ar@{-}"e";"i",
\ar@{-}"e";"j",
\ar@{-}"j";"l",
\endxy$$
\caption{The planted planar tree with leaves in even levels $C_{1}(C_{n}(C_{1},\dots,C_{1}))$ for $n=5$.}
\label{c1cnc1}
\end{figure}

\begin{thm}\label{aux1}
The morphism $f'\colon\mathcal{O}\r\mathcal{P}$ is an operad morphism. Moreover, if $g'\colon\mathcal{F}(V)\r\mathcal{P}$ is the operad morphism adjoint to $\bar{g}'$, then the following diagram is a push-out in $\operad{\C V}$,
\begin{equation*}
\xymatrix{\mathcal{F}(U)\ar[d]_{g}\ar[r]^{\mathcal{F}(f)}&\mathcal{F}(V)\ar[d]^{g'}\\
\mathcal{O}\ar[r]_{f'}&\mathcal{P}}
\end{equation*}
\end{thm}

\begin{proof}
The morphism $f'$ is an operad morphism by the very definition of the operad structure in $\mathcal{P}$, since $c_i^{0,0}=\circ_i$ is the structure morphism of $\mathcal{O}$ and the unit of $\mathcal{P}$ is the composition of the unit of $\mathcal{O}$ and $f'$, see Lemma~\ref{pond} and \eqref{upush}.  Moreover, the square
$$\xymatrix{U(l)\ar[d]_{\bar{g}(l)}\ar[r]^-{f(l)}&V(l)\ar[d]^{\bar{g}'(l)}\\
\mathcal{O}(l)\ar[r]^-{f'(l)}&\mathcal{P}(l)}$$
commutes for all $l\geq 0$. In fact, the following diagram commutes by some trivial facts, including the very definition of $P_{1}(l)$ in Lemma \ref{pind},
$$
\xymatrix@!C=72pt{
&
U(l)
\ar[d]_\cong^{(\text{right unit})^{-1}}
\ar[ld]_{\bar{g}(l)}
\ar[rr]^-{f(l)}&&
V(l)\ar[d]^{\cong}_{(\text{right unit})^{-1}}\\
\mathcal{O}(l)
\ar[rd]^-\cong_-{(\text{right unit})^{-1}\quad}
\ar@/_40pt/[rddd]_{\id{}}
&
U(l)\otimes \unit^{\otimes (l+1)}
\ar[d]|{\bar{g}(l)\otimes\id{\unit}^{\otimes(l+1)}}
\ar[rd]^{\qquad\id{}\otimes{u}^{\otimes(l+1)}}
\ar[rr]^-{f(l)\otimes\id{\unit}^{(l+1)}}&&
V(l)\otimes \unit^{\otimes (l+1)}\ar[d]^{\id{}\otimes u^{\otimes (l+1)}}\\
&
\mathcal{O}(l)\otimes \unit^{\otimes (l+1)}
\ar[d]|{\id{\mathcal{O}(l)}\otimes u^{\otimes (l+1)}}
\ar@/_43pt/[dd]_<(.2){\text{right unit}}
&U(l)\otimes \mathcal{O}(1)^{\otimes (l+1)}
\ar[ld]^{\quad\bar{g}(l)\otimes\id{}^{\otimes(l+1)}}
\ar@(dr,ld)[r]_{f(l)\otimes\id{\mathcal{O}(1)}^{(l+1)}}
\ar@/^45pt/[ldd]^{\psi^{C_{1}(C_{l}(C_{1},\dots,C_{1}))}_1}
&
V(l)\otimes \mathcal{O}(1)^{\otimes (l+1)}
\ar[dd]_<(.7){\bar{\psi}_1^{C_{1}(C_{l}(C_{1},\dots,C_{1}))}}\\&
\mathcal{O}(l)\otimes \mathcal{O}(1)^{\otimes (l+1)}
\ar[d]^{\mathcal{O}(p_{C_{1}(C_{l}(C_{1},\dots,C_{1}))})}
&&
\\&
\mathcal{O}(l)\ar@(d,d)[rr]_-{\varphi_{1}(l)}&&P_1(l)}
$$
and its outer (commutative) square
$$\xymatrix{U(l)\ar[ddd]_{\bar{g}(l)}\ar[r]^{f(l)}&V(l)\ar[d]_{\cong}^{(\text{right unit})^{-1}}\\
&V(l)\otimes \unit^{\otimes (l+1)}\ar[d]^{\id{}\otimes u^{\otimes (l+1)}}\\
&V(l)\otimes \mathcal{O}(1)^{\otimes (l+1)}\ar[d]^{\bar{\psi}_1^{C_{1}(C_{l}(C_{1},\dots,C_{1}))}}\\
\mathcal{O}(l)\ar[r]^{\varphi_{1}(l)}&P_1(l)}$$
composed with the canonical morphism $P_1(l)\r\colim_{r\geq 0}P_r(l)=\mathcal{P}(l)$
yields the former square.

Suppose we are given an operad $\mathcal{P}'$ and morphisms $f''\colon \mathcal{O}\r \mathcal{P}'$ in $\operad{\C V}$ and $\bar{g}''\colon V\r \mathcal{P}'$ in $\C{V}^{\mathbb N}$ 
such that the square
\begin{equation*}\tag{a}
\xymatrix{U(l)\ar[d]_{\bar{g}(l)}\ar[r]^-{f(l)}&V(l)\ar[d]^{\bar{g}''(l)}\\
\mathcal{O}(l)\ar[r]^-{f''(l)}&\mathcal{P}'(l)}
\end{equation*}
commutes for all $l\geq 0$. 
We must show that there is a unique morphism $h\colon \mathcal{P}\r \mathcal{P}'$ in $\operad{\C V}$ such that $f''=hf'$ and $\bar{g}''=h\bar{g}'$ in $\C{V}^{\mathbb N}$. 

We define morphisms
$$h_r(l)\colon P_r(l)\To \mathcal{P}'$$
by induction on $r\geq0$ as follows. We set $h_0(l)=f''(l)$. Assume we have defined up to~$h_{r-1}(l)$. Then we define $h_r(l)$ so that $h_r(l)\varphi_r(l)=h_{r-1}(l)$ and, for any planted planar tree $T$ with $l$ leaves concentrated in even levels and $r$ inner vertices in even levels, 
\begin{equation*}\tag{b}
h_r(l)\bar{\psi}^T_r=\mathcal{P}'(p_{T})(\hspace{-6pt}\bigotimes\limits_{v\in I^{e}(T)}\hspace{-8pt} \bar{g}''(\val{T}{v})\otimes
\hspace{-7pt}\bigotimes\limits_{w\in I^o(T)}\hspace{-7pt}f''(\val{T}{w})).
\end{equation*}
The morphism $h_r(l)$ is well defined by the universal property of the push-out definition of $P_r(l)$ in Lemma \ref{pind} since, given $u\in I^e(T)$, 
\begin{align*}
&\hspace{-40pt}\mathcal{P}'(p_{T})(\hspace{-6pt}\bigotimes\limits_{v\in I^{e}(T)}\hspace{-8pt} \bar{g}''(\val{}{v})\otimes
\hspace{-7pt}\bigotimes\limits_{w\in I^o(T)}\hspace{-7pt}f''(\val{}{w}))(f(\val{}{u})\otimes\id{})\\
={}&
\mathcal{P}'(p_{T})(\hspace{-15pt}\bigotimes\limits_{v\in I^{e}(T)\setminus\{u\}}\hspace{-8pt} \bar{g}''(\val{}{v})\otimes
\hspace{-7pt}\bigotimes\limits_{w\in I^o(T)\cup\{u\}}\hspace{-7pt}f''(\val{}{w}))(\id{}\otimes\bar{g}(\val{}{u}))\\
={}&
\mathcal{P}'(p_{T/\star(u)})(\hspace{-17pt}\bigotimes\limits_{v\in I^{e}(T/\star(u))}\hspace{-15pt} \bar{g}''(\val{}{v})\otimes
\hspace{-20pt}\bigotimes\limits_{w\in I^o(T/\star(u))}\hspace{-20pt}f''(\val{}{w}))
(\id{}\otimes\mathcal{O}(p_{\overline{\star}(u)}))
(\id{}\otimes\bar{g}(\val{}{u}))\\
={}&h_{r-1}(l)\bar{\psi}^{T/\star(u)}_{r-1}(\id{}\otimes\mathcal{O}(p_{\overline{\star}(u)}))
(\id{}	\otimes\bar{g}(\val{}{u}))\\
={}&h_{r-1}(l){\psi}^{T}_{r,u}.
\end{align*}
Here, in the first equation we use the commutativity of (a), in the second equation we use the fact that $f''$ is an operad morphism, and in the third equation we use the induction hypothesis. The fourth equation follows from the very definition of $\psi^T_{r,u}$ in the statement of Lemma \ref{pind}. For simplicity, in these equations we have omitted some symmetry isomorphisms in $\C V$.

We have checked that the morphisms $h_r(l)$ induce a morphism of sequences $h\colon\mathcal{P}\r\mathcal{P'}$ in the colimit. It is clear that $hf'=f''$ since $h_{0}=f''$, in particular $h$ preserves units. Moreover, $hg'=g''$ since for $\bar{T}=C_{1}(C_{l}(C_{1},\dots,C_{1}))$, see Figure~\ref{c1cnc1}, if $v\in I^{e}(\bar T)$ is the unique inner vertex in even levels, then
\begin{align*}
h_{1}(l)\bar{\psi}^{\bar{T}}_{1}(\id{V(l)}\otimes u^{\otimes(l+1)})&=
\mathcal{P}'(p_{\bar T})(\bar{g}''(l)\otimes f''(1)^{\otimes(l+1)})(\id{V(l)}\otimes u^{\otimes(l+1)})\\
&=
\mathcal{P}'(p_{\bar T})(\bar{g}''(l)\otimes u^{\otimes(l+1)})\\
&=
\bar{g}''(l).
\end{align*}
Here, in the first equation we use (b), in the second equation we use that $f''$ is an operad morphism and therefore it preserves units, and in the third equation we use relations (3) and (4) in Remark \ref{circi} for the operad $\mathcal{P}'$.

In order to check that $h$ is indeed an operad morphism, we  show that
$$h_{r+s}(l+m-1)c^{r,s}_i(l,m)=h_r(l)\circ_i h_s(m)$$
We proceed by induction on $(r,s)\in\mathbb{N}^2$ with respect to the graded lexicographic order. This is obvious for $r=s=0$, since $f''$ is an operad morphism. If the equation holds up to the predecessor of $(r,s)$ then by induction hypothesis we only have to check that 
the following equation holds,
$$h_{r+s}(l+m-1)c^{r,s}_i(l,m)(\bar{\psi}^T_r\otimes\bar{\psi}^{T'}_s)=(h_r(l)\bar{\psi}^T_r)\circ_i (h_s(m)\bar{\psi}^{T'}_s),$$
for $T'$ a planted planar tree with $m$ leaves concentrated in even levels and $s$ inner vertices in even levels. Let $u\in I^o(T)$ be the inner vertex of the $i^{\text{th}}$ leaf edge of $T$, $u'\in I^o(T')$ the unique level $1$ vertex of $T'$, and $e=\{u,u'\}\in E(T\circ_{i}T')$. Suppose that the $i^{\text{th}}$ leaf edge of $T$ is the $k^{\text{th}}$ incomming edge of $u$. Then,
\begin{align*}
&\hspace{-40pt}h_{r+s}(l+m-1)c^{r,s}_i(l,m)(\bar{\psi}^T_r\otimes\bar{\psi}^{T'}_s)\\
={}&h_{r+s}(l+m-1)d^{r,s}_{i}(T,T')\\
={}&h_{r+s}(l+m-1)\bar{\psi}^{(T\circ_{i}T')/e}_{s+t}(\circ_{k}\otimes\id{})\\
={}&\mathcal{P}'(p_{(T\circ_{i} T')/e})(\hspace{-20pt}\bigotimes\limits_{v\in I^{e}((T\circ_{i} T')/e)}\hspace{-20pt} \bar{g}''(\val{}{v})\;\;\otimes
\hspace{-15pt}\bigotimes\limits_{w\in I^o((T\circ_{i} T')/e)}\hspace{-20pt}f''(\val{}{w}))(\circ_{k}\otimes\id{})\\
={}&\mathcal{P}'(p_{(T\circ_{i} T')/e})(\circ_{k}\otimes\id{})(\hspace{-15pt}\bigotimes\limits_{v\in I^{e}(T)\cup I^{e}(T')}\hspace{-15pt} \bar{g}''(\val{}{v})\otimes
\hspace{-15pt}\bigotimes\limits_{w\in I^o(T)\cup I^{o}(T')}\hspace{-15pt}f''(\val{}{w}))\\
={}&(\mathcal{P}'(p_{T})\circ_{i}\mathcal{P}'(p_{T'}))(\hspace{-15pt}\bigotimes\limits_{v\in I^{e}(T)\cup I^{e}(T')}\hspace{-15pt} \bar{g}''(\val{}{v})\otimes
\hspace{-15pt}\bigotimes\limits_{w\in I^o(T)\cup I^{o}(T')}\hspace{-15pt}f''(\val{}{w}))\\
={}&(h_r(l)\bar{\psi}^T_r)\circ_i (h_s(m)\bar{\psi}^{T'}_s).
\end{align*}
Here $\circ_{k}$ denotes  either $\circ_{k}\colon \mathcal{O}(\val{T}{u})\otimes \mathcal{O}(\val{T'}{u'})\r
\mathcal{O}(\val{(T\circ T')/e}{[e]})$ or $\circ_{k}\colon \mathcal{P}'(\val{T}{u})\otimes \mathcal{P}'(\val{T'}{u'})\r
\mathcal{P}'(\val{(T\circ T')/e}{[e]})$. 
Moreover,  in the first equation we use the inductive definition of $c^{r,s}_{i}(l,m)$ in Lemma~\ref{pond}, in the second equation we use the definition of $d_{i}^{r,s}(T,T')$ also in Lemma~\ref{pond}, in the third equation we use (b), in the fourth equation we use that $f''$ is an operad morphism, in the fifth equation we use the construction of operadic functors from operads in Proposition \ref{ofo}, and in the final equation we use (b) again. Furthermore, for simplicity we have  omitted some symmetry isomorphisms in~$\C V$ in these equations.

The uniqueness of $h$ follows from the fact that 
the morphism $\bar{\psi}^{T}_{r}$ defined in Lemma \ref{pind} is related to the operadic functor of $\mathcal{P}$ by the following equation,
$$\bar{\psi}^{T}_{r}=\mathcal{P}(p_{T})(\hspace{-6pt}\bigotimes\limits_{v\in I^{e}(T)}\hspace{-8pt} \bar{g}'(\val{T}{v})\otimes
\hspace{-7pt}\bigotimes\limits_{w\in I^o(T)}\hspace{-7pt}f'(\val{T}{w})).$$
Therefore, if $h'\colon \mathcal{P}\r \mathcal{P}'$is an  operad morphism satisfying $h'f'=f''$ and $h'\bar{g}'=\bar{g}''$, and if we denote $h'_{r}(l)$ the composition of $h'(l)$ with the canonical morphism to the colimit $P_{r}(l)\r\mathcal{P}(l)$, then the morphisms $h'_{r}(l)$ must satisfy $h'_{0}(l)=f''(l)$, and also (b) after replacing $h_{r}(m)$ with $h'_{r}(m)$, therefore $h'=h$ by the universal property of the push-outs $P_{r}(l)$ and the colimit $\mathcal{P}$.
\end{proof}

\section{Proof of Theorem \ref{elt}}

Assume in this section that $\C V$ is also a cofibrantly generated monoidal model category (see Definition \ref{defm}) satisfying the monoid axiom \cite[Definition 3.3]{ammmc}. In order to explain what this means, let us recall some terminology from \cite{hmc}. 

Given an ordinal $\lambda$, a directed diagram $X\colon\lambda\r\C{V}$ is \emph{continuous} if for any limit ordinal $\alpha<\lambda$, the canonical morphism $$\colim_{i<\alpha}X_{i}\To X_{\alpha}$$
is an isomorphism. The natural morphism from the first object to the colimit $$X_{0}\To\colim_{i<\lambda} X_{i}$$
is said to be the \emph{transfinite composition} of the morphisms in the diagram. We here do not exclude the possibility that $\lambda$ be finite.

Given a class of morphisms $K$ in $\C{V}$, a \emph{relative $K$-cell complex}  is a transfinite composition of morphisms $X\colon \lambda\r \C{V}$ such that for any $i<\lambda$ with $i+1<\lambda$ the morphism $X_{i}\r X_{i+1}$ fits into a push-out diagram as follows, where the top horizontal arrow is in $K$,
$$\xymatrix{A\ar[d]\ar@{}[rd]|{\text{push}}\ar[r]^-{\text{in }K}& B\ar[d]\\X_{i}\ar[r]&X_{i+1}}$$
A plain \emph{$K$-cell complex} is a relative $K$-cell complex with $X_{0}=0$ the initial object of $\C{V}$.

If $I$ and $J$ are sets of generating cofibrations and generating trivial cofibrations in~$\C V$, respectively, then the cofibrations in~$\C{V}$ are exactly the retracts of relative $I$-cell complexes, and the trivial cofibrations are the retracts of relative $J$-cell complexes. In particular the cofibrant objects in $\C{V}$ are the retracts of $I$-cell complexes. So far, nothing of this needs either the monoidal structure of $\C{V}$ or its model category structure. 

\begin{defn}\label{max}
The \emph{monoid axiom} for $\C{V}$ says that, for
$$K=\{f\otimes X\,;\,f\text{ is a trivial cofibration and }X\text{ is an object in }\C{V}\},$$
all relative $K$-cell complexes are weak equivalences.
\end{defn}

\begin{prop}
Consider a push-out diagram in $\operad{\C V}$ as follows,
 $$\xymatrix{\mathcal{F}(U)\ar[d]_{g}\ar[r]^{\mathcal{F}(f)}\ar@{}[rd]|{\text{push}}&\mathcal{F}(V)\ar[d]^{g'}\\
\mathcal{O}\ar[r]^-{f'}&\mathcal{P}}$$
If $f$ is a trivial cofibration then $f'(n)\colon \mathcal{O}(n)\r \mathcal{P}(n)$ is a relative $K$-cell complex, $n\geq 0$, where $K$ is the class in the previous definition.
\end{prop}

\begin{proof}
By the push-out product axiom  (Definition \ref{defm}), the morphism (\ref{monstruo}) in Lemma \ref{pind} is the tensor product of a trivial cofibration and an object in $\C V$, i.e. $\eqref{monstruo}\in K$. Therefore, by Theorem \ref{aux1}, $f'(n)$ is a relative $K$-cell complex.
\end{proof}

Consider  the associated sets of  generating cofibrations and  generating trivial cofibrations in $\C{V}^{\mathbb{N}}$, $I_{\mathbb{N}}$ and $J_{\mathbb{N}}$, respectively, see Remark~\ref{jn}.

\begin{cor}\label{ayyy}
If $\C V$ satisfies the monoid axiom, then a morphism in~$\C{V}^{\mathbb{N}}$ underlying a relative $\mathcal{F}(J_{\mathbb{N}})$-cell complex in $\operad{\C V}$ is a weak equivalence in~$\C{V}^{\mathbb{N}}$.
\end{cor}

Now we are ready to prove Theorem \ref{elt}.

\begin{proof}[Proof of Theorem \ref{elt}]
It is easy to see that operadic functors are the same as algebras over the monad associated to free operad adjunction in  Section \ref{ropo}. Therefore, using  the equivalence between   operads and operadic functors in Proposition \ref{ofo}, one can easily show that the natural comparison functor from operads to algebras over the  the free operad monad is an equivalence of categories. Moreover, this monad preserves filtered colimits, see  the explicit construction in  Section \ref{ropo}, therefore the category $\operad{\C V}$ is complete and cocomplete \cite[Proposition 4.3.6]{borceux2}. Furthermore, the forgetful functor $\operad{\C V}\r\C{V}^{\mathbb{N}}$ also preserves filtered colimits \cite[Proposition 4.3.2]{borceux2}, in particular, since $\mathcal{F}$ is a left adjoint and sources of morphisms in $I$ and $J$ are presentable in $\C V$, then sources of morphisms in $\mathcal{F}(I_{\mathbb{N}})$ and $\mathcal{F}(J_{\mathbb{N}})$ are presentable in~$\operad{\C{V}}$.

We can apply \cite[Lemma 2.3]{ammmc} in order to prove the existence of the claimed model structure in $\operad{\C V}$. The smallness condition has already been checked, and condition (1) of \cite[Lemma 2.3]{ammmc} has been established in Corollary \ref{ayyy}. 

Recall that a model category is right proper if the pull-back of a weak equivalence along a fibration is again a weak equivalence \cite[Definition 13.1.1 (2)]{hirschhorn}. The statement about right properness is obvious since fibrations and weak equivalences in $\operad{\C{V}}$ are detected by the forgetful functor $\operad{\C{V}}\r\C{V}^{\mathbb{N}}$, and this functor is a right adjoint, so it preserves all limits, in particular pull-backs.

Recall also that a model category is combinatorial if it is cofibrantly generated and locally presentable. If $ {\C V}$ is combinatorial then $\operad{\C V}$ is locally presentable by \cite[2.3 (1) and the Theorem in 2.78]{adamekrosicky}, hence it is combinatorial. 
\end{proof}

\section{Algebras}

Assume we have a strong braided monoidal functor $\C{V}\r Z(\C{C})$, where $Z(\C{C})$ is the center of $\C{C}$, defined in \cite{tybotc}. Such a functor consists of an ordinary  functor $$z\colon\C{V}\To\C{C},$$  
together with natural isomorphisms,
\begin{align*}
\text{multiplication}\colon z(X)\otimes z(X')&\st{}\To z(X\otimes X'),\\
\text{unit}\colon \unit_{\C C}&\st{}\To z(\unit_{\C V}),\\
\zeta(X,Y)\colon z(X)\otimes Y &\st{}\To Y\otimes z(X),
\end{align*}
such that the multiplication and the unit satisfy well-known coherence laws \cite[Definition 6.4.1]{borceux2} and  the following three diagrams of isomorphisms commute,
$$
\begin{array}{c}\xymatrix@C=40pt{z(X)\otimes z(X')\ar[r]^{\zeta(X,z(X'))}_{{}}\ar[d]_{\text{mult.}}^{}&z(X')\otimes z(X)\ar[d]^{\text{mult.}}_{}\\
z(X\otimes X')\ar[r]_{z(\text{sym.})}^{{}}&z(X'\otimes X)}\end{array}\qquad
\begin{array}{c}\xy
(0,0)*+{\unit_{\C C}\otimes Y}="a",
(20,-10)*+{z(\unit_{\C V})\otimes Y}="c",
(-20,-10)*+{Y}="b",
(12,-25)*+{Y\otimes z(\unit_{\C V})}="e",
(-12,-25)*+{Y\otimes \unit_{\C C}}="d",
(-1,-28)*{\scriptstyle Y\otimes\,\text{unit}},
\ar"a";"b"_{\text{left unit}\quad}^{{}}
\ar"a";"c"^{\quad\text{unit}\,\otimes Y}_{}
\ar"b";"d"^-{(\text{right unit})^{-1}}^{{}}
\ar"c";"e"_{{}}^{\zeta(\unit_{\C V},Y)}
\ar"d";"e"^{}
\endxy
\end{array}
$$

$$\xy/r2.3pt/:
(-15,0)*+{z(X)\otimes (z(X')\otimes Y)}="a",
(35,0)*+{(z(X)\otimes z(X'))\otimes Y}="aa",
(-30,-20)*+{z(X)\otimes (Y\otimes z(X'))}="b",
(50,-20)*+{z(X\otimes X')\otimes Y}="c",
(-30,-40)*+{(z(X)\otimes Y)\otimes z(X')}="bb",
(50,-40)*+{Y\otimes z(X\otimes X')}="e",
(-15,-60)*+{(Y\otimes z(X))\otimes z(X')}="d",
(35,-60)*+{Y\otimes(z(X)\otimes z(X'))}="dd",
\ar"a";"b"_-{z(X)\otimes\zeta(X',Y)}^{{}}
\ar"aa";"c"^-{\text{mult.}\otimes Y}	
\ar"bb";"d"_-{\zeta(X,Y)\otimes z(X')}^{{}}
\ar"c";"e"_{{}}^-{\zeta(X\otimes X',Y)}
\ar"dd";"e"_-{\scriptstyle Y\otimes\text{mult.}}
\ar@{<-}"a";"aa"_{{}}^-{\;\,\text{assoc.}}
\ar"b";"bb"^{{}}_-{\text{assoc.}}
\ar"d";"dd"^{{}}_-{\text{assoc.}}
\endxy$$

Moreover, suppose that the functor $z(-)\otimes Y\colon\C{V}\r \C{C}$
has a right adjoint, 
\begin{align*}
\hom_{\C C}(Y,-)\colon&\C{C}\To \C{V}.
\end{align*}
We will use the \emph{evaluation morphism},
$$\text{evaluation}\colon z(\hom_{\C C}(Y,Z))\otimes Y\To Z,$$
which is the adjoint of the identity in $\hom_{\C C}(Y,Z)$.

\begin{defn}\label{eop}
The \emph{endomorphism operad} of an object $Y$ in $\C C$ is the non-symmetric operad $\texttt{End}_{\C C}(Y)$ in $\C V$ with 
$$\texttt{End}_{\C C}(Y)(n)=\hom_{\C C}(Y\otimes\st{n}\cdots\otimes Y,Y).$$
The unit $$u\colon \unit_{\C{V}}\To \texttt{End}_{\C C}(Y)(1)$$ is the adjoint of
$$\xymatrix@C=40pt{z(\unit_{\C{V}})\otimes Y\ar[r]^-{\text{unit}^{-1}\otimes Y}&\unit_{\C{C}}\otimes Y\ar[r]^-{\text{left unit}}&Y.}$$
The composition laws, $1\leq i\leq m$, $n\geq 0$,
$$\circ_i\colon\texttt{End}_{\C C}(Y)(m)\otimes \texttt{End}_{\C C}(Y)(n)\To \texttt{End}_{\C C}(Y)(m+n-1)$$
are the adjoints of
$$\xymatrix{z(\hom_{\C C}(Y^{\otimes m},Y)\otimes \hom_{\C C}(Y^{\otimes n},Y))\otimes Y^{\otimes (m+n-1)}
\ar[d]_{\text{mult.}^{-1}\,\otimes \id{} 
}^\cong\\
z(\hom_{\C C}(Y^{\otimes m},Y))\otimes z(\hom_{\C C}(Y^{\otimes n},Y))\otimes Y^{\otimes (m+n-1)}
\ar[d]^\cong_{
\id{}\otimes\zeta(\hom_{\C C}(Y^{\otimes n},Y),Y^{\otimes (i-1)})\otimes 
\id{}
}\\
z(\hom_{\C C}(Y^{\otimes m},Y))\otimes Y^{\otimes (i-1)}\otimes z(\hom_{\C C}(Y^{\otimes n},Y))\otimes Y^{\otimes n}\otimes Y^{\otimes (m-i)}
\ar[d]_{
\id{}\otimes \,\text{evaluation}\,\otimes \id{}
}
\\
z(\hom_{\C C}(Y^{\otimes m},Y))\otimes Y^{\otimes (i-1)}\otimes Y\otimes Y^{\otimes (m-i)}\ar@{=}[d]\\
z(\hom_{\C C}(Y^{\otimes m},Y))\otimes Y^{\otimes m}
\ar[d]_{\text{evaluation}}\\
Y}$$
Here we have omitted some obvious associativity isomorphisms in $\C C$. 

Given a non-symmetric operad $\mathcal{O}$ in $\C V$, an \emph{$\mathcal{O}$-algebra} in $\C C$ is an object $Y$ in $\C C$ together with an operad morphism $\mathcal{O}\r \texttt{End}_{\C C}(Y)$. 

Equivalently, an $\mathcal{O}$-algebra structure on $Y$ is given by morphisms in~$\C C$, $n\geq 0$,
\begin{equation}\label{refi}
\nu_n\colon z(\mathcal{O}(n))\otimes Y^{\otimes n}\To Y,
\end{equation}
such that the following diagrams commute, $1\leq i\leq m$, $n\geq 0$,
$$\xymatrix@C=35pt{
z(\unit_{\C V})\otimes Y\ar[r]^-{z(u)\otimes \id{}}\ar@{<-}[d]_-{\text{unit}\,\otimes\, \id{}}^-{\cong}&z(\mathcal{O}(1))\otimes Y\ar[d]^-{\nu_{1}}\\
\unit_{\C C}\otimes Y\ar[r]_-{\text{left unit}}^-{\cong}&Y\\
}$$
$$\xymatrix@C=-70pt{&z(\mathcal{O}(m))\otimes z(\mathcal{O}(n))\otimes Y^{\otimes (m+n-1)}\ar[rd]^-{\quad\text{mult.}\,\otimes \id{}
}_-{\cong}
\ar[ld]_{
\id{}\otimes \zeta(\mathcal{O}(n),Y^{\otimes (i-1)})\otimes \id{}
\qquad\qquad
}^-{\cong}
&\\
z(\mathcal{O}(m))\otimes Y^{\otimes (i-1)}\otimes z(\mathcal{O}(n))\otimes Y^{\otimes n}\otimes Y^{\otimes (m-i)}
\ar[d]_{
\id{}\otimes\nu_n\otimes \id{} 
}
&&z(\mathcal{O}(m)\otimes \mathcal{O}(n))\otimes Y^{\otimes (m+n-1)}\ar[d]^{z(\circ_i)\otimes \id{}}\\
z(\mathcal{O}(m))\otimes Y^{\otimes m}\ar[rd]_-{\nu_m}&&z(\mathcal{O}(m+n-1))\otimes Y^{\otimes (m+n-1)}\ar[ld]^{\quad\nu_{m+n-1}}\\
&Y&}$$
An \emph{$\mathcal{O}$-algebra morphism} $f\colon Y\r Z$ is a morphism in $\C C$ such that the following squares commute, $n\geq 0$,
$$\xymatrix{z(\mathcal{O}(n))\otimes Y^{\otimes n}\ar[r]^-{\nu_n^Y} \ar[d]_{
\id{}\otimes f^{\otimes n}}&Y\ar[d]^f\\
z(\mathcal{O}(n))\otimes Z^{\otimes n}\ar[r]_-{\nu_n^Z} &Z}$$
The category of $\mathcal{O}$-algebras in $\C C$ will be denoted by $\algebra{\mathcal{O}}{\C C}$.
\end{defn}

\begin{rem}\label{alal}
The initial $\mathcal{O}$-algebra in $\C C$ is $z(\mathcal{O}(0))$ with structure morphisms
$$\xy
(0,0)*+{\nu_{n}\colon z(\mathcal{O}(n))\otimes z(\mathcal{O}(0))^{\otimes n}}="a",
(50,0)*+{z(\mathcal{O}(n)\otimes \mathcal{O}(0)^{\otimes n})}="b",
(90,0)*+{z(\mathcal{O}(0)).}="c"
\ar"a";"b"^-{\text{mult.}}_-{\cong}
\ar"b";"c"^-{z(\mu_{n;0,\st{n}{\dots},0})}
\endxy$$
Here we use the convention $\mu_{0;\emptyset}=\id{\mathcal{O}(0)}$. If $A$ is an $\mathcal{O}$-algebra, the structure morphism $\nu_{0}^{A}\colon z(\mathcal{O}(0))\r A$ is the unique morphism of $\mathcal{O}$-algebras $z(\mathcal{O}(0))\r A$. The final $\mathcal{O}$-algebra in $\C C$ is the final object of $\C C$ endowed with the only possible $\mathcal{O}$-algebra structure.
\end{rem}

\section{The relevant algebra push-out}\label{rapo}

Assume we are in the same circumstances as in the previous section. Let $\mathcal{O}$ be a non-symmetric operad in $\C V$. 
The functor $\algebra{\mathcal{O}}{\C C}\r\C{C}$ forgetting the $\mathcal{O}$-algebra structure has a left adjoint $$\mathcal{F}_{\mathcal{O}}\colon\C{C}\To\algebra{\mathcal{O}}{\C C},$$ the \emph{free $\mathcal{O}$-algebra} functor, explicitly defined as
$$\mathcal{F}_{\mathcal{O}}(Y)=\coprod_{p\geq 0}z(\mathcal{O}(p))\otimes Y^{\otimes p}.$$
The action of $\mathcal{O}$ on $\mathcal{F}_{\mathcal{O}}(Y)$,
$$\xy
(0,0)*{
\begin{aligned}
z(\mathcal{O}(n))\otimes \mathcal{F}_{\mathcal{O}}(Y)^{\otimes n}&=
z(\mathcal{O}(n))\otimes\bigotimes_{i=1}^{n}\left(\coprod_{p_{i}\geq 0}z(\mathcal{O}(p_{i}))\otimes Y^{\otimes p_{i}}\right)\\
&\cong \hspace{-10pt}
\coprod_{p_{1},\dots,p_{n}\geq 0}\hspace{-10pt}z(\mathcal{O}(n))\otimes\bigotimes_{i=1}^{n}\left(z(\mathcal{O}(p_{i}))\otimes Y^{\otimes p_{i}}\right)\\
&\cong \hspace{-10pt}
\coprod_{p_{1},\dots,p_{n}\geq 0}\hspace{-10pt}z(\mathcal{O}(n))\otimes z(\mathcal{O}(p_{1}))
\otimes\cdots\otimes z(\mathcal{O}(p_{n}))
\otimes Y^{\otimes \sum_{i=1}^{n}p_{i}}\\
&\cong \hspace{-10pt}
\coprod_{p_{1},\dots,p_{n}\geq 0}\hspace{-10pt}z\left(\mathcal{O}(n)\otimes \mathcal{O}(p_{1})
\otimes\cdots\mathcal{O}(p_{n})\right)
\otimes Y^{\otimes \sum_{i=1}^{n}p_{i}}
\\&\\&\\&\\
\mathcal{F}_{\mathcal{O}}(Y)&=\coprod_{p\geq 0}z(\mathcal{O}(p))\otimes Y^{\otimes p},
\end{aligned}}
\ar@(d,lu)(-45.3,26);(-40,-28)_{\nu_{n}}
\endxy$$
is defined as the morphism which sends the factor $(p_{1},\dots,p_{n})\in\mathbb{N}^{n}$ in the source to the factor $p=p_{1}+\cdots+p_{n}\in\mathbb{N}$ in the target via $z(\mu_{n;p_{1},\dots,p_{n}})\otimes\id{}$, $n\geq 1$. For $n=0$, the morphism $\nu_{0}\colon z(\mathcal{O}(0))\r\mathcal{F}_{\mathcal{O}}(Y)$ is the inclusion of the factor $p=0$ of the coproduct.

The unit of the adjunction is the following composite morphism in $\C C$,
$$\xy
(0,0)*+{Y}="a",
(23,0)*+{\unit_{\C C}\otimes Y}="b",
(50,0)*+{z(\unit_{\C V})\otimes Y}="c",
(80,0)*+{z(\mathcal{O}(1))\otimes Y}="d",
(115,0)*+{\mathcal{F}_{\mathcal{O}}(Y).}="e",
\ar"a";"b"_-\cong^-{(\text{left unit})^{-1}}
\ar"b";"c"_-\cong^-{\text{unit}\otimes\id{}}
\ar"c";"d"^-{z(u)\otimes\id{}}
\ar"d";"e"^-{\begin{array}{c}
\scriptstyle\text{inclusion of}\vspace{-5pt} \\
\scriptstyle\text{the factor }p=1
\end{array}}
\endxy$$
Moreover, given an $\mathcal{O}$-algebra $A$, the counit of the adjunction is defined by the multiplication morphisms in \eqref{refi},
$$(\nu_p)_{p\geq 0}\colon \mathcal{F}_{\mathcal{O}}(A)\To A.$$

In this section we give an explicit construction of the push-out of a diagram in $\algebra{\mathcal{O}}{\C C}$ as follows,
\begin{equation}\label{po2}
\xymatrix{\mathcal{F}_{\mathcal{O}}(Y)\ar[d]_{g}\ar[r]^{\mathcal{F}_{\mathcal{O}}(f)}&\mathcal{F}_{\mathcal{O}}(Z)\\
A&}
\end{equation}

Consider the adjoint diagram in ${\C C}$,
$$\xymatrix{Y\ar[d]_{\bar g}\ar[r]^{f}&Z\\
A&}$$
The push-out of \eqref{po2} is an $\mathcal{O}$-algebra $B$ together with morphisms $f'\colon A\r B$ in $\algebra{\mathcal{O}}{\C C}$ and $\bar{g}'\colon Z\r B$ in $\C{C}$ such that
$f'\bar{g}=\bar{g}'f$ in $\C{C}$. Moreover,  given an $\mathcal{O}$-algebra $B'$ and morphisms $f''\colon A\r B'$ in $\algebra{\mathcal{O}}{\C C}$ and $\bar{g}''\colon Z\r \mathcal{P}'$ in $\C{C}$ with $f''\bar{g}=\bar{g}''f$ in~$\C{C}$, 
there is a unique morphism $h\colon B\r B'$ in $\algebra{\mathcal{O}}{\C C}$ such that $f''=hf'$ and $\bar{g}''=h\bar{g}'$ in $\C{C}$. 

The following lemma allows an inductive definition of the push-out \eqref{po2} as an object in $\C C$. We omit  proofs in this section since the results are simpler analogs of those in Section  \ref{ropo}, and the proofs follow very much the same steps.

\begin{lem}\label{pind2}
There is a sequence  in $\C{C}$,
$$A=B_0 \st{\varphi_1}\To B_1 \r\cdots\r B_{t-1} \st{\varphi_t}\To B_t \r\cdots 
,$$
where the morphism $\varphi_t$ is the push-out of 
\begin{equation}\label{killo}
\coprod\limits_{n\geq 1}\coprod\limits_{
\begin{array}{c}
\scriptstyle S\subset\{1,\dots,n\}\vspace{-3pt}\\
\scriptstyle \card(S)=t
\end{array}
}
z(\mathcal{O}(n))\otimes k_1^S\odot\cdots\odot k_n^S;\qquad
k_i^S=
\left\{
\begin{array}{ll}
 f,&i\in S;\\
0\r A,& i\notin S;
\end{array}
\right.
\end{equation}
along the unique morphism,
\begin{equation}\label{killo2}
(\psi_t^{n,S})_{n,S}\colon \coprod\limits_{n\geq 1}\coprod\limits_{
\begin{array}{c}
\scriptstyle S\subset\{1,\dots,n\}\vspace{-3pt}\\
\scriptstyle \card(S)=t
\end{array}
}\hspace{-15pt}
z(\mathcal{O}(n))\otimes s(k_1^S\odot\cdots\odot k_n^S)\To B_{t-1},
\end{equation}
such that for $t=1$ and $1\leq i\leq n$, 
$$\psi^{n,\{i\}}_1=\nu_n(\id{z(\mathcal{O}(n))}\otimes \id{}^{\otimes (i-1)}\otimes\bar{g}\otimes\id{}^{\otimes(n-i)}),$$
and 
for $t>1$ and $i\in S$,
$$\psi^{n,S}_t(\id{z(\mathcal{O}(n))}\otimes\kappa_i)=\bar{\psi}^{n,S\setminus\{i\}}_{t-1}
(\id{z(\mathcal{O}(n))}\otimes \id{}^{\otimes (i-1)}\otimes\bar{g}\otimes\id{}^{\otimes(n-i)}).$$
Here $(\bar{\psi}^{n,S'}_{t-1})_{n,S'}$ denotes the push-out of $({\psi}^{n,S'}_{t-1})_{n,S'}$, i.e. \eqref{killo2} for $t-1$, along   \eqref{killo}.
\end{lem}

We now endow $$B=\colim_{t\geq 0} B_t$$ with an $\mathcal{O}$-algebra structure.

\begin{lem}\label{pond2}
There are unique morphisms in $\C C$,
$$c_n^{t_1,\dots,t_n}\colon z(\mathcal{O}(n))\otimes B_{t_1}\otimes\cdots\otimes B_{t_n}\To B_{t_1 +\cdots +t_n},\qquad n\geq 1,\; t_i\geq 0,$$
such that
$$c^{0,\st{n}\dots,0}_n=\nu_n\colon z(\mathcal{O}(n))\otimes A^{\otimes n}\To A,$$
and, with the convention $\bar{\psi}_{0}^{p_{i},S_{i}}=\nu_{p_{i}}$, if $S_i\subset\{1,\dots, p_i\}$ is a subset of cardinality $\card S_i=t_i$, $1\leq i\leq n$, then 
$$\begin{array}{l}
c_n^{t_1,\dots,t_i,\dots,t_n}(\id{}^{\otimes (i-1)}\otimes\varphi_{t_i}\otimes\id{}^{\otimes(n-i)})=
\varphi_{t_1+\cdots +t_n}c_n^{t_1,\dots,t_i-1,\dots,t_n},\\{}\\
c_n^{t_1,\dots,t_n}(\bar{\psi}_{t_1}^{p_1,S_1}\otimes\cdots\otimes \bar{\psi}_{t_n}^{p_n,S_n})\\
\quad\qquad \qquad \qquad=
\bar{\psi}_{t_1+\cdots+t_n}^{p_1+\cdots+p_n,\bigcup_{i=1}^n(S_i+(p_1+\cdots+p_{i-1}))}\left(z(\mu_{n;p_{1},\dots,p_{n}})\otimes\id{}^{\otimes\sum_{i=1}^{n}p_{i}}\right).
\end{array}$$
Here $S+p=\{i+p\,;\, i\in S\}$ and $\mu$ is the multiplication of the operad $\mathcal{O}$. For simplicity, in these equations we have omitted some obvious structure isomorphisms of $\C V$, $\C C$ and $z$.
\end{lem}

We define $$f'\colon A=B_{0}\To \colim_{t\geq 0}B_{t}=B$$
as the canonical morphism to the colimit. Moreover, 
for $n\geq 1$ we define
$$\nu^{B}_{n}\colon z(\mathcal{O}(n))\otimes B^{\otimes n}\To B$$
as the colimit of the morphisms $c_{n}^{t_{1},\dots,t_{n}}$ in the previous lemma, $t_{i}\geq 0$, and for $n=0$,
$$\nu_{0}^{B}\colon z(\mathcal{O}(0))\st{\nu_{0}^{A}}\To A\st{f'}\To B.$$
Furthermore,  we define $\bar{g}'\colon Z\r B$ as the following composite morphism
$$\hspace{-3pt}\xy
(0,0)*+{Z}="a",
(23,0)*+{\unit_{\C C}\otimes Z}="b",
(50,0)*+{z(\unit_{\C V})\otimes Z}="c",
(80,0)*+{z(\mathcal{O}(1))\otimes Z}="d",
(102,0)*+{B_{1}}="e",
(122,0)*+{B.}="f",
\ar"a";"b"_-\cong^-{(\text{left unit})^{-1}}
\ar"b";"c"_-\cong^-{\text{unit}\otimes\id{}}
\ar"c";"d"^-{z(u)\otimes\id{}}
\ar"d";"e"^-{\bar{\psi}_{1}^{1,\{1\}}}
\ar"e";"f"^-{\begin{array}{c}
\scriptstyle\text{projection to}\vspace{-5pt} \\
\scriptstyle\text{the colimit}
\end{array}}
\endxy$$

\begin{thm}\label{aux2}
The morphisms $\nu_{n}^{B}$, $n\geq 0$, define an $\mathcal{O}$-algebra structure on $B$, $f'\colon A\r B$ is an $\mathcal{O}$-algebra  morphism, and if $g'\colon\mathcal{F}_{\mathcal{O}}(Z)\r B$ is the adjoint of $\bar{g}'\colon Z\r B$, then the following square is a push-out in $\algebra{\mathcal{O}}{\C C}$,
$$\xymatrix@C=30pt{\mathcal{F}_{\mathcal{O}}(Y)\ar[d]_{g}\ar[r]^{\mathcal{F}_{\mathcal{O}}(f)}&\mathcal{F}_{\mathcal{O}}(Z)\ar[d]^{g'}\\
A\ar[r]^-{f'}&B}$$
\end{thm}

\section{Proof of Theorems \ref{elt2} and \ref{elt3}}

Suppose that we are in the same conditions as in the two previous sections. Assume also that $\C V$ and $\C C$ are  monoidal model categories (see Definition \ref{defm}) and that the composite functor $$\C{V}\st{z}\To Z(\C{C})\st{\text{forget}}\To\C{C}$$ is a left Quillen functor \cite[Definition 1.3.1]{hmc}. We will need a non-symmetric version of the monoid axiom in Definition \ref{max}.

\begin{defn}\label{nsmax}
The \emph{monoid axiom} for $\C{C}$ says that, for
\begin{align*}
K'={}&\{f_{1}\odot\cdots \odot f_{n}\;;\;n\geq 1,\, S\subset\{1,\dots,n\}\text{ is a subset with }\card S\geq 1, \,
\\&
\qquad\qquad \qquad\quad f_{i}\text{ is a trivial cofibration if }i\in S,\\
& \qquad\qquad \qquad\quad f_{i}\colon 0\r X_{i} \text{ for some object }X_{i}\text{ in }\C{C}\text{ if }i\notin S\},
\end{align*}
all relative $K'$-cell complexes are weak equivalences.
\end{defn}

Notice that, as a consequence of the push-out product axiom, this is indeed equivalent to the monoid axiom in Definition \ref{max} when $\C{C}$ is symmetric. In any case, if all objects in $\C{C}$ are cofibrant then the monoid axiom  is a consequence of the push-out product axiom.

Suppose from now on that $\C C$ satisfies the monoid axiom and is cofibrantly generated with sets of generating cofibrations and generating trivial cofibrations $I$ and~$J$, respectively, with presentable sources.

\begin{prop}\label{1y2}
Consider a push-out diagram in $\algebra{\mathcal{O}}{\C C}$ as follows.
 $$\xymatrix{\mathcal{F}_{\mathcal{O}}(Y)\ar[d]_{g}\ar[r]^{\mathcal{F}_{\mathcal{O}}(f)}\ar@{}[rd]|{\text{push}}&\mathcal{F}_{\mathcal{O}}(Z)\ar[d]^{g'}\\
A\ar[r]^-{f'}&B}$$
\begin{enumerate}
\item If $f$ is a trivial cofibration in $\C C$, then the underlying morphism $f'\colon A\r B$ in $\C C$ is a relative $K'$-cell complex, where $K'$ is the class in Definition \ref{nsmax}.

\item Suppose $A$ is cofibrant in $\C C$, $f$ is a cofibration in $\C C$, and  $\mathcal{O}(n)$ is cofibrant in $\C V$, $n\geq 0$. Then the morphism $f'\colon A\r B$ is a cofibration in $\C C$, in particular $B$ is cofibrant in $\C C$.
\end{enumerate}
\end{prop}

\begin{proof}
In case (1), the morphism \eqref{killo} in Lemma \ref{pind2} is in $K'$, hence (1) follows from Theorem \ref{aux2}.

In case (2), since~$z$ is a left Quillen functor,  the objects $z(\mathcal{O}(n))$ are cofibrant in~$\C C$.  Therefore, by the push-out product axiom (Definition \ref{defm}) the morphism \eqref{killo} is a cofibration in $\C C$. Furthermore, by Theorem \ref{aux2} the morphism  $f'\colon A\r B$ is a transfinite composition of cofibrations in~$\C C$, hence  a cofibration in $\C C$ itself  \cite[Proposition 10.3.4]{hirschhorn}.
\end{proof}

As an immediate consequence of (1)  here and the monoid axiom, we obtain the following result.

\begin{cor}\label{ayyy2}
A morphism in~$\C{C}$ underlying a relative $\mathcal{F}_{\mathcal{O}}(J)$-cell complex in $\algebra{\mathcal{O}}{\C C}$ is a weak equivalence in~$\C{C}$.
\end{cor}

Now we are ready to prove Theorem \ref{elt2}.

\begin{proof}[Proof of Theorem \ref{elt2}]
Using the explicit description of the free operad adjunction at the beginning of Section \ref{rapo}, it is easy to see that $\mathcal{O}$-algebras are the same thing as algebras over the monad associated to the free $\mathcal{O}$-algebra adjunction. Moreover, this monad preserves filtered colimits, see again the explicit construction, therefore the category $\algebra{\mathcal{O}}{\C C}$ is complete and cocomplete \cite[Proposition 4.3.6]{borceux2}. Furthermore, 
the forgetful functor $\algebra{\mathcal{O}}{\C C}\r\C{C}$ also preserves filtered colimits \cite[Proposition 4.3.2]{borceux2}, in particular, since $\mathcal{F}_{\mathcal{O}}$ is a left adjoint and sources of morphisms in $I$ and $J$ are presentable in $\C{C}$, then sources of morphisms in $\mathcal{F}_{\mathcal{O}}(I)$ and $\mathcal{F}_{\mathcal{O}}(J)$ are presentable in~$\algebra{\mathcal{O}}{\C C}$.


We can apply \cite[Lemma 2.3]{ammmc} in order to prove the existence of the claimed model structure in $\algebra{\mathcal{O}}{\C C}$.  The smallness condition has already been checked, and condition (1) of \cite[Lemma 2.3]{ammmc} has been established in Corollary \ref{ayyy2}. 

The statement about right properness is obvious since fibrations and weak equivalences in $\algebra{\mathcal{O}}{\C C}$ are detected by the forgetful functor $\algebra{\mathcal{O}}{\C C}\r\C{C}$, and this functor is a right adjoint, so it preserves all limits, in particular pull-backs.

If $ {\C C}$ is combinatorial then $\algebra{\mathcal{O}}{\C C}$ is locally presentable by \cite[2.3 (1) and the Theorem in 2.78]{adamekrosicky}, hence it is combinatorial. 
\end{proof}

\begin{lem}\label{cofc}
Suppose that $\mathcal{O}$ is an operad in $\C V$ with $\mathcal{O}(n)$ cofibrant  for all $n\geq 0$. 
Then any cofibrant $\mathcal{O}$-algebra is also cofibrant as an object in $\C{C}$. 
\end{lem}

\begin{proof}
Cofibrant $\mathcal{O}$-algebras are retracts of $\mathcal{F}_{\mathcal{O}}(I)$-cell complexes, and  cofibrant objects in $\C C$ are closed under retracts, so it is enough to check that  $\mathcal{F}_{\mathcal{O}}(I)$-cell complexes are cofibrant in $\C C$. The initial $\mathcal{O}$-algebra in $\C C$ (see Remark \ref{alal}) is cofibrant in $\C C$, since $\mathcal{O}(0)$ is cofibrant in $\C V$ and $z$ is a left Quillen functor. Using Proposition \ref{1y2} (2), an induction argument proves that any  $\mathcal{F}_{\mathcal{O}}(I)$-cell complex is cofibrant in $\C C$.

\end{proof}

\begin{cor}\label{ayyy3}
Let $\mathcal{O}$ be an operad in $\C V$ with $\mathcal{O}(n)$ cofibrant  for all $n\geq 0$. 
Then, the forgetful functor $\algebra{\mathcal{O}}{\C C}\r \C{C}$ preserves cofibrations with cofibrant source.
\end{cor}

\begin{proof}
This is an immediate consequence of Lemma \ref{cofc} and Proposition \ref{1y2} (2), since cofibrations in $\algebra{\mathcal{O}}{\C C}$ are retracts of relative $\mathcal{F}_{\mathcal{O}}(I)$-cell complexes, the forgetful functor preserves filtered colimits, and cofibrations in $\C C$ are closed under transfinite compositions and retracts.
\end{proof}

\begin{lem}\label{cofq}
Under the hypotheses of Theorem \ref{elt3}, suppose that we have a push-out diagram in $\algebra{\mathcal{O}}{\C C}$, 
$$\xymatrix@C=35pt{\mathcal{F}_{\mathcal{O}}(Y)\ar[d]_{g}\ar@{ >->}[r]^{\mathcal{F}_{\mathcal{O}}(f)}\ar@{}[rd]|{\text{push}}&\mathcal{F}_{\mathcal{O}}(Z)\ar[d]^{g'}\\
A\ar@{ >->}[r]^-{f'}&B}$$
where $f$ is a cofibration in $\C C$ and $A$ is a cofibrant $\mathcal{O}$-algebra. 
If the unit of the adjunction evaluated at $A$ is a weak equivalence $\eta_A\colon A\st{\sim}\r \phi^*\phi_{*}A$, then it is also a weak equivalence  when evaluated at $B$, $\eta_B\colon B\st{\sim}\r \phi^*\phi_{*}B$. 
\end{lem}

\begin{proof}
Since $\phi_{*}$ is left adjoint to $\phi^{*}$, which is the identity on the underlying object in~$\C C$,  there is a natural isomorphism $\phi_{*}\mathcal{F}_{\mathcal{O}}\cong \mathcal{F}_{\mathcal{P}}$ that we regard as an identification, and  the morphism $\phi_{*}(f')$ fits into the following push-out diagram in $\algebra{\mathcal{P}}{\C C}$, 
$$\xymatrix@C=35pt{\mathcal{F}_{\mathcal{P}}(Y)\ar[d]_{ \phi_{*}(g)}\ar@{ >->}[r]^{\mathcal{F}_{\mathcal{P}}(f)}\ar@{}[rd]|{\text{push}}&\mathcal{F}_{\mathcal{P}}(Z)\ar[d]^{ \phi_{*}(g')}\\
\phi_{*}A\ar@{ >->}[r]^-{ \phi_{*}(f')}&\phi_{*}B}$$
The $\mathcal{O}$-algebra $A$ is cofibrant and $\phi_{*}$ is a left Quillen functor, therefore $\phi_{*}A$ is a cofibrant $\mathcal{P}$-algebra, in particular, both $A$ and $\phi_{*}A$ are cofibrant in $\C C$ by Lemma~\ref{cofc}. Notice that the underlying object of $\phi_{*}A$ and $\phi^{*}\phi_{*}A$ in $\C C$ is the same.

Let us call $C=\phi^{*}\phi_{*}B$. By Lemma \ref{pind2},  the morphism in $\C C$ underlying $\eta_{B}$ is the colimit in 
$t\in \mathbb{N}$ of an inductively constructed diagram of cofibrant objects in~$\C C$, $t>0$,
\begin{equation*}\tag{a}
\xymatrix{
\cdots\into B_{t-1}\ar@{ >->}[r]^-{\varphi_t^{B}}\ar@<2.5ex>[d]_-{\eta_{t-1}}^{\sim}&B_t\ar@<-3.2ex>[d]^-{\eta_t}_-{\sim}\into\cdots\\
\cdots\into C_{t-1}\ar@{ >->}[r]^-{\varphi^C_t}&C_t\into\cdots
}
\end{equation*}
such that $B_{0}=A$, $C_{0}=\phi^{*}\phi_{*}A$, $\eta_{0}=\eta_{A}$, the morphism $\eta_t$ is the push-out of the horizontal lines of the following diagram
$$\xymatrix@C=130pt{
B_{t-1}\ar[d]_{\eta_{t-1}}^{\sim}&\bullet\ar[d]_{\text{induced by $\phi$ and $\eta_{0}$}}^\sim
\ar[l]_{\text{\eqref{killo2} for }\mathcal{O}}\ar@{ >->}[r]^{\text{\eqref{killo} for }\mathcal{O}}
&\bullet\ar[d]_{\text{induced by $\phi$ and $\eta_{0}$}}^\sim\\
C_{t-1}&\bullet
\ar[l]^{\text{\eqref{killo2} for }\mathcal{P}}\ar@{ >->}[r]_{\text{\eqref{killo} for }\mathcal{P}}
&\bullet
}$$
and $\varphi_{t}^{B}$ and $\varphi_{t}^{C}$ are the natural morphisms to the push-out. 

The objects $\mathcal{O}(n)$ and $\mathcal{P}(n)$ are cofibrant in $\C V$ and $z$ is a left Quillen functor, hence $z(\mathcal{O}(n))$ and $z(\mathcal{P}(n))$ are cofibrant in $\C C$, $n\geq 0$. Moreover, $f$ is a cofibration in $\C C$ and $A$ and $\phi^{*}\phi_{*}A$ are cofibrant in $\C C$. Therefore Lemma \ref{wedot} shows that the square on the right has weak equivalences in the columns and cofibrations in the rows. In particular,  $\varphi_{t}^{B}$ and $\varphi_{t}^{C}$ are cofibrations in $\C C$ and, by the gluing property in left proper model  categories \cite[Proposition 13.5.4]{hirschhorn}, $\eta_{t}$ is a weak equivalence in $\C C$. 

To conclude, $\eta_{B}=\colim_{t\geq 0}\eta_{t}$ is a weak equivalence in $\C C$ since  (a) is a weak equivalence between  cofibrant objects in the Reedy model category of directed diagrams in $\C C$ indexed by $\mathbb{N}$ \cite[Theorem 5.1.3]{hmc} and Ken Brown's lemma \cite[Lemma 1.1.12]{hmc} applies, because $\colim_{t\geq 0}$ is a left Quillen functor \cite[Corollary 5.1.6]{hmc}.
\end{proof}

Finally,  we are ready to prove Theorem \ref{elt3}.

\begin{proof}[Proof of Theorem \ref{elt3}]
We will use the criterion in \cite[Corollary 1.3.16 (c)]{hmc} to detect Quillen equivalences. The functor $\phi^*$ preserves and reflects weak equivalences, since it is the identity on the underlying object in $\C C$. Therefore, it is enough to check that the unit of the adjunction $\eta_A\colon A\r \phi^*\phi_{*}A$ is a weak equivalence for any cofibrant $\mathcal{O}$-algebra $A$. 

Weak equivalences are closed under retracts and cofibrant $\mathcal{O}$-algebras are retracts of $\mathcal{F}_{\mathcal{O}}(I)$-cell complexes, so we can suppose that $A$ is an $\mathcal{F}_{\mathcal{O}}(I)$-cell complex, $A=\colim_{i<\lambda}A_{i}$. We now proceed by induction on the ordinal $\lambda$.

For $\lambda =1$, $A$ is the initial $\mathcal{O}$-algebra, see Remark \ref{alal}. Then $\phi_{*}A$ is the initial $\mathcal{P}$-algebra, since $\phi_{*}$ is a left adjoint, and $\eta_{A}=z(\phi(0))\colon z(\mathcal{O}(0))\r z(\mathcal{P}(0))$. The morphism $\phi(0)$ is a weak equivalence between cofibrant objects in $\C V$ and $z$ is a left Quillen functor, therefore $z(\phi(0))$ is also a weak equivalence between cofibrant objects in $\C C$ by \cite[Lemma 1.1.12]{hmc}. 

If $\lambda=\alpha+1$ and the result is true for $\alpha$, then it is also true for $\lambda$ by the previous lemma. 

Suppose now that  $\lambda$ is a limit ordinal and that the result is true for all $i <\lambda$. The functor $\phi_{*}$ preserves colimits, since it is a left adjoint, and $\phi^{*}$ preserves filtered colimits, because it is the identity over $\C C$ and forgetful functors from algebras to~$\C C$ preserve filtered colimits. In particular $\eta_{A}=\colim_{i<\lambda}\eta_{i}$ is a colimit of weak equivalences by induction hypothesis. By Proposition \ref{1y2} (2), an $\mathcal{F}_{\mathcal{O}}(I)$-cell complex is a colimit of a continuous diagram of cofibrations between cofibrant objects in $ \C C$, and the same is true  for $\mathcal{F}_{\mathcal{P}}(I)$-cell complexes. This applies to $A$ and $\phi_{*}A$.  Such  diagrams  are cofibrant objects in Reedy model categories of directed diagrams in $\C C$  \cite[Theorem 5.1.3]{hmc}. Therefore, $\eta_{A}$ is the colimit of a weak equivalence between cofibrant objects in the Reedy model category of directed diagrams in $\C C$ indexed by $\lambda$. Now, Ken Brown's lemma \cite[Lemma 1.1.12]{hmc} shows that $\eta_{A}$ is a weak equivalence, since $\colim_{i< \lambda}$ is a left Quillen functor \cite[Corollary 5.1.6]{hmc}.
\end{proof}

\section{An application to enriched categories and $A_{\infty}$-categories}\label{appl}

In this section we lay the foundations to construct model categories of categorified algebraic structures. This is applied to enriched categories and enriched $A_\infty$-categories.


\begin{defn}
Given a set $S$, an \emph{$S$-graph $M$ with object set $S$} is a collection of objects in $\C{V}$ indexed by $S\times S$,  $M=\{M(x,y)\}_{x,y \in S}$. The category $\graphs{\C{V}}{S}$ of $\C{V}$-graphs with object set $S$, where morphisms are defined in the obvious way, is biclosed monoidal with tensor product,
$$(M\otimes_{S} N)(x,y)= \coprod_{z\in S}M(z,y)\otimes N(x,z).$$
The unit object $\unit_{S}$ is 
$$\unit_{S}(x,y)=\left\{
\begin{array}{ll}
\unit,&\text{the monoidal unit of }\C{V},\text{ if }x=y;\\
0,&\text{the initial object of }\C{V},\text{ if }x\neq y.
\end{array}
\right.$$

This monoidal category  is clearly non-symmetric, unless $S$ is a singleton. The right adjoint of $M\otimes -$ is the functor $\hom_{l}^{S}(M,-)$ defined as
\begin{align*}
\hom_{l}^{S}(M,P)(x,y)&=\prod_{z\in S}\hom(M(y,z),P(x,z)),
\end{align*}
and the right adjoint of $-\otimes N$ is the functor $\hom_{r}^{S}(N,-)$ defined as
\begin{align*}
\hom_{r}^{S}(N,P)(x,y)&=\prod_{z\in S}\hom(N(z,x),P(z,y)).
\end{align*}

We have a strong braided monoidal functor $z\colon \C{V}\r\graphs{\C{V}}{S}$ defined as
$$z(A)(x,y)=\left\{
\begin{array}{ll}
A,&\text{if }x=y;\\
0,&\text{if }x\neq y.
\end{array}
\right.$$
Moreover, 
\begin{align*}
(z(A)\otimes_{S} M)(x,y)&=A\otimes M(x,y),&
( M\otimes_{S} z(A))(x,y)&=M(x,y)\otimes A,
\end{align*}
and the natural isomorphism
$$\zeta(A,M)\colon z(A)\otimes_{S} M\cong M\otimes_{S} z(A),$$
is defined as the symmetry isomorphism of $\C{V}$ coordinatewise.

\end{defn}

\begin{rem}
If $\C V$ is a model category, the category $\graphs{\C{V}}{S}$ inherits from~$\C{V}$ a product model category structure, where fibrations, cofibrations and weak equivalences are defined coordinatewise.  If $\C V$ is cofibrantly generated (resp. combinatorial) then so is $\graphs{\C{V}}{S}$, compare Remark \ref{jn}. Moreover, since $S\times S$ is a set, an $S$-graph $M$ is presentable provided $M(x,y)$ is presentable for all $x,y\in S$. In particular, if $\C V$ has sets of generating cofibrations and generating trivial cofibrations with presentable source, then so does $\graphs{S}{\C V}$.  Furthermore, 
if $\C V$ is right proper then the product model category $\graphs{S}{\C V}$ is also right proper.

Notice that the composite functor  $\C{V}\st{z}\r Z(\graphs{\C{V}}{S})\r\graphs{\C{V}}{S}$ preserves fibrations, cofibrations and weak equivalences, and it has a right adjoint defined by
$$M\mapsto \prod_{x\in S}M(x,x).$$
This adjoint pair is therefore a Quillen adjunction.
\end{rem}

\begin{prop}
If $\C{V}$ satisfies the monoid axiom  then $\graphs{\C{V}}{S}$ also satisfies the monoid axiom.
\end{prop}

\begin{proof}
It is enought to notice, using the symmetry of $\C{V}$ and the push-out product axiom in $\C{V}$, that any morphism $f_{1}\odot\cdots\odot f_{n}$ in the class of morphisms $K'$ of $\graphs{\C{V}}{S}$ in Definition \ref{nsmax} is componentwise a morphism in the class $K$ of $\C{V}$ in Definition \ref{max}.
\end{proof}

Categories enriched on $\C V$ with set of objects $S$ are the same as monoids in $\graphs{\C{V}}{S}$. These monoids are the same as algebras over the non-symmetric operad $\mathtt{Ass}^{\C V}$ in $\C{V}$ defined by $\mathtt{Ass}^{\C V}(n)=\unit$, $n\geq 0$. All compositions  in $\mathtt{Ass}^{\C V}$ are unit isomorphisms $\unit\otimes\unit\cong\unit$ and the unit opf the operad $u\colon\unit\r\mathtt{Ass}^{\C V}(1)$ is the identity.  This operad is generated by the `elements' in degree $0$ and $2$; the degree $2$ `element' represents the composition law, and the degree $0$ `element' represents the identities.  In order to simplify notation, we denote $$\cats{\C V}{S}=\algebra{\mathtt{Ass}^{\C V}}{\graphs{\C{V}}{S}}.$$

An $A_\infty$-category enriched on $\C V$ with set of objects $S$ is an algebra over a cofibrant replacement $\mathtt{Ass}^{\C V}_\infty$ of $\mathtt{Ass}^{\C V}$, which is a trivial fibration  $\phi\colon\mathtt{Ass}^{\C V}_\infty\st{\sim}
\onto\mathtt{Ass}^{\C V}$ in $\operad{\C V}$ with cofibrant source. We simply  denote $$\ainfcats{\C V}{S}=\algebra{\mathtt{Ass}_\infty^{\C V}}{\graphs{\C{V}}{S}}.$$

Combining the previous proposition with Theorem \ref{elt2} we obtain the following corollary, which improves \cite[Theorem 3.3]{lvc}.

\begin{cor}
Let $\C V$ be a cofibrantly generated closed symmetric monoidal category satisfying the monoid axiom. Suppose that $\C V$ has sets of generating cofibrations and generating trivial cofibrations with presentable source. Then $\cats{\C V}{S}$ is a model category where an enriched functor $F\colon C\r D$ is a weak equivalence (resp. fibration) if $F(x,y)\colon C(x,y)\r D(x,y)$ is  a weak equivalence (resp. fibration) in $\C V$ for all $x,y\in S$, and similarly for   $\ainfcats{\C V}{S}$. Moreover, these model categories are right proper (resp. combinatorial) provided $\C V$ is.
\end{cor}

The following corollary also uses Theorem \ref{elt3}.

\begin{cor}
In the conditions of the previous corollary, assume in addition that $\C V$ is left proper 
 and the monoidal unit $\unit_{\C V}$ is cofibrant. Then the pull-back functor $\phi^*$ from enriched categories to enriched $A_\infty$-categories and the strictification functor $\phi_{*}$ in the other direction form a Quillen equivalence,
$$\xymatrix{\ainfcats{\C V}{S}\ar@<.5ex>[r]^-{\phi_{*}}& \cats{\C V}{S}.\ar@<.5ex>[l]^-{\phi^*}}$$
In particular, the derived adjoint pair is an equivalence between the homotopy categories of enriched categories and enriched $A_\infty$-categories, 
$$\xymatrix{\ho\ainfcats{\C V}{S}\ar@<.5ex>[r]^-{\mathbb{L}\phi_{*}}& \ho\cats{\C V}{S}.\ar@<.5ex>[l]^-{\phi^*}}$$
\end{cor}


\begin{thebibliography}{DK80b}

\bibitem[AR94]{adamekrosicky}
J.~Ad{\'a}mek and J.~Rosick{\'y}, \emph{Locally presentable and accessible
  categories}, London Mathematical Society Lecture Note Series, vol. 189,
  Cambridge University Press, Cambridge, 1994.

\bibitem[BJT97]{cmmc}
H.-J. Baues, M.~Jibladze, and A.~Tonks, \emph{Cohomology of monoids in monoidal
  categories}, Proceedings of the Renaissance Conferences (Providence, RI)
  (J.-L. Loday, J.~D. Stasheff, and A.~A. Voronov, eds.), Contemp. Math., vol.
  202, Amer. Math. Soc., 1997, pp.~137--165.

\bibitem[BM03]{ahto}
C.~Berger and I.~Moerdijk, \emph{Axiomatic homotopy theory for operads},
  Comment. Math. Helv. \textbf{78} (2003), no.~4, 805--831.

\bibitem[Bor94]{borceux2}
F.~Borceux, \emph{Handbook of categorical algebra 2}, Encyclopedia of Math. and
  its Applications, no.~51, Cambridge University Press, 1994.

\bibitem[DK80a]{csl}
W.~G. Dwyer and D.~M. Kan, \emph{Calculating simplicial localizations}, J. Pure
  Appl. Algebra \textbf{18} (1980), no.~1, 17--35.

\bibitem[DK80b]{fcha}
\bysame, \emph{Function complexes in homotopical algebra}, Topology \textbf{19}
  (1980), no.~4, 427--440.

\bibitem[DK80c]{slc}
\bysame, \emph{Simplicial localizations of categories}, J. Pure Appl. Algebra
  \textbf{17} (1980), no.~3, 267--284.

\bibitem[Dun01]{lvc}
B.~I. Dundas, \emph{Localization of {$V$}-categories}, Theory Appl. Categ.
  \textbf{8} (2001), no.~10, 284--312.

\bibitem[GK94]{kdo}
V.~Ginzburg and M.~Kapranov, \emph{Koszul duality for operads}, Duke Math. J.
  \textbf{76} (1994), no.~1, 203--272.

\bibitem[Har10]{htmommc}
J.~E. Harper, \emph{Homotopy theory of modules over operads and non-{$\Sigma$}
  operads in monoidal model categories}, J. Pure Appl. Algebra \textbf{214}
  (2010), no.~8, 1407--1434.

\bibitem[Hin97]{haha}
V.~Hinich, \emph{Homological algebra of homotopy algebras}, Comm. Algebra
  \textbf{25} (1997), no.~10, 3291--3323.

\bibitem[Hin03]{ehaha}
\bysame, \emph{Erratum to ``{H}omological algebra of homotopy algebras''},
  \texttt{arXiv:math/0309453v3 [math.QA]}, 2003.

\bibitem[Hir03]{hirschhorn}
P.~S. Hirschhorn, \emph{Model categories and their localizations}, Mathematical
  Surveys and Monographs, vol.~99, American Mathematical Society, Providence,
  RI, 2003.

\bibitem[Hov99]{hmc}
M.~Hovey, \emph{Model categories}, Mathematical Surveys and Monographs,
  vol.~63, American Mathematical Society, Providence, RI, 1999.

\bibitem[JS91]{tybotc}
A.~Joyal and R.~Street, \emph{Tortile {Y}ang-{B}axter operators in tensor
  categories}, J. Pure Appl. Algebra \textbf{71} (1991), no.~1, 43--51.

\bibitem[Kel05]{bcect}
G.~M. Kelly, \emph{Basic concepts of enriched category theory}, Repr. Theory
  Appl. Categ. (2005), no.~10, vi+137, Reprint of the 1982 original [Cambridge
  Univ. Press].

\bibitem[Rap09]{cgmc}
G.~Raptis, \emph{On the cofibrant generation of model categories}, J. Homotopy
  Relat. Struct. \textbf{4} (2009), no.~1, 245--253.

\bibitem[SS00]{ammmc}
S.~Schwede and B.~E. Shipley, \emph{Algebras and modules in monoidal model
  categories}, Proc. London Math. Soc. (3) \textbf{80} (2000), no.~2, 491--511.

\end{thebibliography}

\providecommand{\bysame}{\leavevmode\hbox to3em{\hrulefill}\thinspace}
\providecommand{\MR}{\relax\ifhmode\unskip\space\fi MR }
\providecommand{\MRhref}[2]{%
  \href{http://www.ams.org/mathscinet-getitem?mr=#1}{#2}
}
\providecommand{\href}[2]{#2}

\end{document}